\documentclass[11pt,letterpaper]{amsart}
\usepackage{mathrsfs}
\usepackage{geometry}

\usepackage{hyperref}
\usepackage{graphicx}
\usepackage{amssymb}
\usepackage{epstopdf}
\usepackage{amsmath,amscd}
\usepackage{amsthm}
\usepackage{float}
\usepackage{enumitem}
\usepackage{setspace}
\usepackage{cleveref}
\usepackage{breakurl}
\usepackage{color}

\usepackage{setspace}

\usepackage[numbers,sort&compress]{natbib}

\usepackage[caption = false]{subfig}
\theoremstyle{plain} \textwidth=430pt \textheight=620pt

\newtheorem{theorem}{Theorem}[section]
\newtheorem{definition}[theorem]{Definition}
\newtheorem{lemma}[theorem]{Lemma}
\newtheorem{example}[theorem]{Example}

\newtheorem{proposition}[theorem]{Proposition}
\newtheorem{claim}[theorem]{Claim}
\newtheorem{corollary}[theorem]{Corollary}
\newtheorem{conjecture}[theorem]{Conjecture}

\newcommand\ZZ{{\mathbb Z}}
\newcommand\RR{{\mathbb R}}
\newcommand\sA{{\mathcal A}}

\newcommand\LL{{\mathbb L}}
\newcommand\sD{{\mathcal D}}

\newcommand\HH{{\mathbb H}}
\newcommand\TT{{\mathbb T}}

\newcommand\df{\textbf}

\title{Constrained percolation, Ising model and XOR Ising model on planar lattices}

\author{Zhongyang Li}
\address{Department of Mathematics,
University of Connecticut,
Storrs, Connecticut 06269-3009, USA}
\email{zhongyang.li@uconn.edu}
\urladdr{\url{https://mathzhongyangli.wordpress.com}}

\begin{document}
\maketitle

\begin{abstract}We study site percolation models on planar lattices including the $[m,4,n,4]$ lattice and the square tilings on the Euclidean plane ($\RR^2$) or the hyperbolic plane ($\mathbb{H}^2$), satisfying certain local constraints on degree-4 faces. These models are closely related to Ising models and XOR Ising models (product of two i.i.d Ising models) on regular tilings of $\RR^2$ or $\mathbb{H}^2$. In particular, we obtain a description of the numbers of infinite ``$+$'' and ``$-$'' clusters of the ferromagnetic Ising model  on a vertex-transitive triangular tiling of $\HH^2$ for different boundary conditions and coupling constants. Our results show the possibility that such an Ising configuration has infinitely many infinite ``$+$'' and ``$-$'' clusters, while its random cluster representation has no infinite open clusters. Percolation properties of corresponding XOR Ising models are also discussed. 
\end{abstract}

\section{Introduction}

A constrained percolation model is a probability measure on subgraphs of a lattice satisfying certain local constraints. Each subgraph is called a configuration. These models are abstract mathematical models for ubiquitous phenomena in nature, and have been interesting topics in mathematical and scientific research for long.
Examples of constrained percolation models include the dimer model (see \cite{RK09}), the 1-2 model (see \cite{GrLReview}), the six-vertex model (or 6V model, see \cite{Bax08,BCG16,KMSW16}), and general vertex models (see \cite{Val,SB,ZLlv}). The study of these models may give deep insights to understand many natural phenomena, such as structure of matter, phase transition, limit shape, and critical behavior.

 We are interested in the classical percolation problem in a constrained model: under which probability measure does there exist an infinite connected set (infinite cluster) in which every vertex is present in the random configuration, or equivalently, included in the randomly-chosen subgraph? Such a question has been studied extensively in the unconstrained case - in particular the i.i.d Bernoulli percolation - see, for instance, \cite{ha60,he80,HS94,blps99,fh15,GrPc}. The major difference between the constrained percolation and the unconstrained percolation lies in the fact that imposing local constraints usually makes stochastic monotonicity, which is a crucial property when studying the unconstrained model, invalid. Therefore new techniques need to be developed to study constrained percolation models.
 
Some constrained percolation models, including the 1-2 model, the periodic plane dimer model, certain 6V models, are exactly solvable; see \cite{KOS06,ZLejp,ZL12, ZLct, GrL15,BCG16}. The integrability properties of these models make it possible to compute the correlations. When the parameters associated to the probability measure vary, different behaviors of the local correlations imply a phase transition from a microscopic point of view. If we consider phase transitions from a macroscopic, or geometric point of view, different approaches may be applied to study the existence of infinite clusters for a large class of constrained percolation models. 

In \cite{HL16}, we studied a constrained percolation model on the $\ZZ^2$ lattice, and showed that if the underlying probability measure satisfies mild assumptions like symmetry, ergodicity and translation-invariance, then with probability 0 the number of infinite clusters is nonzero and finite. The technique makes use of the planarity and amenability of the 2D square grid $\ZZ^2$.  As an application, we obtained percolation properties for the XOR Ising model (a random spin configuration on a graph in which each spin is the product of two spins from two i.i.d Ising models, see \cite{DBW11}) on $\ZZ^2$, with the help of the combinatorial correspondence between the XOR Ising model and the dimer model proved in \cite{Dub,bd14}.  In this paper, we further develop the technique to study constrained percolation models on a number of planar lattices, which may be amenable or non-amenable, including the $[m,4,n,4]$ lattice and the square tilings of the hyperbolic plane; see \cite{CFKP97} for an introduction to hyperbolic geometry.

The XOR Ising model was first introduced in \cite{DBW11} with interesting conformal invariance properties at criticality. For positive integers $m,n\geq 3$, the $[m,4,n,4]$ lattice is a vertex-transitive planar graph in which each vertex is incident to 4 faces with degrees $m,4,n,4$ in cyclic order.
The constrained percolation model on the  $[m,4,n,4]$ lattice is of special interest because there is a measure-preserving correspondence between its configurations and the XOR Ising configurations on the $m$-regular lattice or the $n$-regular lattice. The Euclidean-plane version of such a correspondence was introduced in \cite{bd14}. When $\frac{1}{m}+\frac{1}{n}<\frac{1}{2}$, the $[m,4,n,4]$ lattice is no longer amenable but can be embedded into the hyperbolic plane. Although phase transitions and conformal invariance for statistical mechanical models in the Euclidean plane have been studied extensively, statistical mechanical models, including the Ising model and the related random cluster model,  have been fascinating problems for mathematicians and physicists for a long time, however, a lot of things remain unknown.  For example, it is well-known that for statistical mechanical models in the hyperbolic plane, there is an ``intermediate'' phase between the non-percolation phase and unique-percolation phase, which usually does not exist for statistical mechanical models in the Euclidean plane; a lot of descriptions of the ``intermediate'' phase seem to be ``qualitative'' while not ``quantitative'' - for which values of the parameters does the model have such an ``intermediate'' phase?  Indeed, the general results we obtain in this paper can be used to prove further results concerning percolation properties of the XOR Ising model on the hexagonal and the triangular lattices, as well as on regular tilings of the hyperbolic plane. 

The specific geometric properties of non-amenable graphs make it an interesting problem to study percolation models on such graphs; and a set of techniques have been developed in the past few decades; see \cite{BS96,blps,blps99,HP,LS99,Sch99,HPS,PSN,Wu00,RS01,HJL02,NP12,LP} for an incomplete list.
In this paper, we also study the general automorphism-invariant percolation models on transitive planar graphs. 

One of the most classical percolation models is the \textbf{i.i.d Bernoulli site percolation} on a graph, in which the vertices are open (resp.\ closed) with probability $p$ (resp.\ $1-p$) independently, where $p\in[0,1]$.  The \textbf{critical probability $p_c$} is the supremum of $p$'s such that almost surely there are no infinite open clusters. A  graph $G=(V,E)$ is a \textbf{vertex-transitive graph} if there exists a subgroup $\Gamma\subseteq \mathbf{Aut}(G)$ of the automorphism group $G$ such that for any two vertices $v,w\in V$, there exist $\gamma\in \Gamma$ satisfying $\gamma v=w$.  The number of \textbf{ends} of a connected graph
is the supremum over its finite subgraphs of the number of infinite components that remain
after removing the subgraph.
 
 Our results may be related to the following two conjectures. More precisely, we prove the following conjectures for some special vertex-transitive planar graphs.  

\begin{conjecture}(Conjecture 7 of \cite{BS96})\label{cj1} Suppose that $G$ is a planar, connected graph, and the minimal vertex degree in $G$ is at least 7. In an i.i.d Bernoulli site percolation on $G$,  at every $p$ in the range $(p_c,1-p_c)$, there are infinitely many infinite open clusters in the i.i.d Bernoulli site percolation on $G$. Moreover, we conjecture that $p_c<\frac{1}{2}$, and the above interval is nonempty.
\end{conjecture}

In \Cref{exl2}, we explain why \Cref{cj1} is true for the i.i.d Bernoulli site percolation on vertex-transitive triangular tilings of the hyperbolic plane where each vertex has degree $n\geq 7$.

\begin{conjecture}(Conjecture 8 of \cite{BS96})\label{cj2} Let $G$ be a planar, connected graph. Let $p=\frac{1}{2}$ be the probability that a vertex is open and assume that a.s. percolation occurs in the site percolation on $G$. Then almost surely there are infinitely many infinite clusters.
\end{conjecture}

Our \Cref{p118} implies that \Cref{cj2} is true for automorphism-invariant site percolation (not necessarily independent, or insertion tolerant) on vertex-transitive triangular tilings of the hyperbolic plane where each vertex has degree $n\geq 7$ if the underlying measure is ergodic and invariant under switching state-1 vertices and state-0 vertices.

We then apply our results concerning the general automorphism-invariant percolation models on transitive planar graphs to study the infinite ``$+$''-clusters and ``$-$''-clusters for the Ising model on vertex-transitive triangular tilings of the hyperbolic plane where each vertex has degree $n\geq 7$, and describe the behaviors of such clusters with respect to varying coupling constants under the free boundary condition and the wired boundary condition. A surprising result we obtain is that it is possible that the random cluster representation of the Ising model has no infinite open clusters, while the Ising model has infinitely many infinite ``$+$''-clusters and infinitely many infinite ``$-$''-clusters - in contrast with the Ising percolation and its random cluster representation on the 2d square grid $\ZZ^2$ (see \cite{crpr76,Hig93,GrGrc}) where the Ising model has an infinite ``$+$'' or ``$-$''-cluster if and only if its random cluster representation has an infinite open cluster. 

The main tools to prove these results are the planar duality of graphs, ergodicity and symmetry of probability measures, as well as properties of amenablity and non-amenablity. One characteristic of the constrained percolation obtained from a natural correspondence with the XOR Ising model, which is not shared with the unconstrained percolation, is that given such a constraint, there are two sets of ``contours'' separating clusters of vertices of different states. These two sets of contours lie on two planar graphs dual to each other, and the present edges in these two different sets of contours never cross. As a result, there are four types of infinite components in our constrained percolation model: infinite ``0''-cluster, infinite ``1''-cluster, infinite planar contour and infinite dual contour. The geometric configurations of these infinite components, together with the ergodicity and symmetry of the probability measure, lead to interesting properties that are particular and unique to the constrained percolation model.

The organization of the paper is as follows. 

In \Cref{xorh}, we introduce the $[m,4,n,4]$ lattice and state the result concerning constrained percolation models on the $[m,4,n,4]$ lattice. In \Cref{Is}, we state the main results concerning infinite clusters in the Ising model on regular triangular tilings of the hyperbolic plane, and, in particular, provide a description of the numbers of infinite ``$+$'' and ``$-$'' clusters of the ferromagnetic Ising model with the free boundary condition, the ``$+$" boundary condition or the ``$-$" boundary condition on such a lattice for different values of coupling constants. In \Cref{xIs}, we state the main results concerning infinite clusters in the XOR Ising model on regular triangular tilings of the hyperbolic plane and its dual graph. In Section \ref{xori}, we state the result proved in this paper concerning the percolation properties of the XOR Ising model on the hexagonal lattice and the triangular lattice. In \Cref{sthp}, we introduce the square tilings of the hyperbolic plane, state and prove the main result concerning constrained percolation models on such a lattice.

The remaining sections are devoted to prove the theorems stated in preceding sections.
  In \Cref{p23}, we prove \Cref{m23}.  In \Cref{p212}, we prove \Cref{m21}. In \Cref{p211}, we prove \Cref{m22}.  In \Cref{peh}, we discuss the applications of the techniques developed in the proof of \Cref{m23}
 to prove results concerning unconstrained site percolation on vertex-transitive, triangular tilings of the hyperbolic plane in preparation of proving \Cref{ipl,coii,xorc}. In \Cref{pipl}, we prove \Cref{ipl}. In \Cref{pxorc}, we prove \Cref{coii} and \Cref{xorc}.
 In \Cref{p412}, we prove \Cref{chi,lth}.  In \Cref{ctcl}, we prove combinatorial results concerning contours and clusters in preparation to prove the main theorems.

\section{Constrained percolation on the $[m,4,n,4]$ lattice }\label{xorh}

In this section, we state the main result proved in this paper for the constrained percolation models on the $[m,4,n,4]$ lattice. We shall start with a formal definition of the $[m,4,n,4]$ lattice.

Let $m,n$ be positive integers satisfying
\begin{eqnarray}
&& m\geq 3,\qquad n\geq 3\label{cmn1}\\
&&\frac{1}{m}+\frac{1}{n}\leq \frac{1}{2}.\label{cmn2}
\end{eqnarray}

The $[m,4,n,4]$ lattice is a vertex-transitive graph which can be embedded into the Euclidean plane or the hyperbolic plane such that each vertex is incident to 4 faces with degrees $m,4,n,4$ in cyclic order. When $\frac{1}{m}+\frac{1}{n}=\frac{1}{2}$, the graph is amenable and can be embedded into the Euclidean plane. When $\frac{1}{m}+\frac{1}{n}<\frac{1}{2}$, the graph is non-amenable and can be embedded into the hyperbolic plane (\cite{DR}). Note that when $m=n=4$, the graph is the square grid embedded into the 2D Euclidean plane.  See Figure \ref{3464} for an illustration of the [3,4,6,4] lattice, \Cref{3474} for the [3,4,7,4] lattice, and \Cref{4646} for the [6,4,6,4] lattice.

\begin{figure}
\includegraphics[width=.5\textwidth]{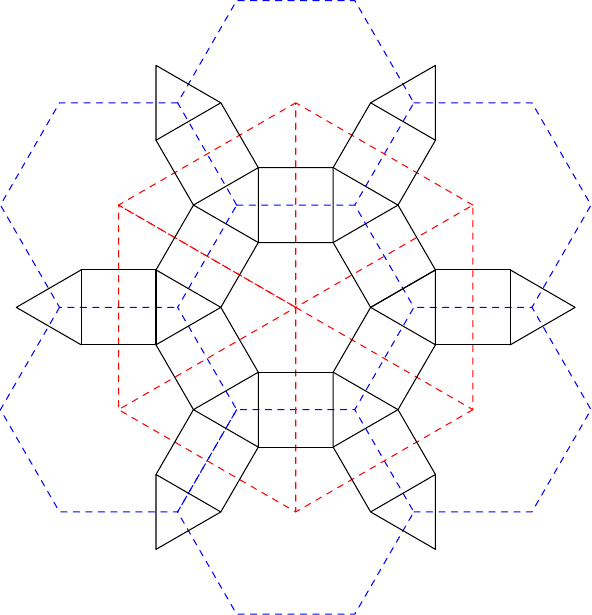}
\caption{The [3,4,6,4] lattice, the auxiliary hexagonal lattice and triangular lattice. Black lines represent the [3,4,6,4] lattice; dashed red lines represent the triangular lattice; dashed blue lines represent the hexagonal lattice.}
\label{3464}
\end{figure}

\begin{figure}
\includegraphics[width=.6\textwidth]{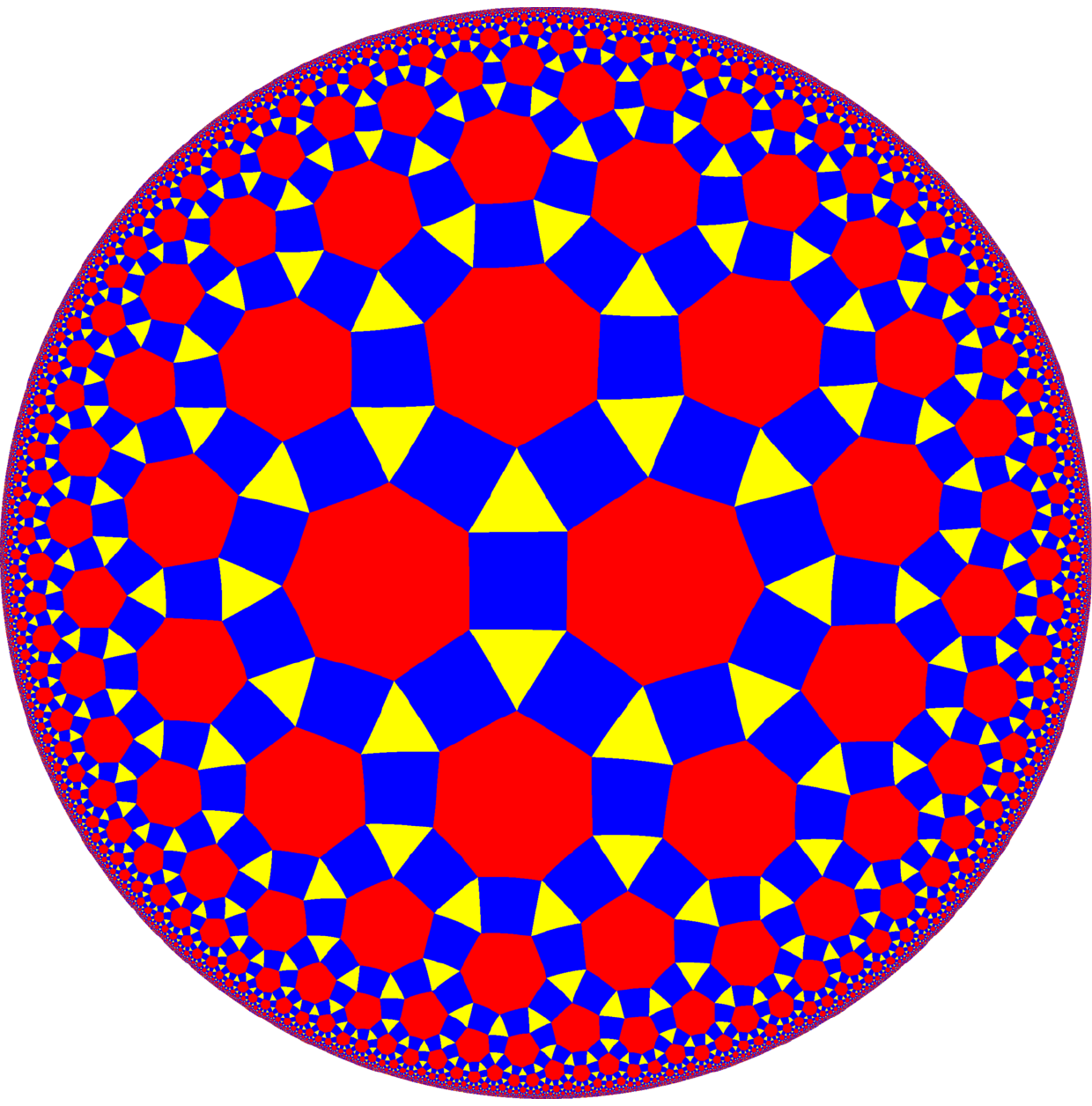}
\caption{The [3,4,7,4] lattice (picture from the wikipedia)}
\label{3474}
\end{figure}

\begin{figure}
\includegraphics[width=.6\textwidth]{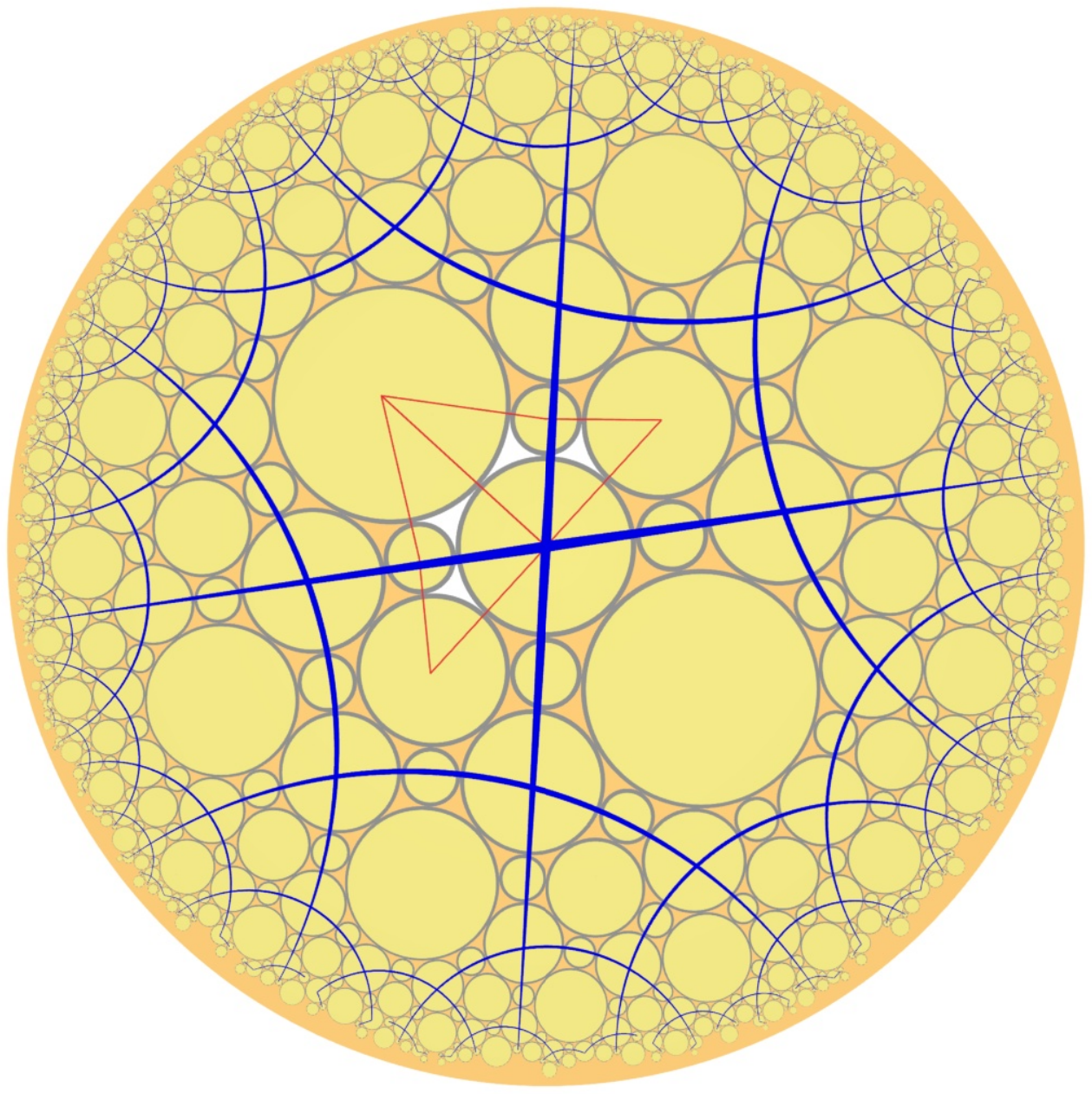}
\caption{The [6,4,6,4] lattice represented by blue lines (picture from \url{http://epinet.anu.edu.au/})}
\label{4646}
\end{figure}

Let $G=(V,E)$ be an $[m,4,n,4]$ lattice. We color all the faces of degree $m$ or $n$ with white and all the other faces with black, such that any two faces sharing an edge have different colors. We consider the site percolation on $V$ satisfying the following constraint (see Figure \ref{lcc}):
\begin{itemize}
\item around each black face, there are six allowed configurations $(0000)$, $(1111)$, $(0011)$, $(1100)$, $(0110)$, $(1001)$, where the digits from the left to the right correspond to vertices in clockwise order around the black face, starting from the lower left corner. See Figure \ref{lcc}.
\end{itemize}

\begin{figure}
\subfloat[0000]{\includegraphics[width=.12\textwidth]{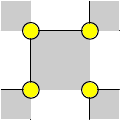}}\qquad\qquad\qquad\qquad
\subfloat[0011]{\includegraphics[width = .12\textwidth]{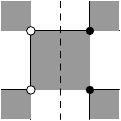}}\qquad\qquad\qquad\qquad
\subfloat[0110]{\includegraphics[width = .12\textwidth]{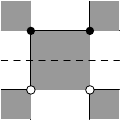}}\\
\subfloat[1111]{\includegraphics[width = .12\textwidth]{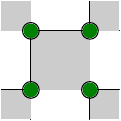}} \qquad\qquad\qquad\qquad
\subfloat[1100]{\includegraphics[width = .12\textwidth]{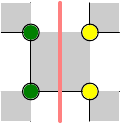}}\qquad\qquad\qquad\qquad
\subfloat[1001]{\includegraphics[width = .12\textwidth]{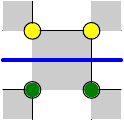}}
\caption{Local configurations of the constrained percolation around a black square. Red and blue lines mark contours separating 0's and 1's (in $\LL_1$ and $\LL_2$ respectively). Yellow (resp.\ green) disks represent 0's (resp.\ 1's).}
\label{lcc}
\end{figure}

Let $\Omega\subset\{0,1\}^V$ be the probability space consisting of all the site configurations on $G$ satisfying the constraint above.
To the $[m,4,n,4]$ lattice $G$, we associate two auxiliary lattices $\LL_1=(V(\LL_1),E(\LL_1))$ and $\LL_2=(V(\LL_2),E(\LL_2))$ as follows. Each vertex of $\LL_1$ (resp.\ $\LL_2$) is located at the center of each degree-$m$ face (resp. degree-$n$ face) of $G$. Two vertices of $\LL_1$ (resp.\ $\LL_2$) are joined by an edge of $\LL_1$ (resp.\ $\LL_2$) if and only if the two corresponding $m$-faces (resp.\ $n$-faces)  of  $G$ are adjacent to the same square face of $G$ through a pair of opposite edges (edges of a square face that do not share a vertex), respectively.

We say an edge $e\in E(\LL_1)\cup E(\LL_2)$ \textbf{crosses} a square face of the $[m,4,n,4]$ lattice if the edge $e$ crosses a pair of opposite edges of the square face.
 Note that 
\begin{enumerate}[label=\roman*]
\item $\LL_1$ (resp.\ $\LL_2$) is a planar lattice in which each face has degree $n$ (resp.\ $m$) and each vertex has degree $m$ (resp.\ $n$).
\item $\LL_1$ and $\LL_2$ are planar dual to each other.
\item Each edge in $E(\LL_1)\cup E(\LL_2)$ crosses a unique square face of the $[m,4,n,4]$ lattice. When $m\neq 4$ and $n\neq 4$, each square face of the $[m,4,n,4]$ lattice is crossed by a unique edge $e_1\in E(\LL_1)$ and a unique edge $e_2\in E(\LL_2)$; and moreover, $e_1$ and $e_2$ are dual to each other.
 \end{enumerate}
 
 When $m=n$, both $\LL_1$ and $\LL_2$ are lattices in which each face has degree $n$ and each vertex has degree $n$.
When $m=3$ and $n=6$, $\LL_1$ is the hexagonal lattice and $\LL_2$ is the triangular lattice; see \Cref{3464}. When $m=3$ and $n=7$, $\LL_2$ is a vertex-transitive triangular tiling of the hyperbolic plane, in which each vertex has degree 3; see the left graph of \Cref{3h7}; while $\LL_1$ is the $[7,7,7]$ lattice on the hyperbolic plane, see the right graph of \Cref{3h7}.

\begin{figure}
\includegraphics[trim=100 200 120 200,clip,width=.45\textwidth]{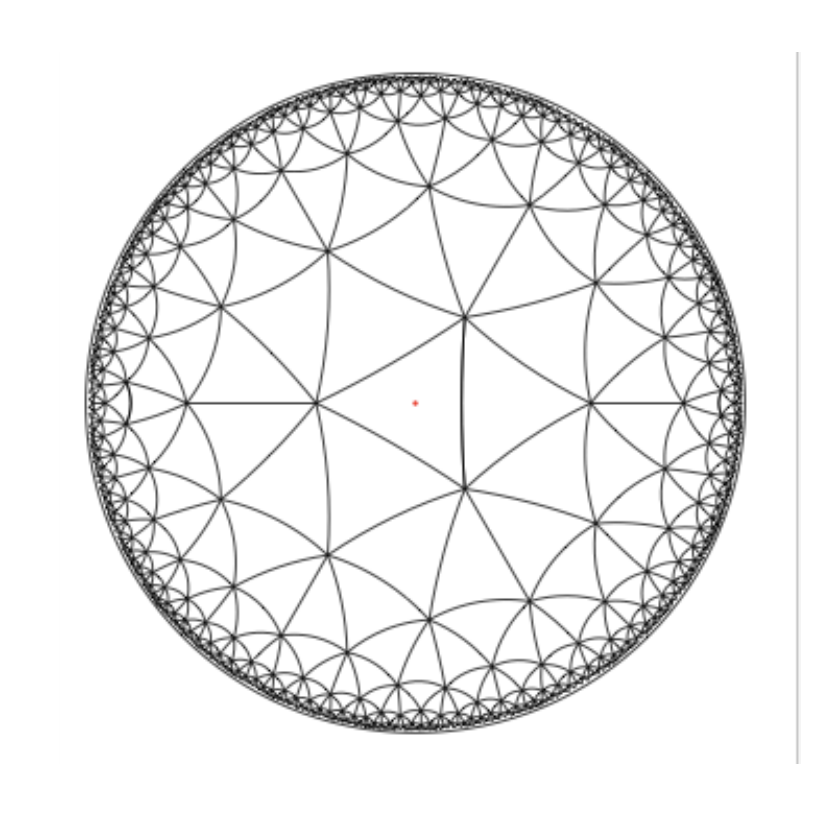}\qquad\includegraphics[trim=120 200 100 200, clip, width=.45\textwidth]{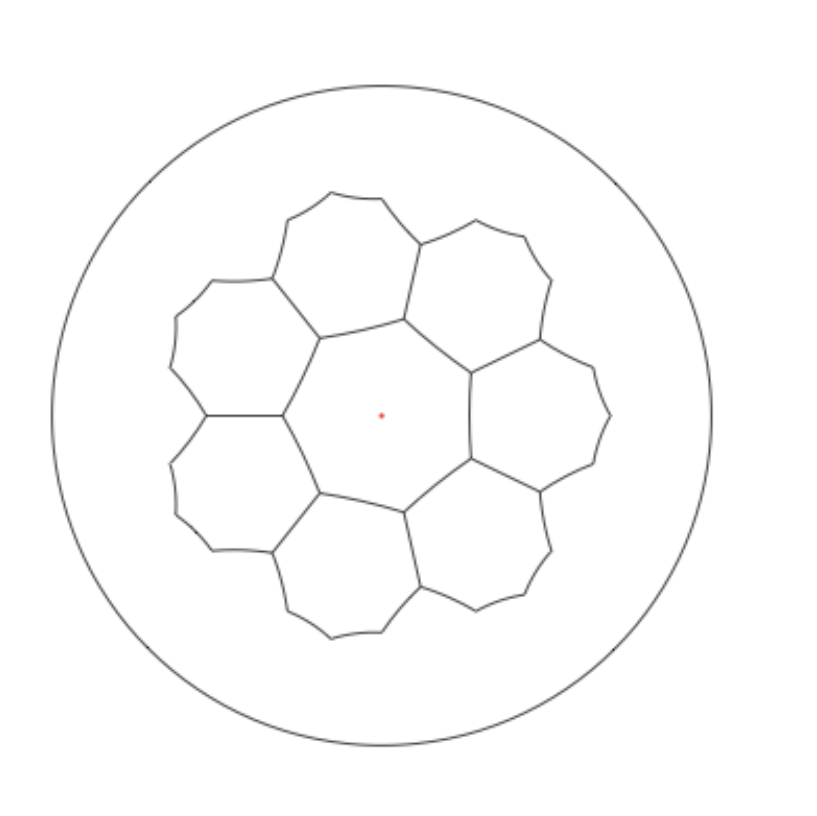}
\caption{The [3,3,3,3,3,3,3] lattice on the left and the [7,7,7] lattice on the right }
\label{3h7}
\end{figure}

Let $\Phi\subset \{0,1\}^{E(\LL_1)\cup E(\LL_2)}$ be the set of contour configurations satisfying the condition that each vertex of $V(\LL_1)$ and $V(\LL_2)$ is incident to an even number of present edges, and present edges in $E(\LL_1)$ and $E(\LL_2)$ never cross. Any constrained percolation configuration $\omega\in \Omega$ is mapped to a contour configuration $\phi(\omega)\in \Phi$, where an edge $e$ in $E(\LL_1)$ or $E(\LL_2)$ is present (i.e., have state 1) if and only if the following condition holds
\begin{itemize}
\item Let $S$ be the square face of $G$ crossed by $e$. Let $v_1,v_2,v_3,v_4$ be the four vertices of $S$, such that $v_1$ and $v_2$ are on one side of $e$ and $v_3$, $v_4$ are on the other side of $e$. Then $v_1$ and $v_2$ have the same state, $v_3$ and $v_4$ have the same state, and $v_1$ and $v_3$ have different states.
\end{itemize} 
See \Cref{3464contour} for a contour configuration obtained from a constrained percolation configuration on the $[3,4,6,4]$ lattice. Note that the mapping $\phi:\Omega\rightarrow \Phi$ is 2-to-1 since $\phi(\omega)=\phi(1-\omega)$. 

A \df{contour}  is a connected component of present edges in a contour configuration in $\Phi$. A contour may be finite or infinite depending on the number of edges in the contour. Since present edges of a contour configuration in $E(\LL_1)$ and in $E(\LL_2)$ never cross, either all the edges in a contour are edges of $\LL_1$, or all the edges in a contour are edges in $\LL_2$. We call a contour \df{primal contour} (resp.\ \df{dual contour}) if all the edges in the contour are edges of $\LL_1$ (resp.\ $\LL_2$).

\begin{figure}
\includegraphics[width=.5\textwidth]{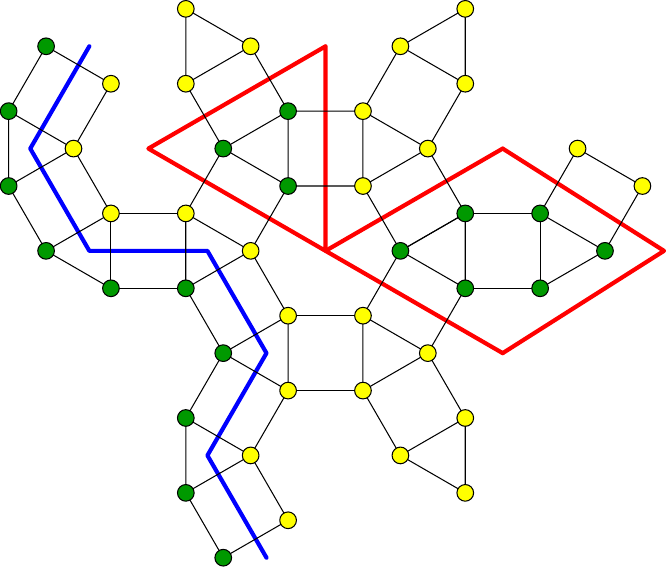}
\caption{A constrained percolation configuration on the [3,4,6,4] lattice. Red lines represent contours on the triangular lattice. Blue lines represent contours on the hexagonal lattice.}
\label{3464contour}
\end{figure}

 Let $\Gamma$ be the automorphism group $\mathrm{Aut}(G)$ of the graph $G$. For $i\in\{1,2\}$, let $\Gamma_i$ be the automorphism group $\mathrm{Aut}(\LL_i)$ of the graph $\LL_i$. Let $\mu$ be a probability measure on $\Omega$. The measure $\mu$ naturally induces a measure on $\Phi$ under the 2-to-1 mapping $\phi$; with a little abuse of notation, we still use $\mu$ to denote this induced measure on contour configurations $\Phi$. We may assume that $\mu$ satisfies the following conditions:
\begin{enumerate}[label=({A}{\arabic*})]
\item $\mu$ is $\Gamma$-invariant;
\item $\mu$ is $\Gamma_i$-ergodic for $i=1,2$; i.e.\ any $\Gamma_i$-invariant event has $\mu$-probability 0 or 1;
\item $\mu$ is symmetric: let $\theta:\Omega\rightarrow\Omega$ be the map defined by $\theta(\omega)(v)=1-\omega(v)$, for each vertex $v\in V$, then $\mu$ is invariant under $\theta$, that is, for any event $A$, $\mu(A)=\mu(\theta(A))$.
\end{enumerate}

Let $\Phi_1$ (resp.\ $\Phi_2$) be the set of all contour configurations on $\LL_1$ (resp.\ $\LL_2$) satisfying the condition that each vertex of $\LL_1$ (resp.\ $\LL_2$) has an even number of incident present edges. For each contour configuration $\psi\in\Phi$, we have $\psi=\psi_1\cup\psi_2$, where $\psi_1\in \Phi_1$ and $\psi_2\in\Phi_2$; moreover, $\psi_1\cap\psi_2=\emptyset$. 

Let $\nu_1$ (resp.\ $\nu_2$) be the marginal distribution of $\mu$ on $\Phi_1$ (resp.\ $\Phi_2$). When $\frac{1}{m}+\frac{1}{n}=\frac{1}{2}$, the $[m,4,n,4]$ lattice is amenable. 
It is not hard to see that if $\frac{1}{m}+\frac{1}{n}=\frac{1}{2}$, then $(m,n)\in \{(4,4),(3,6),(6,3)\}$. When $m=n=4$, the $[m,4,n,4]$ lattice is the 2D square grid, on which the constrained percolation models was discussed in \cite{HL16}. 

Now we consider the case when $(m,n)=(3,6)$. As discussed before, in this case $\LL_1$ is the hexagonal lattice $\HH$, and $\LL_2$ is the triangular lattice $\TT$. We first define the finite energy condition of a random contour configuration on a planar graph.

\begin{definition}Let $G=(V,E)$ be a vertex-transitive, planar graph. Let $\Phi$ be the set of all contour configurations on $G$, in which each contour configuration is a subset of edges such that each vertex is incident to an even number of present edges. Let $\nu$ be a probability measure on $\Phi$. We say $\nu$ has \textbf{finite energy} if for any face $S$ of $G$, let $\partial S\subset E$ consist of of all the sides of the polygon $S$. Define $\phi_S$ to be the configuration obtained by switching the states of each element of $\partial S$, i.e.\ $\phi_S(e)=1-\phi(e)$ if $e\in\partial S$, and $\phi_S(e)=\phi(e)$ otherwise. Let $E$ be an event, and define
\begin{eqnarray}
E_S=\{\phi_S:\phi\in E\}.\label{ef}
\end{eqnarray}
Then $\nu(E_S)>0$ whenever $\nu(E)>0$.
\end{definition}

 We may assume that $\nu_1$ or $\nu_2$ has \df{finite energy} as follows.
\begin{enumerate}[label=({A}{\arabic*})]
\setcounter{enumi}{3}
\item $\nu_1$  has finite energy.
\item $\nu_2$  has finite energy.
\end{enumerate}

See \Cref{feh,fet} for illustrations of the configuration-changing process on the hexagonal lattice $\HH$ and the triangular lattice $\TT$, respectively.

\begin{figure}
 \centering
  \includegraphics[width=0.2\textwidth]{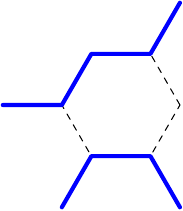}\qquad\qquad\qquad\qquad \includegraphics[width=0.2\textwidth]{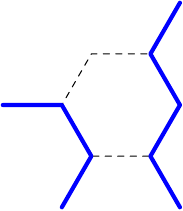}
 \caption{Change of contour configurations in $\LL_1=\HH$}\label{feh}
 \end{figure}
 
 \begin{figure}
 \centering
  \includegraphics[width=0.15\textwidth]{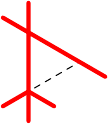}\qquad\qquad\qquad\qquad\qquad \includegraphics[width=0.15\textwidth]{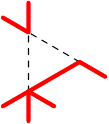}
 \caption{Change of contour configurations in $\LL_2=\TT$}\label{fet}
 \end{figure}

For a random contour configuration $\psi\in\Phi_1$ (resp.\ $\psi\in \Phi_2$) whose distribution is the marginal distribution $\nu_1$ (resp.\ $\nu_2$) of $\mu$ on $\Phi_1$ (resp.\ $\Phi_2$), $\psi$ induces a random constrained configuration $\omega\in \phi^{-1}(\psi)$ as follows. Let $v_0$ be a fixed vertex of $G$. Assume that $\omega(v_0)=1$ with probability $\frac{1}{2}$, and $\omega(v_0)=0$ with probability $\frac{1}{2}$, and is independent of $\psi$. For two vertices $v_1,v_2$ of $G$ joined by an edge $e$, $v_1$ and $v_2$ have different states if and only if $e$ crosses a present edge in $\psi$. Let $\lambda_1$ (resp.\ $\lambda_2$) be the distribution of $\omega$. We may further make the following assumptions
\begin{enumerate}[label=({A}{\arabic*})]
\setcounter{enumi}{5}
\item $\lambda_1$ is $\Gamma_1$-ergodic;
\item $\lambda_2$ is $\Gamma_2$-ergodic.
\end{enumerate}

Also we may sometimes assume that 
\begin{enumerate}[label=({A}{\arabic*})]
\setcounter{enumi}{7}
\item $\mu$ is $\Gamma_1$-invariant.
\end{enumerate}

The main theorems of this section are stated as follows.

\begin{theorem}\label{m23}Let $G$ be the $[3,4,n,4]$ lattice with $n\geq 7$. Let $s_0$ (resp.\ $s_1$) be the number of infinite 0-clusters (resp.\ 1-clusters). Let $t_1$ (resp.\ $t_2$) be the number of infinite $\LL_1$-contours (resp.\ $\LL_2$-contours).
\begin{enumerate}
\item Let $\mu$ be a probability measure on $\Omega$ satisfying (A2),(A3),(A7),(A8).  Then $\mu$-a.s. $(s_0,s_1,t_1)=(\infty,\infty,\infty)$.
\item Let $\mu$ be a probability measure on $\Omega$ satisfying (A2),(A3),(A6),(A7),(A8). Then $\mu$-a.s. $(s_0,s_1,t_1,t_2)\in\{(\infty,\infty,\infty,1),(\infty,\infty,\infty,\infty)\}$.
 \end{enumerate}
\end{theorem}

The case $m=3,n\geq 7$ is of special interest, because in this case $\LL_1$ is a cubic graph (each vertex has degree 3), and $\LL_2$ is a triangular tiling of the hyperbolic plane. As a result, any infinite contour on $\LL_1$ must be a doubly infinite self-avoiding path. An application of Theorem \ref{m23} is illustrated in the following example.

\begin{example}\label{exl2}Consider the i.i.d.\ Bernoulli site percolation on the regular tiling $\LL_2$ of the hyperbolic plane with triangles, such that each vertex has degree $n\geq 7$. Assume that each vertex of $\LL_2$ takes value 1 with probability $\frac{1}{2}$. The corresponding contour configuration on the dual graph $\LL_1$ to the site percolation on $\LL_2$ induces a constrained configuration in the $[3,4,n,4]$ lattice satisfying (A8),(A2),(A3),(A7). Then by \Cref{m23} $\mu$-a.s.  $(s_0,s_1)=(\infty,\infty)$.
\end{example}

\begin{theorem}\label{m21}Let $G$ be the $[m,4,n,4]$ lattice such that
\begin{eqnarray}
&&m\geq 3,\qquad n\geq 3;\label{mn1}\\
&&\frac{1}{m}+\frac{1}{n}=\frac{1}{2}\notag
\end{eqnarray}
Let $\mu$ be a probability measure on $\Omega$. Then
\begin{enumerate}
\item if $\mu$ satisfies (A1)-(A6), then almost surely there are no infinite contours in $\LL_2$;
\item if $\mu$ satisfies (A1)-(A5) and (A7), then almost surely there are no infinite contours in $\LL_1$;
\item if $\mu$ satisfies (A1)-(A7), almost surely there are neither infinite contours nor infinite clusters.
\end{enumerate}
\end{theorem}

When $m=n$, the two lattice $\LL_1$ and $\LL_2$ are isomorphic to each other, this allows us to use symmetry to obtain the following theorem.

\begin{theorem}\label{m22}Let $G$ be the $[m,4,n,4]$ lattice satisfying
\begin{eqnarray*}
m=n\geq 5.
\end{eqnarray*}
Let $\mu$ be a probability measure on $\Omega$. Let $s_0$ (resp.\ $s_1$) be the number of infinite 0-clusters (resp.\ 1-clusters), and let $t_1$ (resp.\ $t_2$) be the number of infinite $\LL_1$-contours (resp.\ $\LL_2$-contours). If $\mu$ satisfies (A1)-(A3); then $\mu$-a.s. $(s_0,s_1,t_1,t_2)=(\infty,\infty,\infty,\infty)$.
\end{theorem}

\Cref{m23} is proved in \Cref{p23}. \Cref{m21} is proved in \Cref{p212}, and \Cref{m22} is proved in \Cref{p211}.

\section{Ising model on transitive, triangular tilings of the hyperbolic plane}\label{Is}

In this section, we state the main result concerning the percolation properties of the Ising model on transitive, triangular tilings of the hyperbolic plane. These results, as given in \Cref{ipl}, will be proved in \Cref{pipl}.

The random cluster representation of an Ising model on a transitive, triangular tiling of the hyperbolic plane can be defined as in \cite{HJL02}. Here we briefly summarize basic facts about the Fortuin-Kasteleyn random cluster model, which is a probability measure on bond configurations of a graph, and the related Potts model. See \cite{GrGrc} for more information.

The \textbf{random cluster measure} $RC:=RC_{p,q}^{G_0}$ on a finite graph $G_0=(V_0,E_0)$ with parameters $p\in[0,1]$ and $q\geq 1$ is the probability measure on $\{0,1\}^{E_0}$ which to each $\xi\in \{0,1\}^{E_0}$ assigns probability
\begin{eqnarray}
RC(\xi):\propto q^{k(\xi)}\prod_{e\in E_0}p^{\xi(e)}(1-p)^{1-\xi(e)}.\label{drc}
\end{eqnarray}
where $k(\xi)$ is the number of connected components in $\xi$.

Let $G=(V,E)$ be an infinite, locally finite, connected graph. For each $q\in[1,\infty)$ and each $p\in (0,1)$, let $WRC_{p,q}^G$ be the random cluster measure with the wired boundary condition, and let $FRC_{p,q}^G$ be the random cluster measure with the free boundary condition. More precisely, $WRC_{p,q}^G$  (resp.\ $FRC_{p,q}^G$) is the weak limit of $RC$'s defined by (\ref{drc}) on larger and larger finite subgraphs approximating $G$, where we assume that all the edges outside each finite subgraph are present (resp.\ absent).

The Gibbs measure $\mu^{+}$ (resp.\ $\mu^{-}$) for the Ising model on $G$ with coupling constant $J\geq 0$ on each edge and ``$+$''-boundary conditions (resp.\ ``$-$''-boundary conditions) can be obtained by considering a random configuration of present and absent edges according to the law $WRC_{p,2}^G$, $p=1-e^{-2J}$, and assigning to all the vertices in each infinite cluster the state ``$+$'' (resp.\ ``$-$''), and to all the vertices in each finite cluster a state from $\{+,-\}$, chosen uniformly at random for each cluster and independently for different clusters. 

The Gibbs measure $\mu^{f}$ for the Ising model on $G$ with coupling constant $J\geq 0$ on each edge and free boundary conditions can be obtained by considering a random configuration of present and absent edges according to the law $FRC_{p,2}^G$, $p=1-e^{-2J}$, and assigning to all the vertices in each  cluster a state from $\{+,-\}$, chosen uniformly at random for each cluster and independently for different clusters. 

When there is no confusion, we may write $FRC_{p,q}^G$ and $WRC_{p,q}^G$ as $FRC_{p,q}$ and $WRC_{p,q}$ for simplicity.
Assume that $G$ is transitive. Then measures $FRC_{p,q}$ and $WRC_{p,q}$ are $\mathrm{Aut}(G)$-invariant, and $\mathrm{Aut}(G)$-ergodic; see the explanations on Page 295 of \cite{RS01}. 

Now we introduce the following definitions.

\begin{definition}A transitive graph $G=(V,E)$ is \textbf{unimodular}, if there exists a subgroup $\Gamma\subseteq \mathrm{Aut}(G)$ acting transitively on $G$, such that for any two vertices $u,v\in V$, we have
\begin{eqnarray*}
|\mathrm{Stab}_{u}(v)|=|\mathrm{Stab}_v(u)|;
\end{eqnarray*}
where $\mathrm{Stab}_{u}\subseteq \Gamma$ is the stabilizer of $u$, i.e. 
\begin{eqnarray}
\mathrm{Stab}_u=\{\gamma\in \Gamma: \gamma u=u\};\label{stab} 
\end{eqnarray}
and $|\cdot|$  is the cardinality of a set.
\end{definition}

\begin{definition}A graph $G=(V,E)$ is called \textbf{amenable}, if its edge isoperimetric constant
\begin{eqnarray}
\imath_E(G):=\mathrm{inf}_{K\subset V,|K|<\infty}\frac{|\partial_E K|}{|K|}=0.\label{eic}
\end{eqnarray}
where $\partial_E K$ is the set of edges with exactly one endpoint in $K$ and one endpoint not in $K$.

If the edge isoperimetric constant is strictly positive, the graph is called \textbf{nonamenable}.
\end{definition}

If we further assume that $G$ is unimodular, nonamenable and planar,
it is known that there exists $p_{c,q}^{w},\ p_{c,q}^{f},\ p_{u,q}^{w},\ p_{c,q}^{f}\in[0,1]$, such that $FRC_{p,q}$-a.s. the number of infinite clusters equals
\begin{eqnarray}
\left\{\begin{array}{cc}0&\mathrm{for}\ p\leq p_{c,q}^f\\ \infty&\mathrm{for}\ p\in (p_{c,q}^f,p_{u,q}^f)\\1&\mathrm{for}\ p> p_{u,q}^f; \end{array}\right.\label{frc}
\end{eqnarray}
and $WRC_{p,q}$-a.s. the number of infinite clusters equals
\begin{eqnarray}
\left\{\begin{array}{cc}0&\mathrm{for}\ p< p_{c,q}^w\\ \infty&\mathrm{for}\ p\in (p_{c,q}^w,p_{u,q}^w)\\1&\mathrm{for}\ p\geq p_{u,q}^w.\end{array}\right.;\label{wrc}
\end{eqnarray}
see expressions (17),(18), Theorem 3.1 and Corollary 3.7 of \cite{HJL02}.

It is well known that for the i.i.d Bernoulli percolation on a infinite, connected, locally finite transitive graph $G$, there exist $p_c,p_u$ such that
\begin{enumerate}[label=(\roman*)]
\item $0<p_c\leq p_u\leq 1$;
\item for $p\in[0,p_c)$ there is no infinite cluster a.s.
\item for $p\in(p_c,p_u)$ there are infinitely many infinite clusters, a.s.
\item for $p\in(p_u,1]$, there is exactly one infinite cluster, a.s.
\end{enumerate}
The monotonicity in $p$ of the uniqueness of the infinite cluster was proved in \cite{HP,Sch99}. Combining with Theorem 7.5 of \cite{LP} (proved first in \cite{ns81}), we obtain statements (i)-(iv) above. It is proved that $p_c=p_u$ for amenable transitive graphs (see \cite{BS96}); and conjectured that $p_c<p_u$ for transitive non-amenable graphs. The conjecture $p_c<p_u$ was proved for transitive hyperbolic planar graphs (see \cite{BS20}) and  non-amenable Cayley graphs with small spectral radii (see \cite{PSN,RS01,Th15}) or large girths (see \cite{NP12}).

\begin{theorem}\label{ipl}Let $\LL_2$ be a triangulation of the hyperbolic plane such that each vertex has degree $n\geq 7$. Consider the Ising model with spins located on vertices of $\LL_2$ and coupling constant $J\in \RR$ on each edge.  Let $p_c$, $p_u$ be critical i.i.d Bernoulli site percolation probabilities on $\LL_2$ as defined by (i)-(iv) above.
\begin{enumerate}
\item Let $h>0$ satisfy
\begin{eqnarray}
\frac{e^{-h}}{e^{h}+e^{-h}}=p_c\label{pch}
\end{eqnarray}
  Let $\mu^+$ (resp.\ $\mu^-$. $\mu^f$) be the infinite-volume Ising Gibbs measure with ``$+$''-boundary conditions (resp.\ ``$-$'' boundary conditions, free boundary conditions).  If 
 \begin{eqnarray}
 n|J|<h,\label{jh} 
 \end{eqnarray}
 then $\mu$-a.s. there are infinitely many infinite ``$+$''-clusters, infinitely many infinite ``$-$''-clusters and infinitely many infinite contours, where $\mu$ is an arbitrary $\mathrm{Aut}(\LL_2)$-invariant Gibbs measure for the Ising model on $\LL_2$ with coupling constant $J$.
 
  \item  Assume $J\geq 0$. If one of the following conditions 
 \begin{enumerate}[label=(\alph*)]
\item  $\mu^f$ is $\mathrm{Aut}(\LL_2)$-ergodic; 
\item $\mathrm{inf}_{u,v\in V(\LL_2)}\langle \sigma_{u}\sigma_{v}\rangle_{\mu^f}=0$, where $\sigma_{u}$ and $\sigma_{v}$ are two spins associated to vertices $u,v\in V(\LL_2)$ in the Ising model; 
\item $0\leq J<\frac{1}{2}\ln\left(\frac{1}{1-p_{u,2}^{f}}\right)$, where $p_{u,2}^{f}$ is the critical probability for the existence of a unique infinite open cluster of the corresponding random cluster representation of the Ising model on $\LL_2$, with free boundary conditions as given in (\ref{frc});  
\item  $0\leq J<\frac{1}{2}\ln\left(\frac{1}{1-p_{u,1}}\right)$, where $p_{u,1}$ is the critical probability for the existence of a unique infinite open cluster for the i.i.d Bernoulli bond percolation on $\LL_2$;
\end{enumerate} 
  holds, then $\mu^f$-a.s. there are infinitely many infinite ``$+$''-clusters and infinitely many infinite ``$-$''-clusters. Indeed, we have $(d)\Rightarrow(c)\Rightarrow(b)\Rightarrow (a)$.

  \item Assume
 \begin{eqnarray}
 J\geq\frac{1}{2}\ln\left(\frac{1}{1-p_{u,2}^{w}}\right).\label{jj1}
 \end{eqnarray}
 Let $\sA_+$ be the event that there is a unique infinite ``$+$''-cluster, no infinite ``$-$''-clusters and no infinite contours; and let $\sA_-$ be the event that there is a unique infinite ``$-$''-cluster, no infinite ``$+$''-clusters and no infinite contours.
  then 
 \begin{eqnarray}
 \mu^+(\sA_+)&=&1.\label{+1}\\
\mu^-(\sA_-)&=&1. \label{-1}
 \end{eqnarray}
  \item If 
 \begin{eqnarray}
 J>\frac{1}{2}\ln\left(\frac{1}{1-p_{u,2}^{f}}\right)\label{jj2}
 \end{eqnarray}
  then (\ref{+1}) and (\ref{-1}) hold, and moreover,
   \begin{eqnarray}
\mu^{f}(\sA_+)=\mu^{f}(\sA_-)=\frac{1}{2}.\label{hmf}
 \end{eqnarray}
  \end{enumerate}
\end{theorem}

From Theorem \ref{ipl}, we can also obtain the following corollary:

\begin{corollary}\label{c34}Let $\LL_2$ be a triangulation of the hyperbolic plane such that each vertex has degree $n\geq 7$. Consider the Ising model with spins located on vertices of $\LL_2$ and coupling constant $J\geq 0.$ on each edge. If 
  \begin{eqnarray}
  J< \frac{1}{2}\ln\left(\frac{1}{1-p_{c,2}^{w}}\right),\label{jj3}
  \end{eqnarray}
  then for any Gibbs measure $\mu$ for the Ising model on $\LL_2$ with coupling constant $J$,  $\mu$-a.s. there are infinitely many infinite ``$+$''-clusters, infinitely many infinite ``$-$''-clusters and infinitely many infinite contours. Here $p_{c,2}^{w}$ is defined as in (\ref{wrc}).
\end{corollary}

We can see that when the conditions of Corollary \ref{c34} are satisfied, almost surely there are no infinite open clusters in the corresponding random cluster representation of the Ising model, however, the conclusion of the corollary says that there are infinitely many infinite ``$+$''-clusters and infinitely many infinite ``$-$''-clusters in the Ising model.

\section{XOR Ising model on transitive, triangular tilings of the hyperbolic plane}\label{xIs}

In this section, we state the main result concerning the percolation properties of the XOR Ising model on transitive, triangular tilings of the hyperbolic plane. These results, as given in \Cref{coii,xorc}, will be proved in \Cref{pxorc} as applications of \Cref{m23}.

Throughout this section, we let $\LL_1$ be the $[n,n,n]$ regular tiling of the hyperbolic plane, such that each face has degree $n\geq 7$, and each vertex has degree 3. Let $\LL_2$ be the planar dual graph of $\LL_1$. More precisely,  $\LL_2$ is the vertex-transitive triangular tiling of the hyperbolic plane such that each vertex has degree $n\geq 7$. An \textbf{XOR Ising model} on $\LL_2$ is a probability measure on $\sigma_{XOR}\in \{\pm 1\}^{V(\LL_2)}$, such that 
\begin{eqnarray*}
\sigma_{XOR}(v)=\sigma_1(v)\sigma_2(v),\qquad\forall v\in V(\LL_2),
\end{eqnarray*}
where $\sigma_1$, $\sigma_2$ are two i.i.d.\ Ising models with spins located on $V(\LL_2)$. A \textbf{contour configuration} of the XOR Ising configuration on $\LL_2$ is a subset of edges of $\LL_1$ in which each edge has a dual edge in $E(\LL_2)$ joining two vertices $u,v\in V(\LL_2)$ satisfying $\sigma_{XOR}(u)=-\sigma_{XOR}(v)$. A connected component in a contour configuration is called a \textbf{contour}. Obviously each vertex of $\LL_1$ has 0 or 2 incident present edges in a contour configuration of an XOR Ising configuration, since $\LL_1$ has vertex degree 3. Each contour of an XOR Ising configuration on $\LL_2$  is either a self-avoiding cycle or a doubly-infinite self-avoiding path.

We can similarly define an XOR Ising model with spins located on vertices of $\LL_1$, and its contours to be even-degree subgraphs of $\LL_2$.

\begin{theorem}\label{coii} Let $\sigma_1$, $\sigma_2$ be two i.i.d.\ Ising models with spins located on vertices of $\LL_2$, coupling constant $J\in [0,\infty)$ and free boundary conditions. Let $\mu_1^f$ (resp.\ $\mu_2^f$) be the distribution of $\sigma_1$ (resp.\ $\sigma_2$). Assume that one of the following cases occurs
\begin{enumerate}
\item If $\mu_1^f\times \mu_2^f$ is $\mathrm{Aut}(\LL_2)$-ergodic; or
\item $\liminf_{|i-j|\rightarrow\infty}\langle \sigma_{1,i}\sigma_{1,j}\rangle_{\mu_{1,f}}=0$, where $\sigma_{1,i}$ and $\sigma_{1,j}$ are two spins in the Ising model $\sigma_1$ with distance $|i-j|$; or
\item $J$ satisfies Condition (c) of Theorem \ref{ipl} II. 
\item  $J$ sasifies Condition (d) of Theorem \ref{ipl} II.
\end{enumerate}
then $\mu_1^f\times \mu_2^f$-a.s. there are infinitely many infinite ``$+$''-clusters and infinitely many infinite ``$-$''-clusters.
\end{theorem}

\begin{theorem}\label{xorc} Let $\sigma_1$, $\sigma_2$ be two i.i.d.\ Ising models with spins located on vertices of $\LL_1$, and coupling constant $K\geq 0$. For $i=1,2$, let $\mu_{i,+}$ (resp.\ $\mu_{i,-}$) be the distribution of $\sigma_i$ with ``$+$''-boundary conditions (resp.\ ``$-$''-boundary conditions). Let $J\geq 0$ be given by
\begin{eqnarray}
e^{-2J}=\frac{1-e^{-2K}}{1+e^{-2K}},\label{jkr}
\end{eqnarray}
and let $t$ be the number of infinite contours. Let $\mu_{++}$ (resp.\ $\mu_{--}$, $\mu_{+-}$) be the product measure of $\mu_{1,+}$ and $\mu_{2,+}$ (resp.\ $\mu_{1,-}$ and $\mu_{2,-}$, $\mu_{1,+}$ and $\mu_{2,-}$). Assume $J$ satisfies the assumption of \Cref{coii}, then we have
\begin{eqnarray*}
\mu_{++}(t\in\{0,\infty\})=\mu_{--}(t\in\{0,\infty\})=\mu_{+-}(t\in\{0,\infty\})=1.
\end{eqnarray*}
\end{theorem}

\section{XOR Ising models on the hexagonal and triangular lattices}\label{xori}

In this section, we define the XOR Ising models on the hexagonal and triangular lattices, and state the main results proved in this paper concerning the percolation properties of these models.

Let $\sigma_1$, $\sigma_2$ be two i.i.d.\ ferromagnetic Ising models with spins located on vertices of the hexagonal lattice $\HH=(V_{\HH}, E_{\HH})$. The hexagonal lattice has edges in three different directions. Assume that both $\sigma_1$ and $\sigma_2$ have nonnegative coupling constants $J_a$, $J_b$, $J_c$ on edges of $\HH$ with the three different directions, respectively. Assume also that the distributions of both $\sigma_1$ and $\sigma_2$ are weak limits of Gibbs measures under periodic boundary conditions. Recall that the XOR Ising model $\sigma_{XOR}(v)=\sigma_1(v)\sigma_2(v)$, for $v\in V_{\HH}$. 

A \df{contour configuration} for an XOR Ising configuration, $\sigma_{XOR}$, defined on $\HH$ (resp.\ $\TT$), is a subset of $\{0,1\}^{E(\TT)}$ (resp.\ $\{0,1\}^{E(\HH)}$), whose state-1-edges (present edges) are edges of $\TT$ (resp.\ $\HH$) separating neighboring vertices of $\HH$ (resp.\ $\TT$) with different states in $\sigma_{XOR}$. (Note that $\HH$ and $\TT$ are planar duals of each other.)  Contour configurations of the XOR Ising model were first studied in \cite{DBW11}, in which the scaling limits of contours of the critical XOR Ising model are conjectured to be level lines of Gaussian free field. It is proved in \cite{bd14} that the contours of the XOR Ising model on a plane graph correspond to level lines of height functions of the dimer model on a decorated graph, inspired by the correspondence between Ising model and bipartite dimer model in \cite{Dub}. We will study  the percolation properties of the XOR Ising model on $\HH$ and $\TT$, as an application of the main theorems proved in this paper for the general constrained percolation process.

Let
\begin{eqnarray}
&&f(x,y,z)=e^{-2x}+e^{-2y}+e^{-2z}+e^{-2(x+y)}+e^{-2(x+z)}+e^{-2(y+z)}-e^{-2(x+y+z)}-1.\label{fxyz}\\
&&g(x,y,z)=e^{2x}+e^{2y}+e^{2z}-e^{2(x+y+z)}.\label{gxyz}
\end{eqnarray}

We say the XOR Ising model on $\HH$ with coupling constants $(J_a,J_b,J_c)$ is in the \df{high-temperature state} (resp.\ \df{low-temperature state}, \df{critical state}) if $f(J_a,J_b,J_c)>0$ (resp.\ $f(J_a,J_b,J_c)<0$, $f(J_a,J_b,J_c)=0$). Note that $f(J_a,J_b,J_c)=0$ is the well-known condition that an Ising model on the 2D hexagonal lattice $\HH$ is critical; see, for example, \cite{ZL12} for a rigorous proof. The XOR Ising model $\sigma_1\cdot \sigma_2$ on $\HH$ is in the high-temperature state (resp.\ low-temperature state, critical state), if and only if both $\sigma_1$ and $\sigma_2$ are in the high-temperature state (resp.\ low-temperature state, critical state).

Let $\TT=(V_{\TT},E_{\TT})$ be the dual triangular lattice of $\HH$. We also consider the XOR Ising model with spins located on $V_{\TT}$. Assume that the coupling constants on edges with 3 different directions are $K_a$, $K_b$ and $K_c$, respectively, such that $K_a, K_b, K_c\geq 0$. Also for $i\in\{a,b,c\}$, assume that $K_i$ is the coupling constant on an edge of $\TT$ dual to an edge of $\HH$ with coupling constant $J_i$. We say the XOR Ising model on the triangular lattice is in the \df{low-temperature state} (resp.\ \df{high-temperature state}, \df{critical state}) if $g(K_a,K_b,K_c)<0$ (resp. $g(K_a,K_b,K_c)>0$, $g(K_a,K_b,K_c)=0$). Again these come from the known fact that  if $g(K_a,K_b,K_c)<0$ (resp. $g(K_a,K_b,K_c)>0$, $g(K_a,K_b,K_c)=0$), both Ising models, each of which is a factor of the XOR Ising model, are in the low-temperature state (resp.\ high-temperature state, critical state).

Similar to the square grid case, in the high temperature state, the Ising model on the hexagonal lattice or the triangular lattice has a unique Gibbs measure, and the spontaneous magnetization vanishes; while in the low temperature state, the Gibbs measures are not unique and the spontaneous magnetization is strictly positive under the ``$+$''-boundary condition. See  \cite{JL72, Ai80,ZL12,DC13}.

If 
\begin{eqnarray}
e^{-2K_{\tau}}=\frac{1-e^{-2J_{\tau}}}{1+e^{-2J_{\tau}}},\qquad\mathrm{for}\ \tau=a,b,c,\label{dua}
\end{eqnarray}
then the XOR Ising model on $\HH$ with coupling constants $(J_a,J_b,J_c)$ is in the low-temperature state (resp.\ high-temperature state, critical state) if and only if the XOR Ising model on the triangular lattice with coupling constants $(K_a,K_b,K_c)$ is in the high-temperature state (resp.\ low-temperature state, critical state).

We define clusters and contours with respect to  an XOR Ising configuration on $\HH$ or $\TT$ in the usual way. Then we have the following theorems.

\begin{theorem}\label{chi}Consider the critical XOR Ising model on $\HH$ or $\TT$. Then  
\begin{enumerate}
\item almost surely there are no infinite clusters;
\item almost surely there are no infinite contours.
\end{enumerate}
\end{theorem}

\begin{theorem}\label{lth}In the low-temperature XOR Ising model on $\HH$ or on $\TT$, almost surely there are no infinite contours.
\end{theorem}

\Cref{chi,lth} are proved in \Cref{p412}.

\section{Square tilings of the hyperbolic plane}\label{sthp}

In this section, we introduce the square tilings of the hyperbolic plane, and then state and prove properties of the constrained percolation models on such graphs. We first discuss known results about percolation on non-amenable graphs that will be used to prove main theorems of the paper.

The following lemma is proved in \cite{BS20,blps}.

\begin{lemma}\label{lbs}Let G be a quasi-transitive, non-amenable, planar graph with one end,
and let $\omega$ be an invariant percolation on G. Then a.s. the number of infinite 1-clusters
of $\omega$ is 0, 1, or $\infty$.
\end{lemma}

\begin{proof}See Lemma 3.5 of \cite{BS20}. 
\end{proof}

\begin{lemma}(Threshold for bond percolation on non-amenable graphs)\label{tbng} Let $G=(V,E)$ be a non-amenable graph. Let $\Gamma\subseteq \mathrm{Aut}(G)$ be a closed unimodular quasi-transitive subgroup, and let $o_1,\ldots,o_L$ be a complete set of representatives in $V$ of the orbits of $G$. For $1\leq i\leq L$, let $\mathrm{Stab}_{o_i}$ is defined as in (\ref{stab}) and
\begin{eqnarray*}
\eta_i:&=&|\mathrm{Stab}_{o_i}|.
\end{eqnarray*}
 Let $\mathbb{P}$ be a bond percolation on $G$ whose distribution is $\Gamma$-invariant. Let $D_i$ be the random degree of $o_i$ in the percolation subgraph, and let $d_i$ be the degree of $o_i$ in $G$. Write $p_{\infty,v}$ for the probability that $v\in V$ is in an infinite component. Let $p_{\infty,i}$ be the probability that $o_i$ is in an infinite cluster. Then
\begin{eqnarray}
\sum_{i=1}^{L}\frac{(d_i-\alpha(G))p_{\infty,i}}{\eta_i}\geq \sum_{j=1}^{L}\frac{\mathbb{E}D_j-\alpha(G)}{\eta_j}\label{cic}
\end{eqnarray}
where $\alpha(G)$ is a constant depending on the structure of the graph $G$ defined by
\begin{eqnarray*}
\alpha_K:&=&\frac{1}{|K|}\sum_{x\in K}\deg_K(x)\\
\alpha(G):&=&\sup\{\alpha_K: K\subset G\ \mathrm{is\ finite}\}
\end{eqnarray*}
In particular, if the right-hand side of (\ref{cic}) is positive, then there is an infinite component in the percolation subgraph with positive probability.
\end{lemma}
\begin{proof}See Theorem 4.1 of \cite{blps}.
\end{proof}

Let $G=(V,E)$ be a graph corresponding to a square tiling of the hyperbolic plane. Assume that 
\begin{enumerate}
\item each face of $G$ has 4 edges; and
\item each vertex of $G$ is incident to $2n$ faces, where $n\geq 3$.
\end{enumerate}
See Figure \ref{fig:46} for an example of such a graph when $n=3$.

\begin{figure}
\includegraphics[trim=0 220  100 220,clip,width=.6\textwidth]{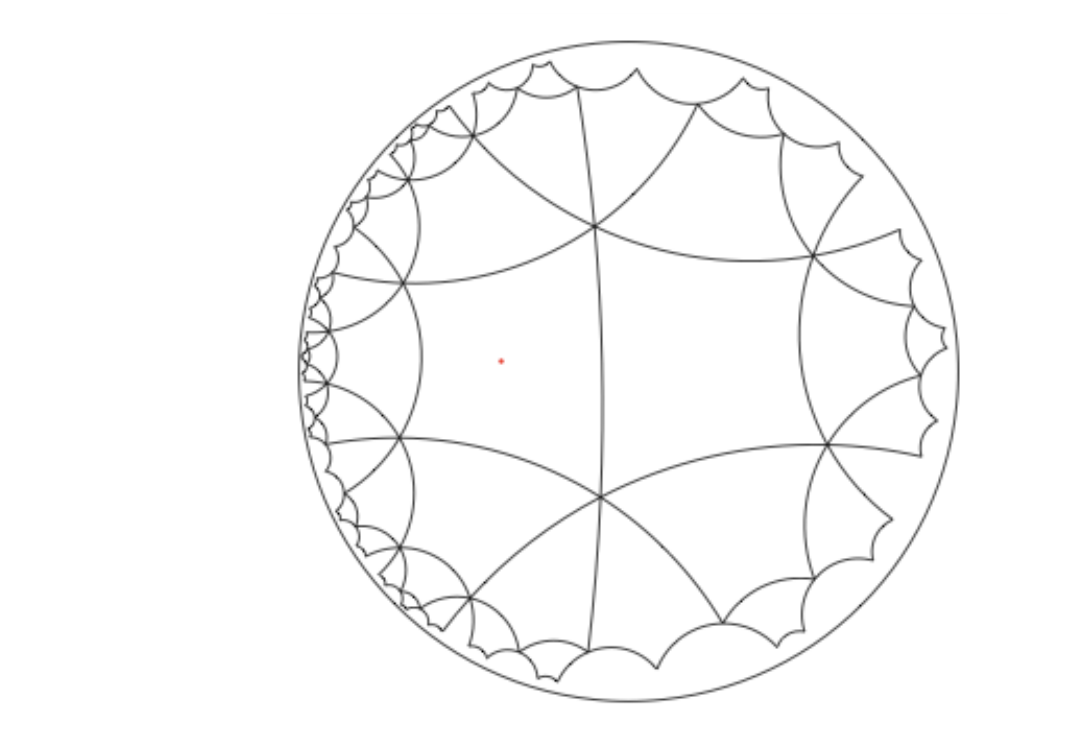}
\caption{The [4,4,4,4,4,4] lattice: each face has degree 4, and each vertex have degree 6}
\label{fig:46}
\end{figure}

We can color all the faces of $G$ by black and white such that black faces can share edges only with white faces and vice versa. Let $G=(V,E)$ denote the graph embedded into the hyperbolic plane as described above. 

We consider the site configurations in $\{0,1\}^V$. We impose the following constraint on site configurations
\begin{itemize}
\item Around each black face, there are six allowed configurations $(0000)$, $(1111)$, $(0011)$, $(1100)$, $(0110)$, $(1001)$, where the digits from the left to the right correspond to vertices in clockwise order around the black face, starting from the lower left corner. See Figure \ref{lcc}.
\end{itemize}
Let $\Omega\subset\{0,1\}^V$ be the set of all configurations satisfying the constraint above. We use $\Omega$ to denote the sample space throughout this paper, however, $\Omega$ have different meanings in different sections.

Note that $G$ is a vertex-transitive graph. Since each face of $G$ has an even number of edges, $G$ itself is a bipartite graph - we can color the vertices of $G$ by red and green such that red vertices are adjacent only to green vertices and vice versa. We assign an integer in $1,2,\ldots, n$ to each white face of $G$ according to the following rules
\begin{enumerate}
\item around each red vertex of $G$, white faces are assigned integers $1,2,\ldots,n$ clockwise; and
\item around each green vertex of $G$, white faces are assigned integers $1,2,\ldots,n$ counterclockwise; and
\item any two white faces adjacent to the same black face along two opposite edges have the same assigned integer.
\end{enumerate}
See Figure \ref{fig:label} for an example of assignments of integers $1,2,3$ to the white faces of the $[4,4,4,4,4,4]$ lattice.

\begin{figure}
\includegraphics[width=.4\textwidth]{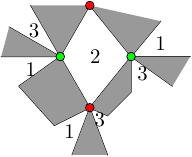}
\caption{Labels of white faces of the [4,4,4,4,4,4] lattice}
\label{fig:label}
\end{figure}

For $1\leq i\leq n$, we construct a graph $\LL_i$ as follows. The vertex set of $\LL_i$ consists of all the white faces of $G$ whose assigned integers are $i$. Two vertices of $\LL_i$ are joined by an edge of $\LL_i$ if and only if they correspond to two white faces of $G$ adjacent to the same black face along two opposite edges. We have the following proposition regarding the connected components of $\LL_i$

\begin{proposition}\label{l63}
When $n\geq 3$, each component of $\LL_i$ ($1\leq i\leq n$) is a regular tree of degree 4. For any integer $i$ satisfying $1\leq i\leq n$, the edges of $\LL_{i-1}$ (if $i=1$, $\LL_{i-1}:=\LL_n$) and $\LL_i$ cross; the edges of $\LL_{i+1}$ (if $i=n$, $\LL_{i+1}:=\LL_1$) and $\LL_i$ cross.
\end{proposition}

\begin{proof}We consider a doubly infinite sequence of edges in $G$ consisting of edges $\ldots, e_{-1}, e_0,e_1,e_2,\ldots,$ such that
\begin{itemize}
\item For each $k\in \ZZ$, $e_k$ and $e_{k+1}$ share a vertex $v$, such that there are exactly $(n-1)$ edges incident to $v$ to the left of $e_k$ and $e_{k+1}$, and $(n-1)$ edges incident to $v$ to the right of $e_k$ and $e_{k+1}$.
\end{itemize}
Then $\ldots, e_{-1},e_0,e_1,\ldots$ form a doubly infinite self-avoiding path in $G$ because its left side and right side are symmetric. Indeed, if the path crosses itself, starting from $e_0$, we move the path along both the positive direction $e_1,e_2,\ldots$ and the negative direction $e_{-1}, e_{-2},\ldots$, until the first time the movements along the two directions meet, and form a cycle $\mathcal{C}_{ab}:=e_{-a},e_{-a+1},\ldots, e_0, \ldots, e_{b-1},e_b$, where $a,b\in\{0,1,2,\ldots\}$. Then $G\setminus \mathcal{C}_{ab}$ has a finite component and an infinite component; moving from $e_{-a}$ to $e_b$ along $\mathcal{C}_{ab}$, the finite component is either on the left or on the right, but this is a contradiction to the fact that on the left and right side of $\ldots, e_{-1}, e_0,e_1,e_2,\ldots$, $G$ is symmetric. We call the infinite self-avoiding path obtained this way a \textbf{central path}.

Assume there is a cycle in $\LL_i$ for some $1\leq i\leq n$, then we can find a face in $\LL_i$. Let $(u,v)$ be an edge of $\LL_i$. Moving from $u$ to $v$, at $v$ there are 3 other incident edges except the edge $(u,v)$; since the graph is embedded in the hyperbolic plane, we may label the three incident edges at $v$ other than $(u,v)$ by the left edge, the middle edge, and the right edge, in such a way that starting from the edge $(u,v)$ and moving around $v$ clockwise along a small circle, one will cross the left edge first, then the middle edge, and finally the right edge.   If we can find a face in $\LL_i$, then the face can be found by always moving along the right edge at each vertex for finitely many times, and the face is on the right of an oriented cycle obtained this way. But when $n\geq 3$, this is not possible since any oriented path in $\LL_i$ obtained by always moving along the right edge at each vertex has a central path on its right, which is infinite.

Note that each black face of $G$ has two pairs of opposite edges. There exists $1\leq i\leq n$, such that along one pair of opposite edges the black face is adjacent to two white faces labeled by $i$, and along the other pair of opposite edges the black face is adjacent to two white faces labeled by $(i+1)$ (if $i=n$, then $i+1=1$). Then from the construction of $\LL_i$'s we can see that an edge of $\LL_i$ and an edge of $\LL_{i+1}$ cross at the black face of $G$.
\end{proof}

Any constrained percolation configuration in $\Omega$ gives rise to a contour configuration on $\cup_{i=1}^n\LL_{i}$. An edge $e$ in $\cup_{i=1}^n\LL_i$ is present in the contour configuration if and only if it crosses a black face $b$ in $G$, such that the states of the vertices of $b$ on the two sides separated by $e$ in the configuration are different, and any two vertices of $b$ on the same side of $e$ have the same state. This is a contour configuration satisfying the condition that each vertex in $\cup_{i=1}^n\LL_i$ has an even number of incident present edges. For any $1\leq i<j\leq n$, present edges in $\LL_i$ and $\LL_j$ can never cross. 

A \df{cluster} is a maximal connected set of vertices in $G$ in which every vertex has the same state in a constrained percolation configuration. A \df{contour} is a maximal connected set of edges in $\cup_{i=1}^n\LL_i$ in which every edge is present in the contour configuration. Note that each contour must be a connected subgraph of $\LL_i$, for some $1\leq i\leq n$. Hence by Proposition \ref{l63}, each contour must be a tree. Since each vertex in a contour has an even number of incident present edges in the contour, each contour must be an infinite tree.

Let $\mu$ be a probability measure on $\Omega$. We may assume that $\mu$ satisfies the following conditions
\begin{enumerate}[label=(D\arabic*)]
\item $\mu$ is $\mathrm{Aut}(G)$-invariant; 
\item $\mu$ is $\mathrm{Aut}(\LL_i)$-ergodic, for $1\leq i\leq n$; 
\item $\mu$ is symmetric, i.e.\ let $\theta:\Omega\rightarrow\Omega$ be the map defined by $\theta(\omega)(v)=1-\omega(v)$, for each $v\in V$, then $\mu$ is invariant under $\theta$, that is, for any event $A$, $\mu(A)=\mu(\theta(A))$.
\end{enumerate}

Note that when $n\geq 3$, the graph $G$ is a non-amenable group. Recall that the number of \df{ends} of a connected graph is the supremum over its finite subgraphs of the number of infinite components that remain after removing the subgraph. 

Here is the main theorem concerning the properties of constrained percolations on the square tilings of the hyperbolic plane.

\begin{theorem}\label{m51}
\begin{enumerate}[label=(\alph*)]
\item Let $\mu$ be a probability measure on $\Omega$ satisfying (D1). Let $n_0$ (resp.\ $n_1$) be the number of infinite 0-clusters (resp. 1-clusters). Then $\mu$-a.s. $(n_0,n_1)\in\{(0,1),(1,0),(1,\infty),(\infty,1),(\infty,\infty)\}$.
\item Let $\nu$ be a probability measure on $\Omega$ satisfying (D1) - (D3). Then $\nu$-a.s. there are infinitely many infinite 0-clusters and infinitely many infinite 1-clusters.
\end{enumerate}
\end{theorem}

In order to prove \ref{m51}, we first prove a few lemmas.

\begin{lemma}\label{cit}In a contour configuration in $\cup_{i=1}^n\LL_i$ as described above, any contour must be an infinite tree (a tree consisting of infinite many edges of $\cup_{i=1}^n \LL_i$) in which each vertex has degree 2 or 4.
\end{lemma}
\begin{proof}This lemma is straightforward from the facts that each contour is a connected subgraph of $\LL_i$ for some $i\in\{1,\ldots,n\}$; each component of $\LL_i$ $1\leq i\leq n$ is a regular tree of degree 4, and each vertex in a contour has an even number of incident present edges.
\end{proof}

\Cref{l2} below is proved in \cite{blps} and \cite{blps99} using the mass transport principle.

\begin{lemma}\label{l2}Let G be a nonamenable graph whose automorphism group has a closed subgroup acting transitively and unimodularly on $G$, and
let $\omega$ be an invariant percolation on G which has a single component a.s. Then
$p_c(\omega) < 1$ a.s., where $p_c(\cdot)$ is the critical i.i.d.\ Bernoulli percolation probability on a graph.
\end{lemma}

\begin{proof}See Theorem 1.5 of \cite{blps}.
\end{proof}

\bigskip

\noindent\textbf{Proof of Theorem \ref{m51}} First we show that Part (a) of the theorem together with Assumptions (D2), (D3) implies Part (b). Let $\nu$ be a probability measure on $\Omega$ satisfying (D1) - (D3). By Assumption (D2) and (D3), there exists a positive integer $k$ (possibly infinite), such that $\nu(n_0=n_1=k)=1$. Then Part (b) follows from Part (a).

Now we prove Part (a). Obviously $(n_0,n_1)\in\{(0,1),(1,0)\}$ if there are no contours. Now assume that contours do exist. By \Cref{lbs}, $n_0,n_1\in\{0,1,\infty\}$. By \Cref{01c}, $n_0,n_1\in\{1,\infty\}$. Let $\phi$ be the contour configuration. If there are infinitely many contours in $\phi$, or there exists a contour of $\phi$ in which infinitely many vertices have degree 4, then $\HH^2\setminus \phi$ has infinitely many unbounded components. By \Cref{ifc}, $n_0+n_1=\infty$. Therefore $\{n_0,n_1\}\in\{(1,\infty),(\infty,1),(\infty,\infty)\}$ in this case.

 Now consider the case that the number of contours is finite and nonzero, and on each contour only finitely many vertices have degree 4. Fix an $i$ satisfying $1\leq i\leq n$, and conditional on the event that the number of contours on $\LL_i$ is finite and nonzero. Choose a contour $\tau$ on $\LL_i$ uniformly at random; then $\tau$ forms an invariant bond percolation on $\LL_i$ which has a single component. By \Cref{l2}, almost surely $\tau$ has infinitely many vertices with degree 4 - since otherwise $p_c(\tau)=1$. Therefore this case does not occur a.s.
$\hfill\Box$

\section{Proof of \Cref{m23}}\label{p23}

In this section, we prove \Cref{m23}. The idea of the proof is to consider all the possible values of $(s_0,s_1,t_1,t_2)$ and exclude those with probability 0 to occur using the symmetry and ergodicity of the probability measure. In Lemma \ref{nooc}, we exclude the case $t_1=t_2=1$; the proof is based on constructing a superimpostion $\hat{G}$ of the lattice $\LL_1$ and its dual lattice $\LL_2$; and the union of contour configurations on $\LL_1$ and $\LL_2$ form an invariant bond percolation on $\hat{G}$, in which the number of infinite clusters can only be $0,1,\infty$ by \Cref{lbs} a.s.; however, if $t_1=t_2=1$, since the contour configurations on $\LL_1$ and $\LL_2$ do not cross each other, the number of infinite clusters in the union would be $2$. Proposition \ref{pp} excludes the case when $(s_0,s_1,t_1,t_2)=(0,0,0,0)$, the proof applies planarity to obtain an infinite sequence of contours, one surrounding another, and then obtain a contradiction with non-amenability. \Cref{lm23} excludes the case $(t_1,t_2)=(0,k)$ and $(t_1,t_2)=(k,0)$ for $1\leq k\leq \infty$ by ergodicity, symmetry and planarity. \Cref{n0111} excludes the case that $(s_0,s_1)=(1,1)$ again by constructing an invariant bound percolation on $G$ with 2 infinite clusters and obtaining a contradiction to \Cref{lbs}. In the proof of \Cref{m23}, we use symmetry, ergodicity, \Cref{lbs} and \Cref{n0111} to obtain that a.s. $(s_0,s_1)\in \{(0,0),(\infty,\infty)\}$; to rule out the case $(s_0,s_1)=(0,0)$, we apply \Cref{lbs} again to show that if $(s_0,s_1)=(0,0)$, then $(s_0,s_1)\in \{(0,0,0),(0,0,1),(0,0,\infty)\}$, we then show that each of the cases has probability 0 to occur by applying \Cref{nooc,lm23} and \Cref{pp}.

 We start with \Cref{nooc}.

\begin{lemma}\label{nooc}Let $G$ be the $[m,4,n,4]$ lattice satisfying (\ref{cmn1}) $m\geq 3, n\geq 3$ and
\begin{eqnarray}
\frac{1}{m}+\frac{1}{n}<\frac{1}{2}.\label{cmn3}
\end{eqnarray}
 Let $\mu$ be a probability measure on $\Omega$ satisfying (A1). Let $t_1$ (resp.\ $t_2$) be the number of infinite $\LL_1$-contours (resp.\ $\LL_2$-contours). Then 
 \begin{eqnarray*}
 \mu((t_1,t_2)=(1,1))=0.
 \end{eqnarray*}
\end{lemma}

\begin{proof}

The proof is inspired by the proof of Corollary 3.6 of \cite{BS20}.

We embed $\LL_1$ and $\LL_2$ in the hyperbolic plane in such a way that every edge $e$ intersects its dual edge $e^*$ at one point $v_e$, and there are no other intersections of $\LL_1$ and $\LL_2$. We define a new graph $\hat{G}=(\hat{V},\hat{E})$, where $\hat{V}=V(\LL_1)\cup V(\LL_2)\cup\{v_e,e\in E(\LL_1)\}$, and an edge in $\hat{E}$ is either a half-edge of $E(\LL_1)$ joining a vertex in $V(\LL_1)$ and a vertex in $\cup\{v_e,e\in E(\LL_1)\}$, or a half-edge of $E(\LL_2)$ joining a vertex in $V(\LL_2)$ and a vertex in $\cup\{v_e,e\in E(\LL_1)\}$.

For $i\in\{1,2\}$, let $\phi_i\in \Phi_i$ be the random contour configuration restricted on $\LL_i$. Let
\begin{eqnarray*}
\hat{\phi}:=\{[v,v_e]\in \hat{E}: v\in V(\LL_1),e\in \phi_1\}\cup\{[v_*,v_e]\in \hat{E}: v_*\in V(\LL_2),e_*\in \phi_2\}
\end{eqnarray*}
We say $\hat{\phi}$ is a \textbf{contour configuration} on $\hat{G}$, and each connected component of $\hat{\phi}$ is called a \textbf{contour}.
Then $\hat{\phi}$ is an invariant bond percolation on the quasi-transitive, non-amenable, planar, one-ended graph $\hat{G}$. Note that the number of infinite components of $\hat{\phi}$ is the number of infinite contours of $\phi_1$ plus the number of infinite contours of $\phi_2$. If there is a positive probability that $(t_0,t_1)=(1,1)$, then the number of infinite components in $\hat{\phi}$ is 2. This contradicts \Cref{lbs}, which says that the number of infinite components in the invariant percolation $\hat{\phi}$ on the quasi-transitive, one-ended, nonamenable, planar graph $\hat{G}$ can only be $0,1$ or $\infty$.
\end{proof}

\begin{proposition}\label{pp}Let $G$ be the $[m,4,n,4]$ lattice with $m,n$ satisfying (\ref{cmn1}) $m\ge3,\ n\geq 3$ and (\ref{cmn3}) $\frac{1}{m}+\frac{1}{n}<\frac{1}{2}$. Let $\omega\in \Omega$ be a $\Gamma$-invariant, $\Gamma_1$-ergodic constrained percolation on $G$. Let $s_0$ (resp.\ $s_1$) be the number of infinite 0-clusters (resp.\ 1-clusters) in $\omega$, and let $t_1$ (resp.\ $t_2$) be the number of infinite $\LL_1$-contours (resp.\ infinite $\LL_2$-contours) in $\omega$. Then almost surely $(s_0,s_1,t_1,t_2)\neq (0,0,0,0)$.
\end{proposition}

\begin{proof}The proof is inspired by Lemma 3.3 of \cite{BS20}. Let $\hat{G}=(\hat{V},\hat{E})$, $\hat{\phi}$ be defined as in the proof of \Cref{nooc}. Note that when $m,n$ satisfy (\ref{cmn1}) and (\ref{cmn3}), $\hat{G}$ is a quasi-transitive, non-amenable, planar and one-ended graph; and that the $[m,4,n,4]$ lattice is exactly the dual graph of $\hat{G}$. It is also known that quasi-transitive planar graphs with one end are unimodular; see \cite{LP}.

Define a \textbf{generalized contour} in a contour configuration $\hat{\phi}$ of $\hat{G}$ to be either a single vertex in $\hat{V}$ which has no incident present edges in $\hat{\phi}$, or a contour in $\hat{\phi}$. This way each vertex $v\in \hat{V}$ has a unique generalized contour in $\hat{\phi}$ passing through the vertex $v$.

Suppose that $(s_0,s_1,t_1,t_2)=(0,0,0,0)$ a.s. Then a.s.\ given a generalized contour $C$ of $\hat{\phi}$, there is a cluster $C'$ of $\omega$ surrounding it. Similarly, for every cluster $C$ in $\omega$, there is a  contour $C'$ in $\hat{\phi}$ that surrounds it. Let $\mathcal{C}_0$ denote the set of all generalized contours of $\hat{\phi}$. We set
\begin{eqnarray*}
\mathcal{C}_{j+1}:=\{C'':C\in\mathcal{C}_j\};
\end{eqnarray*}
in which $C$ is a generalized contour, $C'$ is a cluster, and $C''$ is a contour.
For $C\in \mathcal{C}_0$ and $v\in \hat{V}$, let $r(C):=\sup\{j:C\in \mathcal{C}_j\}$, and define $r(v):=r(C)$ if $C$ is the generalized contour of $v$ in $\hat{\phi}$. Intuitively, we may consider $r(C)$ as the maximal length of sequences of nesting contours, in which $C$ is the outermost contour. 

 Then there exist $i\in\{1,2\}$ and a sequence of finite contours $C_1,C_2,\ldots,C_n,\ldots$ in $\LL_i$, such that $C_{n+1}$ surrounds $C_n$, and
\begin{eqnarray*}
\lim_{n\rightarrow\infty}r(C_n)=\infty.
\end{eqnarray*}

For each $r$ let $\omega^r$ be the set of edges in $E(\LL_i)$ whose both endpoints  $u,v\in V(\LL_i)$ satisfy $r(v)\leq r$ and $r(u)\leq r$. Then $\omega^r$ is an invariant bond percolation and for any $v\in V(\LL_i)$,
\begin{eqnarray*}
\deg_{\LL_i}v=\mathbf{E}\lim_{r\rightarrow\infty}[\deg_{\omega^r}v]\leq \liminf_{r\rightarrow\infty}\mathbf{E}[\deg_{\omega^r}v]\leq\limsup_{r\rightarrow\infty}\mathbf{E}[\deg_{\omega^r}v] \deg_{\LL_i}v.
\end{eqnarray*}
Note that $\LL_i$ is a transitive, non-amenable graph. We have
\begin{eqnarray*}
\alpha(\LL_i)=\mathrm{deg}_{\LL_i}v-\imath_{E}(\LL_i)<\mathrm{deg}_{\LL_i}v
\end{eqnarray*}
where $\imath_E(\LL_i)$ is the edge isoperimetric constant defined as in (\ref{eic}), and $\alpha(\LL_i)$ is defined in \Cref{tbng}.
 By Lemma \ref{tbng}, the right hand side of (\ref{cic}) is strictly positive for sufficiently large $r$; we deduce that $\omega^r$ has infinite components with positive probability for all sufficiently large $r$. 

However, since $(s_0,s_1,t_1,t_2)=(0,0,0,0)$, by the arguments above each vertex in $\LL_i$ is surrounded by infinitely many finite contours in $\LL_i$. This implies that for any $r\in \mathbb{N}$, for any vertex $v\in V(\LL_i)$, there exists a finite contour $C$ surrounding $v$, such that $r(C)>r$, and therefore $C\cap \omega^r=\emptyset$. As a result, the components in $\omega^r$ including $v$ is finite. Then the proposition follows from the contradiction.
\end{proof}

\begin{lemma}\label{lm23}Let $G$ be the $[m,4,n,4]$ lattice with $m,n$ satisfying (\ref{cmn1}), (\ref{cmn3}). Let $(s_0,s_1,t_1,t_2)$ be given as in \Cref{m23}.
\begin{enumerate}
\item Let $\mu$ be a probability measure on $\Omega$ satisfying (A2)(A7).  Then
\begin{eqnarray*}
\mu((t_1,t_2)=(0,k))=0.
\end{eqnarray*} 
for any integer $1\leq k\leq \infty$.
\item Let $\mu$ be a probability measure on $\Omega$ satisfying (A2)(A6).  Then
\begin{eqnarray*}
\mu((t_1,t_2)=(k,0))=0.
\end{eqnarray*} 
for any integer $1\leq k\leq \infty$.
\end{enumerate}
\end{lemma}

\begin{proof}We prove Part I here; Part II can be proved using exactly the same technique.
 By (A2) $\mu$ is $\Gamma_i$ ergodic, either $\mu((t_1,t_2)=(0,k))=0$ or $\mu((t_1,t_2)=(0,k))=1$. Assume that $\mu((t_1,t_2)=(0,k))=1$; we shall obtain a contradiction. Since there exists an infinite $\LL_2$-contour; hence there exists an infinite cluster in $\lambda_2$ containing the infinite $\LL_2$-contour.  By (A7) $\lambda_2$ is $\Gamma_2$-ergodic, and the symmetry of $\lambda_2$, there exist an infinite 0-cluster and an infinite 1-cluster in $\lambda_2$ a.s.. Note that the configuration in $\lambda_2\in\{0,1\}^{V(\LL_2)}$ naturally induces a configuration $\omega\in \Omega$ by the condition that the contour configurations corresponding to $\lambda_2$ and $\omega$ are the same. We can see that if in $\lambda_2$ there exist both an infinite 0-cluster and an infinite 1-cluster, then in the induced constrained configuration $\omega\in\Omega$, there is both an infinite 0-cluster and an infinite 1-cluster. By \Cref{io}, there exist an infinite $\LL_1$-contour. But this is a contradiction to the fact that $t_1=0$. 
\end{proof}

\begin{lemma}\label{n0111}Let $G=(V,E)$ be the $[m,4,n,4]$ lattice with $m,n$ satisfying (\ref{cmn1}) $m\geq 3, n\geq 3$ and (\ref{cmn3}) $\frac{1}{m}+\frac{1}{n}<\frac{1}{2}$. Let $\mu$ be a probability measure on $\Omega$ satisfying (A2), (A8). Let $(s_0,s_1,t_1,t_2)$ be given as in \Cref{m23}. Then
\begin{eqnarray*}
\mu((s_0,s_1)=(1,1))=0.
\end{eqnarray*} 
\end{lemma}

\begin{proof}By (A2) $\mu$ is $\Gamma_i$-ergodic, either $\mu((s_0,s_1)=(1,1))=0$ or $\mu((s_0,s_1)=(1,1))=1$. Assume that $\mu((s_0,s_1)=(1,1))=1$; we shall obtain a contradiction. 

Let $\omega\in\Omega$.  We first construct a bond configuration $\omega_b\in \{0,1\}^E$ by letting an edge $e\in E$ to be present if and only if it joins two edges in $\omega$ with the same state; i.e.\ either both its endpoints have state 0; or both its endpoints have state 1. It is easy to check that the (0 or 1) clusters in $\omega$ are exactly the components in $\omega_b$. Then $\omega_b$ forms a $\Gamma_1$-invariant percolation on $G$. If $(s_0,s_1)=(1,1)$, then $\omega_b$ has exactly two infinite components. But this is a contradiction to \Cref{lbs}.
\end{proof}
 
\noindent\textbf{Proof of \Cref{m23} I.} Assume that $\mu$ is a probability measure on $\Omega$ satisfying (A2),(A3),(A7),(A8).

 Let $(s_0,s_1,t_1,t_2)$ be given as in the theorem. By Lemma \ref{lbs}, we have $\mu$-a.s.\ $s_0\in\{0,1,\infty\}$, $s_1\in\{0,1,\infty\}$ and $t_1\in\{0,1,\infty\}$.
By (A2) $\mu$ is $\Gamma_i$-ergodic and (A3) $\mu$ is symmetric with respect to interchanging state ``0'' and state ``1'', we have $\mu$-a.s.\ $(s_0,s_1)\in\{(0,0),(1,1),(\infty,\infty)\}$. Hence we need to rule out the case that $(s_0,s_1)=(1,1)$ and the case that $(s_0,s_1)=(0,0)$. Almost surely we have $(s_0,s_1)\neq (1,1)$ by \Cref{n0111}. Now we show that almost surely $(s_0,s_1)\neq (0,0)$.

We claim that $\mu$-a.s.\ $t_1\in \{0,\infty\}$. Assume that $\mu$-a.s.\ $t_1=1$, we shall obtain a contradiction. Let $\tau$ be the unique infinite $\LL_1$-contour. Then $\tau$ forms an invariant bond percolation on $\LL_1$ which has a single component a.s.. By Lemma \ref{l2}, $p_c(\tau)<1$ a.s. However, $\tau$ is an even-degree subgraph of $\LL_1$ and $\LL_1$ has vertex-degree 3; as a result, $\tau$ must be a doubly-infinite self-avoiding path. This is a contradiction to the fact that $p_c(\tau)<1$. Therefore we have either $\mu$-a.s.\ $t_1=0$ or $\mu$-a.s.\ $t_1=\infty$.

If $\mu$-a.s.\ $t_1=\infty$, let $\phi$ be the contour configuration on $\LL_1\cup\LL_2$ corresponding to the constrained percolation configuration. Since each infinite contour in $\phi$ is a doubly-infinite self-avoiding path, if there are infinitely many infinite contours,  then $\HH^2\setminus \phi$ has infinitely many unbounded components. Note also that there exists an infinite cluster in each infinite component of $\HH^2\setminus \phi$; hence $\mu$-a.s. $(s_0,s_1,t_1)=(\infty,\infty,\infty)$ in this case.

Now consider the case that $\mu$-a.s. $t_1=0$. 

We assume that $\mu$-a.s.\ $(s_0,s_1,t_1)=(0,0,0)$ and shall again obtain a contradiction. By \Cref{pp}, a.s.\ $(s_0,s_1,t_1,t_2)\neq(0,0,0,0)$. Moreover, it is impossible to have $(s_0,s_1,t_1,t_2)=(0,0,0,\infty)$ since if $t_2=\infty$, then there are infinitely many infinite clusters. By \Cref{lm23}, a.s.\ $(s_0,s_1,t_1,t_2)\neq(0,0,0,1)$. Therefore $\mu((s_0,s_1,t_1)=(0,0,0))=0$.

We next assume that $\mu$-a.s.\ $(s_0,s_1,t_1)=(\infty,\infty,0)$. By \Cref{lm23}, $\mu$-a.s.\ $(s_0,s_1,t_1,t_2)=(\infty,\infty,0,0)$. Since there exists an infinite 0-cluster and an infinite 1-cluster, by \Cref{io}, there exists an infinite contour, and $s_0+s_1>0$. The contradiction implies that $\mu((s_0,s_1,t_1)=(\infty,\infty,0))=0$. This completes the proof of Part I of \Cref{m23}. $\hfill\Box$

\bigskip

\noindent\textbf{Proof of \Cref{m23} II.} Assume that $\mu$ is a probability measure on $\Omega$ satisfying (A2),(A3),(A6),(A7),(A8). By \Cref{m23} I, $\mu$-a.s.\ $(s_0,s_1,t_1)=(\infty,\infty,\infty)$. Part II of \Cref{m23} then follows from \Cref{lm23}.
$\hfill\Box$

\bigskip

\section{Proof of \Cref{m21} }\label{p212}

In this section, we prove  \Cref{m21}.

We first prove that Parts (a) and (b) implies Part (c). If $\mu$ satisfies (A1)-(A7), then by (a) and (b), $\mu$-a.s. there are neither infinite primal contours nor infinite dual contours. Therefore $\mu$-a.s. there are no infinite contours.

Let $\mathcal{E}_0$ (resp.\ $\mathcal{E}_1$) be the event that there exists an infinite 0-cluster (resp.\ infinite 1-cluster). Assume that $\mu(\mathcal{E}_0\cup\mathcal{E}_1)>0$. Then by (A2) $\mu$ is $\Gamma_i$ ergodic,
\begin{eqnarray}
\mu(\mathcal{E}_0\cup\mathcal{E}_1)=1.\label{e01}
\end{eqnarray}
By (A3) $\mu$ is symmetric with respect to exchanging state ``0'' and state ``1'', $\mu(\mathcal{E}_0)=\mu(\mathcal{E}_1)$. By (A2), either $\mu(\mathcal{E}_0)=\mu(\mathcal{E}_1)=1$ or $\mu(\mathcal{E}_0)=\mu(\mathcal{E}_1)=0$. By (\ref{e01}), we have $\mu(\mathcal{E}_0)=\mu(\mathcal{E}_1)=1$. By \Cref{io}, $\mu$-a.s. there exists an infinite contour. But this is a contradiction to the fact that $\mu$-a.s. there are no infinite contours. Therefore $\mu$-a.s. there are no infinite clusters.

Next we prove (a) and (b). Note that the $[m,4,n,4]$ lattice $G$ is amenable if and only if 
\begin{eqnarray}
\frac{1}{m}+\frac{1}{n}=\frac{1}{2}.\label{am2}
\end{eqnarray}
 When $m,n$ are positive integers greater than or equal to 3, the only pairs of $(m,n)$ satisfying (\ref{am2}) are $(m,n)=(4,4)$, $(m,n)=(3,6)$ and $(m,n)=(6,3)$. When $(m,n)=(4,4)$, $G$ is the square grid embedded into $\RR^2$. In this case (a) and (b) were proved in \cite{HL16}. Then cases $(m,n)=(3,6)$ and $(m,n)=(6,3)$ can be proved in the same way. We write down the proof of the case when $(m,n)=(3,6)$ here.
 
 When $(m,n)=(3,6)$, $\LL_1$ is the hexagonal lattice $\HH=(V(\HH),E(\HH))$ and $\LL_2$ is the triangular lattice $\TT=(V(\TT),E(\TT))$.

 \begin{lemma}\label{l81}Assume that $(m,n)=(3,6)$. When $\mu$ satisfies (A1), (A2) and (A5), almost surely there exists at most one infinite contour in $\TT$.
\end{lemma}

\begin{proof}Let $\mathbb{N}$ be the set of all nonnegative integers. Let $\mathcal{N}$ be the number of infinite contours in $\TT$. By (A2), there exists $k_0\in\mathbb{N}\cup\{\infty\}$, s.t. $\mu(\mathcal{N}=k_0)=1$. 

By \cite{blps} (see also Exercise 7.24 of \cite{LP}), (A1) and the fact that the triangular lattice $\TT$ is transitive and amenable, $\mu$-a.s. no infinite contours has more than 2 ends.

The triangular lattice $\TT$ can be obtained from a square grid $\mathbb{S}$ by adding a diagonal in each square face of $\mathbb{S}$.

Let $B_n$ be an $n\times n$ box of $\mathbb{S}$. Let $\widetilde{B}_n$ be the corresponding box in $\TT$, i.e.\ $\widetilde{B}_n$ can be obtained from $B_n$ by adding a diagonal edge on each square face of $B_n$.

Let $\phi$ (resp.\ $\widetilde{\phi}$) be a contour configuration on $\mathbb{S}$ (resp.\ $\TT$), such that $\phi$ and $\widetilde{\phi}$ satisfy the following conditions (note that the vertices in $\partial B_n$ and $\partial\widetilde{B}_n$ are in 1-1 correspondence)
\begin{itemize}
\item for each vertex $v\in\partial B_n$, no edges incident to $v$ outside $\widetilde{B}_n$ are present in $\widetilde{\phi}$ if and only if no edges incident to $v$ outside $B_n$ are present in $\phi$;
\item for each vertex $v\in\partial B_n$, if there are incident present edges of $v$ in $\widetilde{\phi}$ outside $\widetilde{B}_n$, then the parity of the number of incident present edges of $v$ outside $\widetilde{B}_n$ in $\widetilde{\phi}$ is the same as the parity of the number of incident present edges of $v$ outside $B_n$ in $\phi$; i.e.\ either both numbers are even or both are odd.
\end{itemize}

Let $n\geq 2$. Given $\phi$, we can find a configuration $\xi$ in $B_n$, such that $[\phi\setminus B_n]\cup\xi$ is a contour configuration on $\mathbb{S}$ (i.e.\ each vertex of $\mathbb{S}$ has an even number of incident present edges in $[\phi\setminus B_n]\cup\xi$), and all the incident present edges of $\partial B_n$ outside $B_n$ are in the same contour; see Lemma 4.2 of \cite{HL16}. If $\widetilde{\phi}$ and $\phi$ satisfy the conditions described above, then $[\widetilde{\phi}\setminus \widetilde{B}_n]\cup\xi$ is a contour configuration on $\TT$, and all the incident present edges of $\partial\widetilde{B}_n$ outside $\widetilde{B}_n$ are in the same contour.

Note that $\xi$ can be obtained from $\widetilde{\phi}\cap \widetilde{B}_n$ by changing configurations on finitely many triangles in $\widetilde{B}$ as described in (A6). That is because any contour configuration on $\TT$ naturally induces two site configurations $\omega$, $1-\omega$, in $\{0,1\}^{V(\HH)}$, such that two adjacent vertices in $\HH$ have different states if and only if the edge in $\TT$ separating the two vertices are present in the contour configuration. Any two site configurations in $\{0,1\}^{V(\HH)}$ differ only in $\widetilde{B_n}$ can be obtained from each other by changing states on finitely many vertices in $V(\HH)\cap\widetilde{B}_n$. Changing the state at a vertex in $V(\HH)$ corresponds to changing the states on all the edges of the dual triangle face including the vertex in the contour configuration of $\TT$.

We claim that $k_0\in\{0,1\}$. Indeed, if $2\leq k_0<\infty$, we can find a box $\widetilde{B}_n$ in $\TT$, such that $\widetilde{B}_n$ intersects all the $k_0$ infinite contours. Then we can change configurations on finitely many triangles in $\widetilde{B}_n$, such that after the configuration change, there is exactly one infinite contour. By the finite energy assumption (A5), with positive probability, there exists exactly one infinite contour, but this is a contradiction to $\mu(\mathcal{N}=k_0)=1$, where $2\leq k_0<\infty$.

If $k_0=\infty$, we can find a box $\widetilde{B}_m$ in $\TT$, such that $\widetilde{B}_m$ intersects at least 3 infinite contours. Then we can change configurations on finitely many triangles in $\tilde{B}_n$, such that after the configuration change, all the infinite contours intersecting $\tilde{B}_m$ merge into one infinite contour, which has at least 3 ends. By (A5), with positive probability there exists an infinite contour with more than 2 ends. But this is a contradiction to the fact that almost surely no infinite contours have more than two ends. 
\end{proof}

\begin{lemma}\label{l82}Assume that $(m,n)=(3,6)$. When $\mu$ satisfies (A1) and (A4), almost surely there exists at most one infinite contour in $\HH$.
\end{lemma}
\begin{proof}Recall that a contour is a connected set of edges in which each vertex has an even number of incident present edges. Since the hexagonal lattice $\HH$ is a cubic graph, i.e.\ each vertex has 3 incident edges; each vertex in a contour of $\HH$ has 2 incident present edges. As a result, each contour in $\HH$ is either a self-avoiding cycle or a doubly-infinite self-avoiding path. In particular, each infinite contour in $\HH$ is a doubly-infinite self-avoiding path.

Each contour configuration in $\HH$ naturally induces two site configurations $\omega,1-\omega$ in $\{0,1\}^{V(\TT)}$, in which two adjacent vertices of $\TT$ have the same state if and only if the dual edge in $\HH$ separating the two vertices is absent in the contour configuration. The finite energy assumption (A4) implies the finite energy in the induced site configuration in $\{0,1\}^{V(\TT)}$; see \cite{bk89} for a definition. When $\mu$ satisfies (A1) (A4), by the result in \cite{bk89}, almost surely there exists at most one infinite 1-cluster and at most one infinite 0-cluster. In particular, there exist at most two infinite clusters. However, if in $\HH$ there are more than one infinite contour, then there are at least two doubly-infinite self-avoiding paths in $\HH$. As a result, the number of infinite clusters in the induced site configuration on $\TT$ is at least 3. The contradiction implies the lemma.
\end{proof}

\begin{lemma}\label{l83}Let $\omega\in\Omega$ be a constrained percolation configuration on the $[3,4,6,4]$ lattice $G$. Let $\psi=\phi(\omega)\in\Phi$ be the corresponding contour configuration in $E(\HH)\cup E(\TT)$. Assume that $\psi=\psi_1\cup\psi_2$, where $\psi_1$ (resp.\ $\psi_2$) is the contour configuration in $\HH$ (resp. $\TT$). If there is an infinite contour in $\psi_1$ (resp.\ $\psi_2$), then there is an infinite cluster in $\phi^{-1}(\psi_2)$ (resp.\ $\phi^{-1}(\psi_1)$).
\end{lemma}
\begin{proof}Assume that there is an infinite contour $C$ in $\HH$ (resp.\ $\TT$). Let $V_{C}$ be the set consisting of all the vertices of $G$ such that
\begin{itemize}
\item $v\in V_C$ if and only if $v$ is a vertex of a face of $G$ crossed by an edge present in the contour $C$.
\end{itemize}

Let $F$ be a square face of $G$ crossed by $C$; then all the vertices in $F$ are in the same cluster of $\phi^{-1}(\psi_2)$ (resp.\ $\phi^{-1}(\psi_1)$).  That is because $\psi_1\cap\psi_2=\emptyset$, if $F$ is crossed by $C\subseteq\psi_1$ (resp.\ $C\subseteq\psi_2$), then $F\cap \psi_2=\emptyset$ (resp. $F\cap\psi_1=\emptyset$).

Let $F'$ be a triangle (resp.\ hexagon) face of $G$ crossed by $C$; then all the vertices in $F$ are also in the same cluster of $\phi^{-1}(\psi_2)$ (resp.\ $\phi^{-1}(\psi_1)$). That is because the boundary edges of $F'$ cannot be crossed by edges of $\TT$ (resp.\ $\HH$) at all.

We claim that all the vertices in $V_C$ are in the same cluster in $\phi^{-1}(\psi_2)$ (resp.\ $\phi^{-1}(\psi_1)$). Indeed, for any two vertices $u,v\in V_C$, we can find a sequence of faces $F_0,F_1,\ldots,F_k$, such that
\begin{itemize}
\item $u\in F_0$ and $v\in F_k$; and
\item for $0\leq i\leq k$, $F_i$ is crossed by $C$;
\item for $1\leq j\leq k$, $F_j$ and $F_{j-1}$ share a vertex.
\end{itemize}
Then $u$ and $v$ are in the same cluster in $\phi^{-1}(\psi_2)$ (resp.\ $\phi^{-1}(\psi_1)$) since all the vertices in $\cup_{i=0}^{n}F_i$ are in the same cluster in $\phi^{-1}(\psi_2)$ (resp.\ $\phi^{-1}(\psi_1)$).
 Moreover, $|V_C|=\infty$ since $C$ is an infinite contour. Therefore, $\phi^{-1}(\psi_2)$ (resp.\ $\phi^{-1}(\psi_1)$) has an infinite cluster.
\end{proof}

Parts (a) and (b) can be proved using similar techniques; we write down the proof of (a) here.

Let $\mu$ be a probability measure on $\Omega$ satisfying (A1)-(A6). Assume that with strictly positive probability, there exist infinite contours in $\TT$. Then by (A2), $\mu$-a.s. there exist infinite contours in $\TT$. By \Cref{l81}, $\mu$-a.s.\ there exists exactly one infinite contour $C_1$ in $\TT$. By \Cref{l83}, $\mu$-a.s. there exist infinite clusters in $\phi^{-1}(\psi_1)$. Let $\mathcal{F}_0$ (resp.\ $\mathcal{F}_1$) be the event that there exists an infinite 0-cluster (resp.\ infinite 1-cluster) in $\phi^{-1}(\psi_1)$, then
\begin{eqnarray}
\lambda_1(\mathcal{F}_0\cup\mathcal{F}_1)=1\label{f01},
\end{eqnarray}
where the probability measure $\lambda_1$ is defined before (A6). By (A6), and the symmetry of $\lambda_1$, either $\lambda_1(\mathcal{F}_0)=\lambda_1(\mathcal{F}_1)=0$, or $\lambda_1(\mathcal{F}_0)=\lambda_1(\mathcal{F}_1)=1$. By (\ref{f01}), we have $\lambda_1(\mathcal{F}_0\cap\mathcal{F}_1)=1$. By \Cref{io}, $\mu$-a.s. there are infinite contours in $\HH$. By \Cref{l82}, $\mu$-a.s. there is exactly one infinite contour $C_2$ in $\HH$.

Hence there is exactly one infinite contour $C_2$ in $\HH$ and exactly one infinite contour $C_1$ in $\TT$. By \Cref{ct}, there exists
 an infinite cluster incident to both $C_1$ and $C_2$.

Let $\mathcal{H}_0$ (resp.\ $\mathcal{H}_1$) be the event that there exists an infinite 0-cluster (resp.\ infinite 1-cluster) in $\omega$ incident to both the infinite contour in $\HH$ and the infinite contour in $\TT$. Then
\begin{eqnarray}
\mu(\mathcal{H}_0\cup\mathcal{H}_1)=1.\label{eq:h01}
\end{eqnarray}
By (A3), $\mu(\mathcal{H}_0)=\mu(\mathcal{H}_1)$. By (A2), either $\mu(\mathcal{H}_0)=\mu(\mathcal{H}_1)=0$, or $\mu(\mathcal{H}_0)=\mu(\mathcal{H}_1)=1$. By (\ref{eq:h01}), $\mu(\mathcal{H}_0)=\mu(\mathcal{H}_1)=1$, therefore $\mu(\mathcal{H}_0\cap\mathcal{H}_1)=1$, i.e.\ there exist an infinite 0-cluster $\xi_0$ and an infinite 1-cluster $\xi_1$, such that $\xi_0$ is incident to both $C_1$ and $C_2$, and $\xi_1$ is incident to both $C_1$ and $C_2$. But this is a contradiction to \Cref{2c2c}. Therefore we conclude that $\mu$-a.s. there are no infinite contours in $\TT$; this completes the proof of Part (a).

\section{Proof of \Cref{m22}}\label{p211}

In this section, we prove \Cref{m22}.

\begin{lemma}\label{l71}Let $G$ be an $[m,4,m,4]$ lattice with $m\geq 5$. Let $\mu$ be a probability measure on $\Omega$ satisfying (A1)-(A3). Then the distribution of infinite clusters can only be one of the following 2 cases.
\begin{enumerate}
\item There are no infinite clusters $\mu$-a.s.
\item There are infinitely many infinite 1-clusters and infinitely many infinite 0-clusters $\mu$-a.s.
\end{enumerate}
\end{lemma}

\Cref{l71} can be obtained from \Cref{n0111}; it may also be proved using different arguments below.

\begin{proof}Let $s_0$ (resp.\ $s_1$) be the number of infinite 0-clusters (resp.\ 1-clusters). By \Cref{lbs} and (A2) (A3), there exist $k\in\{0,1,\infty\}$, such that $\mu((s_0,s_1)=(k,k))=1$. It suffices to show that $k\neq 1$. 

Let $\sA$ be the event that $(s_0,s_1)=(1,1)$. Assume that $\mu(\sA)=1$, we will obtain a contradiction.

As explained before the constrained site configurations on $G$ correspond to contour configurations in $\LL_1\cup\LL_2$.

Since $\mu$-a.s.\ there exists exactly one infinite 0-cluster and exactly one infinite 1-cluster simultaneously,  then by \Cref{cac},  $\mu$-a.s.\ there exists an infinite primal or dual contour incident to both the infinite 0-cluster and the infinite 1-cluster. Let $\sD_1$ (resp.\ $\sD_2$) be the event that there exists an infinite primal (resp.\ dual) contour in $\LL_1$ (resp.\ $\LL_2$), incident to both the infinite 0-cluster and the infinite 1-cluster. So we have 
\begin{eqnarray}
\mu(\sD_1\cup\sD_2)=1.\label{c1}
\end{eqnarray}
By (A1) $\mu$ is $\Gamma_i$ invariant, we have
\begin{eqnarray}
\mu(\sD_1)=\mu(\sD_2).\label{c2}
\end{eqnarray}
Moreover, by (A2), we have either
\begin{eqnarray}
\mu(\sD_1)=0\ \mathrm{or}\ \mu(\sD_1)=1.\label{c3}
\end{eqnarray}
Combining (\ref{c1}), (\ref{c2}) and (\ref{c3}), we have
\begin{eqnarray}
\mu(\sD_1\cap\sD_2)=1. \label{dcd}
\end{eqnarray}

Thus, by  (\ref{dcd}), we have exactly  one infinite 1-cluster on $G$, denoted by $\xi_1$ and exactly one infinite 0-cluster on $G$, denoted by $\xi_0$. There is an infinite primal contour incident to both $\xi_0$ and $\xi_1$, denoted by $C_1$; as well as an infinite dual contour incident to both $\xi_0$ and $\xi_1$, denoted by $C_2$. But this is impossible by \Cref{2c2c}. The contradiction implies the lemma.
\end{proof}

\begin{lemma}\label{l72}Let $\mu$ be a probability measure on $\Omega$ satisfying (A1)-(A3). Then the distribution of infinite contours can only be one of the following 2 cases.
\begin{enumerate}
\item There are neither infinite primal contours nor infinite dual contours $\mu$-a.s..
\item There are infinitely many infinite primal contours and infinitely many infinite dual contours $\mu$-a.s..
\end{enumerate}
\end{lemma}
\begin{proof}The primal (resp.\ dual) contours form an invariant bond percolation on $\LL_1$ (resp.\ $\LL_2$) under $\mu$. Let $t_1$ (resp.\ $t_2$) be the number of infinite primal (resp.\ dual) contours. By \Cref{lbs} and (A1)-(A3), only 3 cases may occur:
\begin{enumerate}[label=\roman*.]
\item $\mu$-a.s. $(t_1,t_2)=(0,0)$;
\item $\mu$-a.s. $(t_1,t_2)=(\infty,\infty)$;
\item $\mu$-a.s. $(t_1,t_2)=(1,1)$.
\end{enumerate} 

It remains to exclude in Case iii.. Assume that Case iii.\ occurs. Let $C_1$ (resp. $C_2$) be the unique infinite primal (resp.\ dual) contour. By \Cref{ct}, there exists an infinite cluster incident to both $C_1$ and $C_2$. Moreover, by (A2)-(A3), $\mu$-a.s.\ there exists an infinite 0-cluster incident to both $C_1$ and $C_2$, as well as an infinite 1-cluster incident to both $C_1$ and $C_2$. But this is impossible by \Cref{2c2c}.
\end{proof}

\bigskip
\noindent\textbf{Proof of \Cref{m22}.} By \Cref{l71,l72,pp}, it suffices to show that there exists an infinite cluster $\mu$-a.s.\ if and only if there exists an infinite contour $\mu$-a.s.\ if $\mu$ satisfies (A1)-(A3).

First assume that there exists an infinite cluster $\mu$-a.s. By (A2)-(A3), $\mu$-a.s.\ there exist both an infinite 0-cluster and an infinite 1-cluster. By \Cref{io}, $\mu$-a.s.\ there exists an infinite contour.

Now assume that there exists an infinite contour $\mu$-a.s. By (A1)-(A2), $\mu$-a.s.\ there exist both an infinite primal contour and an infinite dual contour. By \Cref{cc}, $\mu$-a.s. there exists an infinite cluster. 
$\hfill\Box$

\section{Percolation on transitive, triangular tilings of hyperbolic plane}\label{peh}

In this section, we discuss the applications of the techniques developed in the proof of \Cref{m23}
 to prove results concerning unconstrained site percolation on vertex-transitive, triangular tilings of the hyperbolic plane. We also discuss results about Bernoulli percolation on such graphs. These results will be used to prove \Cref{ipl,coii,xorc}.

\begin{lemma}\label{nzz}\label{lnz}Let $\LL_2$ a vertex-transitive, regular tiling of the hyperbolic plane with triangles, such that each vertex has degree $n\geq 7$. Consider an $\mathrm{Aut}(\LL_2)$-invariant site percolation $\omega$ on $\LL_2$ with distribution $\mu$. Assume that $\mu$ is $\mathrm{Aut}(\LL_2)$-ergodic. Let $s_0$ (resp.\ $s_1$) be the number of infinite 0-clusters (resp.\ infinite 1-clusters) in the percolation. Then
\begin{eqnarray*}
\mu((s_0,s_1)=(0,0))=0.
\end{eqnarray*}
\end{lemma}
\begin{proof}Since the event $\{(s_0,s_1)=(0,0)\}$ is $\mathrm{Aut}(\LL_2)$-invariant, and $\mu$ is $\mathrm{Aut}(\LL_2)$-ergodic, either $\mu((s_0,s_1)=(0,0))=0$ or  $\mu((s_0,s_1)=(0,0))=1$. Assume that $\mu((s_0,s_1)=(0,0))=1$; we shall obtain a contradiction.

Note that the dual graph $\LL_1$ of $\LL_2$ is a vertex-transitive, non-amenable, planar graph in which each vertex has degree 3. A contour configuration $\phi(\omega)\subset E(\LL_1)$ is a subset of edges of $\LL_1$ in which each present edge has a dual edge in $E(\LL_2)$ joining exactly one vertex with state 0 and one vertex with state 1 in $\omega$. As usual, a contour is a maximal connected component of present edges in a contour configuration. Each contour configuration, by definition, must be an even-degree subgraph of $\LL_1$. Given that $\LL_1$ has vertex-degree 3, each vertex in $\LL_1$ is incident to zero or two present edges in a contour configuration. Let $t$ be the number of infinite contours. Each infinite contour on $\LL_1$ must be a doubly infinite self-avoiding path. 

If $t\geq 1$, let $C_{\infty}$ be an infinite contour. Then $\HH^2\setminus C_{\infty}$ has exactly two unbounded components, since $C_{\infty}$ is a doubly-infinite self-avoiding path. Let $V_{\infty}$ be the set of all vertices of $\LL_2$ lying on a face crossed by $C_{\infty}$. Given $C_{\infty}$ a fixed orientation. Let $V_{\infty}^{+}$ (resp.\ $V_{\infty}^{-}$) be the subset of $V_{\infty}$ consisting of all the vertices in $V_{\infty}$ on the left hand side (resp.\ right hand side) of $C_{\infty}$ when traversing $C_{\infty}$ along the given orientation. Then exactly one of $V_{\infty}^{+}$ and $V_{\infty}^{-}$ is part of an infinite 1-cluster, and the other is part of an infinite 0-cluster. Therefore we have $s_0\geq1$ and $s_1\geq 1$ if $t\geq 1$.

The rest of the proof is an adaptation of the proof of \Cref{pp} to different graphs.
Define a \textbf{generalized contour} in a contour configuration $\phi\subset E(\LL_1)$ to be either a single vertex in $V(\LL_1)$ which has no incident present edges in $\phi$, or a contour in $\phi$. This way for each vertex $v\in V(\LL_1)$, and each contour configuration $\phi\subset E(\LL_1)$, there is a unique generalized contour passing through $v$.
By the arguments in the last paragraph, if $(s_0,s_1)=(0,0)$, then $t=0$. 

Now consider the case when $(s_0, s_1, t)=(0,0,0)$. Given a cluster $C$ in $\omega$, there is a unique contour $C'$ of $\phi(\omega)$ surrounding $\xi$. Similarly, for every generalized contour $C'$, there is a cluster $C''$ that surrounds $C'$. Let $\mathcal{C}_0$ denote the set of all generalized contours of $\omega$. We set
\begin{eqnarray*}
\mathcal{C}_{j+1}:=\{C'':K\in\mathcal{C}_j\}.
\end{eqnarray*}
For $C\in \mathcal{C}_0$ and $v\in V(\LL_1)$, let $r(C):=\sup\{j:C\in \mathcal{C}_j\}$, and define $r(v):=r(C)$ if $C$ is the generalized contour of $v$ in $\phi(\omega)$. For each $r$ let $\omega^r$ be the set of edges in $E(\LL_1)$ whose both end-vertices $u,v\in V(\LL_1)$ satisfy $r(v)\leq r$ and $r(u)\leq r$. Then $\omega^r$ is an invariant bond percolation and
\begin{eqnarray*}
\lim_{r\rightarrow\infty}\mathbf{E}[\deg_{\omega^r}v]=3.
\end{eqnarray*}
By Lemma \ref{tbng}, we deduce that $\omega^r$ has infinite components with positive probability for all sufficiently large $r$. 
 
However, since $(s_0,s_1,t)=(0,0,0)$, by the arguments above each vertex in $v\in V(\LL_1)$ is surrounded by infinitely many contours. This implies that for any $r\in \mathbb{N}$, for any vertex $v\in V(\LL_1)$, there exists a contour $C$ surrounding $v$, such that $r(C)>r$, and therefore $C\cap \omega^r=\emptyset$. As a result, the components in $\omega^r$ including $v$ is finite. The contradiction implies the lemma.
\end{proof}

\begin{lemma}\label{t01}Let $\LL_2$ be the regular tiling of the hyperbolic plane with triangles, such that each vertex has degree $n\geq 7$. Consider an $\mathrm{Aut}(\LL_2)$-invariant site percolation $\omega$ on $\LL_2$ with distribution $\mu$. Let $t$ be the number of infinite contours. Then $\mu$-a.s.\ $t\in\{0,\infty\}$
\end{lemma}
\begin{proof}By \Cref{lbs}, $\mu$-a.s.\ $t\in \{0,1,\infty\}$. Let $\sA$ be the event that $t=1$. Assume $\mu(\sA)>0$, we shall obtain a contradiction. Conditional on the event $\sA$, let $\tau$ be the unique infinite contour. Since $\tau$ is a infinite, connected, even-degree subgraph of $\LL_1$, and each vertex of $\LL_1$ has degree 3, $\tau$ must be a doubly infinite self-avoiding path. Then $p_c(\tau)=1$. But this is a contradiction to \Cref{l2}. Then the lemma follows.
\end{proof}

\begin{lemma}\label{ozz}Let $\LL_2$ be the regular tiling of the hyperbolic plane with triangles, such that each vertex has degree $n\geq 7$. Consider an $\mathrm{Aut}(\LL_2)$-invariant site percolation $\omega$ on $\LL_2$ with distribution $\mu$. Assume that $\mu$ is $\mathrm{Aut}(\LL_2)$-ergodic. Let $s_0$ (resp.\ $s_1$) be the number of infinite 0-clusters (resp.\ infinite 1-clusters) in the percolation. Then
\begin{eqnarray*}
\mu((s_0,s_1)=(1,1))=0.
\end{eqnarray*}
\end{lemma}
\begin{proof}We may construct a $[3,4,n,4]$ lattice embedded into the hyperbolic plane $\HH^2$, such that the $[3,4,n,4]$ lattice, $\LL_1$ and $\LL_2$ satisfy the conditions as described in \Cref{xorh}. Then each percolation configuration $\omega$ in $\{0,1\}^{V(\LL_2)}$ induces a constrained configuration $\widetilde{\omega}\in \Omega$ by the condition that $\omega$ and $\widetilde{\omega}$ has the same contour configuration; and that $\omega(v)=1$ for $v\in V(\LL_2)$ if and only if all the vertices of the $[3,4,n,4]$ lattice in the degree-$n$ face of $\LL_1$ containing $v$ have state 1 in $\widetilde{\omega}$.
Then the lemma follows from \Cref{n0111}.
\end{proof}

\begin{lemma}\label{nzi}\label{lnz}Let $\LL_2$ be the regular tiling of the hyperbolic plane with triangles, such that each vertex has degree $n\geq 7$. Consider a site percolation $\omega$ on $\LL_2$ with distribution $\mu$. Let $s_0$ (resp.\ $s_1$) be the number of infinite 0-clusters (resp.\ infinite 1-clusters) in the percolation. Then
it is not possible that 
\begin{eqnarray*}
(s_0,s_1)=(0,k);\ \mathrm{for}\ 2\leq k\leq \infty.
\end{eqnarray*}
\end{lemma}
\begin{proof}Assume that $s_0=0$, and that there exist at least two distinct infinite 1-clusters $C_1$ and $C_2$. Let $\ell$ be a path consisting of edges of $\LL_2$ joining a vertex $x\in C_1$ and a vertex $y\in C_2$. If $\ell$ does not cross infinite contours, then we can find another path $\ell'$ joining $x$ and $y$ such that $\ell'$ does not cross contours at all. Then $C_1=C_2$. The contradiction implies that there exists at least one infinite contour in $\LL_1$. Since each infinite contour in $\LL_1$ is a doubly infinite self-avoiding path, if there exists an infinite contour, then there exist at least one infinite 0-cluster and at least one infinite 1-cluster.  But this is a contradiction to the fact that $s_0=0$.
\end{proof}

\begin{definition}Let $G=(V,E)$ be a graph. Given a set $A\in 2^V$, and a vertex $v\in V$, denote $\Pi_{v}A=A\cup\{v\}$. For $\sA\subset 2^V$, we write $\Pi_v\sA=\{\Pi_v A: A\in \mathcal{A}\}$. A site percolation process $(\mathbf{P},\omega)$ on $G$ is \textbf{insertion-tolerant} if $\mathbf{P}(\Pi_v\sA)>0$ for every $v\in V$ and every event $\sA\subset 2^V$ satisfying $\mathbf{P}(\sA)>0$.
\end{definition}

We can similarly define an insertion tolerant bond percolation by replacing a vertex with an edge in the above definition. A bond percolation is \textbf{deletion tolerant} if $\mathbf{P}[\Pi_{\neg e}\sA]>0$ whenever $e\in E(G)$ and $\mathbf{P}(\sA)>0$, where $\Pi_{\neg e}\omega=\omega\setminus\{e\}$.

\begin{lemma}\label{icie}Let $G$ be a graph with a transitive, unimodular, closed automorphism group $\Gamma\subset \mathrm{Aut}(G)$. If $(\mathbf{P},\omega)$ is a $\Gamma$-invariant, insertion-tolerant percolation process on $G$ with  infinitely many infinite clusters a.s., then a.s. every infinite cluster has infinitely many ends.
\end{lemma}
\begin{proof}See Proposition 3.10 of \cite{LS99}; see also \cite{HP} and \cite{HPS}.
\end{proof}

\begin{lemma}\label{oiz}\label{lnz}Let $\LL_2$ be the regular tiling of the hyperbolic plane with triangles, such that each vertex has degree $n\geq 7$. Consider an $\mathrm{Aut}(\LL_2)$-invariant, insertion-tolerant site percolation $\omega$ on $\LL_2$ with distribution $\mu$. Assume that $\mu$ is $\mathrm{Aut}(\LL_2)$-ergodic. Let $s_0$ (resp.\ $s_1$) be the number of infinite 0-clusters (resp.\ infinite 1-clusters) in the percolation. Then
\begin{eqnarray*}
\mu((s_0,s_1)=(1,\infty))=0.
\end{eqnarray*}
\end{lemma}
\begin{proof}Assume that $\mu(s_0,s_1)=(1,\infty)=1$; we shall obtain a contradiction. By \Cref{icie}, a.s. every infinite 1-cluster has infinitely many ends. Then the lemma follows from \Cref{l147}.
\end{proof}

\begin{proposition}\label{p118}Let $\LL_2$ be the regular tiling of the hyperbolic plane with triangles, such that each vertex has degree $n\geq 7$. Consider an $\mathrm{Aut}(\LL_2)$-invariant, $\mathrm{Aut}(\LL_2)$-ergodic, insertion-tolerant site percolation $\omega$ on $\LL_2$ with distribution $\mu$. Let $s_0,s_1,t$ be given as in \Cref{nzz,t01}, then
\begin{eqnarray*}
(s_0,s_1,t)\in\{(0,1,0),(1,0,0),(\infty,\infty,\infty)\}\ a.s.
\end{eqnarray*}
\end{proposition}
\begin{proof} By \Cref{lbs}, we have $\mu$-a.s.\ $s_0,s_1,t\in\{0,1,\infty\}$. By \Cref{nzz,ozz,nzi,oiz}, we have $\mu$-a.s.\ $(s_0,s_1)\in\{(0,1),(1,0),(\infty,\infty)\}$. By \Cref{t01}, $\mu$-a.s.\ $t\in\{0,\infty\}$. Moreover, since each infinite contour must be a doubly infinite self-avoiding path, if $t=\infty$, then $s_0+s_1=\infty$. This implies that if $s_0+s_1=1$, then $t=0$, a.s. Moreover, if $s_0+s_1\geq 2$, then $t\geq1$. Then the proposition follows.
\end{proof}

\begin{example}Consider Example \ref{exl2}. We have 
\begin{eqnarray*}
\frac{1}{n-1}\leq p_c<\frac{1}{2}<p_u=1-p_c\leq \frac{n}{n-1}
\end{eqnarray*}
 by Theorems 1.1, 1.2. and 1.3 of \cite{BS20}. 
By \Cref{p118}, we have
\begin{itemize}
\item if $p\in [0,p_c]$, a.s. $(s_0,s_1,t)=(1,0,0)$;
\item if $p\in (p_c,p_u)$, a.s. $(s_0,s_1,t)=(\infty,\infty,\infty)$;
\item if $p\in[p_u,1]$, a.s. $(s_0,s_1,t)=(0,1,0)$.
\end{itemize}
\end{example}

\section{Proof of \Cref{ipl} and \Cref{c34}}\label{pipl}

In this section, we prove \Cref{ipl} and \Cref{c34}. The proof of \Cref{ipl} I. is based on a stochastic domination between the i.i.d. Bernoulli site percolation and the Ising model on the same graph, which correspond to random cluster models with $q=1$ and $q=2$, respectively. The proof of of \Cref{ipl} II(a) is based on the ergodicity and symmetry of the measure $\mu^f$, as well as \Cref{p118}, which lists all the possible numbers of infinite 0-clusters and infinite-1 clusters which have strictly positive probability to occur. \Cref{ipl} II (b)(c)(d) then follow from the fact that any one of the conditions (b)(c) and (d) implies (a). We then prove \Cref{c34} by an inequality between $p_{c,2}^w$ and $p_{u,2}^f$, and \Cref{ipl} II(c). \Cref{ipl} III is proved by the well-known combinatorial correspondence between the Ising model and its random-cluster representation. \Cref{ipl} IV then follows from the decomposition of $\mu^f$ as a convex combination of $\mu^+$ and $\mu^-$ and an inequality between $p_{u,2}^f$ and $p_{u,2}^w$.

  We shall first review the stochastic domination we use to prove these results.

\begin{definition}(Stochastic Domination) Let $G=(V,E)$ be a graph. Let $\Omega=\{0,1\}^E$ (resp.\ $\Omega=\{0,1\}^V$). Then the configuration space $\Omega$ is a partially ordered set with partial order given by $\omega_1\leq \omega_2$ if $\omega_1(e)\leq \omega_2(e)$ for all $e\in E$ (resp.\ $\omega_1(v)\leq \omega_2(v)$ for all $v\in V$). A random variable $X:\Omega\rightarrow\RR$ is  called increasing if $X(\omega_1)\leq X(\omega_2)$ whenever $\omega_1\leq \omega_2$. An  event $A\subset\Omega$ is called increasing (respectively, decreasing) if its indicator function $1_A$ is increasing  (respectively, decreasing). Given two probability measures $\mu_1$, $\mu_2$ on $\Omega$, we write $\mu_1\prec \mu_2$, and we say that $\mu_2$ \textbf{stochastically dominates} $\mu_1$, if $\mu_1(A)\leq \mu_2(A)$ for all increasing events $A\subset \Omega$.
\end{definition}

\begin{lemma}(Holley inequality)\label{hln} Let $G=(V,E)$ be a finite graph. Let $\Omega=\{0,1\}^E$ (resp.\ $\Omega=\{0,1\}^V$). Let $\mu_1$ and $\mu_2$ be strictly positive probability measures on $\Omega$ such that
\begin{eqnarray}
\mu_2(\max\{\omega_1,\omega_2\})\mu_1(\min(\omega_1,\omega_2))\geq \mu_1(\omega_1)\mu_2(\omega_2),\qquad \omega_1,\omega_2\in\Omega,\label{hli}
\end{eqnarray}
Then 
\begin{eqnarray*}
\mu_1\prec \mu_2.
\end{eqnarray*}
\end{lemma}
\begin{proof}See Theorem (2.1) of \cite{GrGrc}; see also \cite{HR74}.
\end{proof}

\begin{lemma}\label{shln}Let $G=(V,E)$ be a finite graph. Let $\Omega=\{0,1\}^E$. Let $\mu_1$ and $\mu_2$ be strictly positive probability measures on $\Omega$ such that
\begin{eqnarray}
\mu_2(\omega\cup\{e\})\mu_1(\omega\setminus\{e\})\geq \mu_1(\omega\cup\{e\})\mu_2(\omega\setminus\{e\}),\qquad\omega\in \Omega, e\in E.\label{sc1}
\end{eqnarray}
Here $\omega$ is interpreted as the subset of $E$ consisting of all the edges with state 1. If either $\mu_1$ or $\mu_2$ satisfies
\begin{eqnarray}
\mu(\omega\cup\{e,f\})\mu(\omega\setminus\{e,f\})\geq \mu(\omega\cup{e}\setminus \{f\})\mu(\omega\cup\{f\}\setminus\{e\}),\label{sc2}
\end{eqnarray}
then (\ref{hli}) holds.
\end{lemma}
\begin{proof}See Theorem 2.6 of \cite{GrGrc}.
\end{proof}

For a planar graph $G$, let $G_*$ be its planar dual graph. The following lemmas concerning planar duality, are proved in \cite{BS20,HJL02}.
 
 \begin{lemma}\label{pd1}Let $G$ be a planar nonamenable quasi-transitive graph, and let $p,p_*\in (0,1)$ satisfy
 \begin{eqnarray}
 p_*=\frac{(1-p)q}{p+(1-p)q}\label{pdp}
 \end{eqnarray}
In the natural coupling of $FRC_{p,q}^G$ and $WRC_{p_*,q}^{G_*}$ as dual measures (i.e.\ a dual edge is present if and only if the corresponding primal edge is absent), the number of infinite clusters with respect to each is a.s.\ one of the following: $(0,1)$, $(1,0)$, $(\infty,\infty)$.
 \end{lemma}
 
\begin{proof}See Proposition 3.5 of \cite{HJL02}, which is proved using the same technique as the proof of Theorem 3.7 of \cite{BS20}.
\end{proof}

\begin{lemma}\label{pd2}Let
\begin{eqnarray*}
h(x):=\frac{x}{1-x}.
\end{eqnarray*}
For any planar non-amenable quasi-transitive graph $G$,
\begin{eqnarray*}
&&h(p_{c,q}^w(G))h(p_{u,q}^f(G_*))=h(p_{u,q}^w(G_*))h(p_{c,q}^f(G))=1,\\
&&0<p_{c,q}^w(G)\leq p_{c,q}^f(G)<1,\ \mathrm{and}\ 0<p_{u,q}^w(G)\leq p_{u,q}^f(G)<1
\end{eqnarray*}
\end{lemma}

\begin{proof}See Corollary 3.6 of \cite{HJL02}.
\end{proof}

We start the proof of \Cref{ipl} with the following lemma.

\begin{lemma}\label{sz}Let $\LL_2$ be a regular tiling of the hyperbolic plane such that each vertex has degree $n$ and each face has degree $m$. Let $q\geq 1$.
Assume that $WRC^{\LL_2}_{p,q}$-a.s. there is a unique infinite open cluster for the random cluster model on $\LL_2$. Let $\tau$ be the unique infinite open cluster in the random cluster configuration $\omega$. We define a site percolation configuration $\xi$ on $V(\LL_2)$, by letting all the vertices in $\tau$ have state 1, and all the other vertices have state 0. Then a.s. $\xi$ has no infinite 0-cluster.
\end{lemma}

\begin{proof}Let $\sA_0$ be the event that $\xi$ has an infinite 0-cluster. By ergodicity of $WRC_{p,q}^{\LL_2}$, either $WRC_{p,q}^{\LL_2}(\sA_0)=0$ or $WRC_{p,q}^{\LL_2}(\sA_0)=1$. Assume that $WRC_{p,q}^{\LL_2}(\sA_0)=1$; we shall obtain a contradiction. 

The dual configuration to the random cluster configuration on $\LL_2$ is a bond configuration on $\LL_1$ such that an edge in $\LL_1$ is present if and only if its dual edge in $\LL_2$ is absent in the random cluster configuration of $\LL_2$.  Note also that the free boundary condition is dual to the wired boundary condition by the relation between dual configurations described above. Moreover, if the random cluster configuration on $\LL_2$ has distribution $WRC_{p,q}^{\LL_2}$, then its dual configuration on $\LL_1$ has distribution $FRC_{p_*,q}^{\LL_1}$; where $p,p_*$ satisfy (\ref{pdp}).

 Let $k$ be the number of infinite clusters in the bond configuration in $\LL_2$, and let $k^{*}$ be the number of infinite clusters in the corresponding dual configuration in $\LL_1$. 
 By \Cref{pd1}, a.s.\ $(k,k^*)\in \{(0,1),(1,0),(\infty,\infty)\}$. (Indeed this is true for any $\mathrm{Aut}(\LL_2)$-invariant, insertion tolerant and deletion tolerant bond configuration.) Hence if $WRC_{p,q}^{\LL_2}$-a.s. there is a unique infinite open cluster, then $FRC_{p_*,q}^{\LL_1}$-a.s. there is no infinite cluster in the corresponding dual configuration.

Since $\tau$ is an infinite cluster, there exists an infinite 1-cluster in $\xi$ by construction. If there exists an infinite 0-cluster in $\xi$ as well, by \Cref{io}, there exists an infinite contour $C$ consisting of edges of $\LL_1$ in which each edge has a dual edge joining a vertex of $V(\LL_2)$ with state 1 in $\xi$ and a vertex of $V(\LL_2)$ with state 0 in $\xi$. Moreover, all the edges in $C$ must be present in the dual configuration of $\omega$, since every edge in $C$ is dual to an edge of $V(\LL_2)$ not open in $\omega$.
Then we have $k=1$ and $k^*\geq 1$. But this is a contradiction to \Cref{pd1}. Hence a.s.\ there are no infinite 0-clusters in $\xi$.
\end{proof}

\begin{lemma}\label{1l2}Let $\LL_2$ be a regular tiling of the hyperbolic plane such that each vertex has degree $n$ and each face has degree $m$.
Then for the graph $\LL_2$,
\begin{eqnarray*}
p_{u,1} \leq p_{u,2}^{w}\leq p_{u,2}^{f}
\end{eqnarray*}
\end{lemma}

\begin{proof}The fact that $p_{u,2}^w\leq p_{u,2}^f$ follows from \Cref{pd2}. 

Now we prove that $p_{u,1}\leq p_{u,2}^w$.  Note that the following stochastic monotonicity result holds:
\begin{eqnarray}
WRC_{p,2}\prec WRC_{p,1}=FRC_{p,1},\label{st21}
\end{eqnarray}
by (4.1) of \cite{RS01}.  

Let $\mathcal{F}_1$ be the event that there exists a unique infinite cluster in the random cluster configuration in $\LL_2$.
By ergodicity of $WRC_{p,2}$ and $WRC_{p,1}$ and the monotonicity of $WRC_{p,2}(\mathcal{F}_1)$ and $WRC_{p,1}(\mathcal{F}_1)$ with respect to $p$, to show that $p_{u,1}\leq p_{u,2}^{w}$, it suffices to show that whenever $WRC_{p,2}(\mathcal{F}_1)=1$, then $WRC_{p,1}(\mathcal{F}_1)=1$.

Let $\mathcal{F}_{1,0}\subset \mathcal{F}_1$ be the event consisting of all the configurations in which both of the following two cases occur
\begin{itemize}
\item there exists a unique infinite cluster $\tau$ in the random cluster configuration on $\LL_2$; and
\item let $\xi\in \{0,1\}^{V(\LL_2)}$ be the site configuration obtained by assigning the state 1 to all the vertices in $\tau$, and the state 0 to all the vertices not in $\tau$; then there exists no infinite 0-cluster in $\xi$.
\end{itemize}

By \Cref{sz}, if $WRC_{p,2}(\mathcal{F}_1)=1$, then $WRC_{p,2}(\mathcal{F}_{1,0})=1$. Since $\mathcal{F}_{1,0}$ is an increasing event, by (\ref{st21}) we have $WRC_{p,2}(\mathcal{F}_{1,0})=1$, then $WRC_{p,1}(\mathcal{F}_{1,0})=1$. Since $\mathcal{F}_{1,0}\subset \mathcal{F}_1$, we have $WRC_{p,1}(\mathcal{F}_1)=1$. This completes the proof.
 \end{proof}

\begin{lemma}\label{labc} Let $\LL_2$ be a vertex-transitive, triangular tiling of the hyperbolic plane such that each vertex has degree $n\geq 7$. Consider the following Conditions (a),(b),(c) and (d) in Part II of \Cref{ipl}. We have
\begin{eqnarray*}
(d)\Rightarrow(c)\Rightarrow(b)\Rightarrow (a);
\end{eqnarray*}
 where $A\Rightarrow B$ means that if $A$ holds, then $B$ holds.
\end{lemma}

\begin{proof}The statement $(b)\Rightarrow (a)$ follows from Theorem 4.1 of \cite{RS01}.

The fact that $(c)\Rightarrow (a)$ follows from Theorem 3.2 (v) of \cite{HJL02}; while the fact that $(c)\Rightarrow (b)$ follows from Theorem 4.1 and Lemma 6.4 of \cite{LS99}.

The fact that $(d)\Rightarrow (c)$ follows from \Cref{1l2}.
\end{proof}

\subsection{Proof of Part I of \Cref{ipl}.}

First note that if (\ref{pch}) hold, then 
\begin{eqnarray}
\frac{e^{h}}{e^{h}+e^{-h}}=p_u.\label{puh}
\end{eqnarray}
by the fact that $p_c+p_u=1$.

Let $\nu_1$ (resp.\ $\nu_2$) be the probability measure for the i.i.d.\ Bernoulli site percolation on $\LL_2$ in which each vertex takes the value ``$+$'' with probability $p_1$ (resp.\ $p_2$) satisfying 
\begin{eqnarray*}
&&\frac{e^{nJ}}{e^{nJ}+e^{-nJ}}<p_1<p_u\\
&& p_c<p_2<\frac{e^{-nJ}}{e^{nJ}+e^{-nJ}}
\end{eqnarray*}
and the value ``$-$'' with probability $1-p_1$ (resp.\ $1-p_2$). Such $p_1$ and $p_2$ exist by (\ref{jh}).

Fix a triangle face $F_0$ of $\LL_2$. Let $B_R=(V(B_R),E(B_R))$ be the finite subgraph of $\LL_2$ consisting of all the faces of $\LL_2$ whose graph distance to $F_0$ is at most $R$. Let $\nu_{1,R}$ (resp.\ $\nu_{2,R}$) be the restriction of $\nu_1$ (resp.\ $\nu_2$) on $B_R$. Let $\mu_{R}^{+}$ (resp. $\mu_{R}^{-})$ be the probability measure for the Ising model on $B_R$ with respect to the coupling constant $J$ and the ``+'' boundary condition (resp.\  the ``$-$'' boundary condition). Let $\omega_1$, $\omega_2$ be two configurations in $\{-1,1\}^{V(B_R)}$.  Then by \Cref{hln,shln}, we can check the F.K.G. lattice conditions below
\begin{eqnarray*}
\nu_{1,R}(\max\{\omega_1,\omega_2\})\mu_{R}^+(\min\{\omega_1,\omega_2\})\geq \nu_{1,R}(\omega_1)\mu_R^+(\omega_2)\\
\mu_{R}^-(\max\{\omega_1,\omega_2\})\nu_{2,R}(\min\{\omega_1,\omega_2\})\geq \mu_R^-(\omega_1)\nu_{2,R}(\omega_1).
\end{eqnarray*}
Then we obtain the following stochastic domination result:
\begin{eqnarray*}
\nu_{2,R}\prec\mu_R^{-}\prec \mu_R^+\prec\nu_{1,R}.
\end{eqnarray*}
Letting $R\rightarrow\infty$, for any $\mathrm{Aut}(\LL_2)$-invariant Gibbs measure $\mu$ for the Ising model on $\LL_2$ with coupling constant $J$, we have
\begin{eqnarray*}
\nu_2\prec \mu^{-}\prec \mu\prec \mu^+\prec\nu_{1}.
\end{eqnarray*}
Since $\nu_2$-a.s. there are infinite ``$+$''-clusters, $\mu$-a.s. there are infinite ``$+$''-clusters. Similarly,  $\mu$-a.s. there are infinite ``$-$''-clusters, since $\nu_1$-a.s. there are infinite ``$-$''-clusters. By \Cref{p118}, we conclude that when (\ref{jh}) hold, $\mu$-a.s. there are infinitely many infinite ``$+$''-clusters, infinitely many infinite ``$-$''-clusters and infinitely many infinite contours. This completes the proof of Part I.

\subsection{Proof of Part II of \Cref{ipl}.}

 We first assume that $\mu^f$ is $\mathrm{Aut}(\LL_2)$-ergodic.
Since $\mu^f$ is also symmetric in switching ``$+$'' and ``$-$'' states,  $\mu^f$-a.s.\ the number of infinite ``$+$''-clusters and the number of infinite ``$-$''-clusters are equal. Then the conclusion follows from \Cref{p118}, which implies that if the number of infinite ``$+$''-clusters and the number of infinite ``$-$''-clusters are equal a.s., then both numbers are infinite.  

By \Cref{labc}, if one of the conditions (a), (b), (c) and (d) in \Cref{labc} holds, then $\mu^f$ almost surely there are infinitely many infinite ``$+$''-clusters and infinitely many infinite ``$-$''-clusters. 

\subsection{Proof of \Cref{c34}.}
 By Proposition 3.2 (i) of \cite{HJL02}, if (\ref{jj3}) holds, then there is a unique infinite-volume Gibbs measure for the Ising model on $\LL_2$ with coupling constant $J$. Since
\begin{eqnarray*}
p_{c,2}^{w}\leq p_{c,2}^{f}\leq p_{u,2}^{f},
\end{eqnarray*}
(\ref{jj3}) implies Condition (c). Then \Cref{ipl} II(c) implies $\mu^f$ a.s. there are neither infinite ``$+$''-clusters nor infinite ``$-$''-cluster.
The corollary now follows from the uniqueness of the infinite-volume Gibbs measure.

\subsection{Proof of Part III of \Cref{ipl}.}

Since $p=1-e^{-2J}$, when (\ref{jj1}) holds, we have $p\geq p_{u,2}^w$. By Corollary 3.7 of \cite{HJL02} (see also (\ref{wrc})), there exists a unique infinite open cluster $\tau$ in the random cluster representation of the Ising model with wired boundary conditions $WRC_{p,2}$-a.s.

By the correspondence of random-cluster configurations and Ising configurations, each infinite cluster in the random cluster representation must be a subset of an infinite cluster in the Ising model. Hence if $WRC_{p,2}$-a.s. there is a unique infinite open cluster, $\mu^+$ a.s.\ there exists an infinite ``$+$''-cluster in the Ising model, and $\mu^{-}$-a.s.\ there exists an infinite ``$-$''-cluster in the Ising model.

Let $\tau$ be the unique infinite open cluster in the random cluster configuration $\omega$. We define a site percolation configuration $\xi$ on $V(\LL_2)$, by letting all the vertices in $\tau$ have state 1, and all the other vertices have state 0. By \Cref{sz}, a.s.\ $\xi$ has no-infinite 0-clusters.  Again by the correspondence of the random cluster configuration and the Ising configuration and \cref{p118}, we obtain $\mu^{+}(\sA_+)=1$, and $\mu^{-}(\sA_-)=1$.

\subsection{Proof of Part IV of \Cref{ipl}.}
The identities (\ref{+1}) and (\ref{-1})  follows Part I and the fact that
\begin{eqnarray*}
p_{u,2}^{w}\leq p_{u,2}^f;
\end{eqnarray*}
and (\ref{hmf}) follows from the fact that when (\ref{jj2}) hold,
\begin{eqnarray*}
\mu^f=\frac{\mu^++\mu^-}{2};
\end{eqnarray*}
The decomposition of $\mu^f$ as convex combination of the extremal measures $\mu^+$ and $\mu^-$ it is a classical results in the case of regular lattices and for our lattices it was proven in Theorem 4.2 of \cite{RS01} and expressions (17) (18) of \cite{HJL02}.

\section{Proof of \Cref{coii,xorc}}\label{pxorc}

In this section, we prove \Cref{coii,xorc}.

\bigskip

\noindent\textbf{Proof of \cref{coii}.} We first prove Part I of the theorem. Let $s_{+}$ (resp.\ $s_{-}$) be the number of infinite ``$+$''-clusters (resp.\ infinite ``$-$''-clusters). By ergodicity, $\mathrm{Aut}(\LL_2)$-invariance and symmetry in ``$+$'' and ``$-$'' of $\mu_{1,f}\times \mu_{2,f}$, as well as \Cref{lbs}, one of the following cases occurs:
\begin{enumerate}[label=(\roman*)]
\item $\mu_1^f\times \mu_2^f((s_+,s_-)=(0,0))=1$; or
\item $\mu_1^f\times \mu_2^f((s_+,s_-)=(1,1))=1$; or
\item $\mu_1^f\times \mu_2^f((s_+,s_-)=(\infty,\infty))=1$
\end{enumerate} 
Case (i) is impossible to occur by \Cref{nzz}.  Case (ii) is impossible to occur by \Cref{ozz}. This completes the proof of Part I.

Now we show that Assumption II implies Assumption I. This follows from Theorem 4.1 of \cite{Sch99}, and the fact that the XOR Ising measure is the product measure of two i.i.d Ising models.

The facts that Assumption IV implies Assumption III and Assumption III implies Assumption II follows from \Cref{labc}. This completes the proof of the theorem.
$\hfill\Box$

\bigskip

Before proving \Cref{xorc}, we shall first introduce the following definition and  proposition in \cite{LS99} (see Theorem 3.3, Remark 3.4).

\begin{definition}\label{idd}Let $G=(V,E)$ be a graph and $\Gamma$ a transitive group acting on $G$. Suppose that $X$ is either $V$, $E$ or $V\cup E$. Let $Q$ be a measurable space and $\Omega:=2^V\times Q^X$. A probability measure $\mathbf{P}$ on $\Omega$ will be called a \textbf{site percolation with scenery} on $G$. The projection onto $2^V$ is the underlying percolation and the projection onto $Q^X$ is the scenery. If $(\omega,q)\in \Omega$, we set $\Pi_v(\omega,q)=(\Pi_v \omega,q)$. We say the percolation with scenery $\mathbf{P}$ is \textbf{insertion-tolerant} if $\mathbf{P}(\Pi_v\mathcal{B})>0$ for every measurable $\mathcal{B}\subset \Omega$ with positive measure. We say that $\mathbf{P}$ has \textbf{indistinguishable infinite clusters} if for every $\mathcal{A}\subset 2^V\times 2^V\times Q^X$ that is invariant under diagonal actions of $\Gamma$, for $\mathbf{P}$-a.e. $(\omega,q)$, either all infinite clusters $C$ of $\omega$ satisfy $(C,\omega,q)\in \mathcal{A}$, or they all satisfy $(C,\omega,q)\notin \mathcal{A}$.
 \end{definition}
 
 \begin{proposition}\label{idc}Let $\mathbf{P}$ be a site percolation with scenery on a graph $G=(V,E)$ with state space $\Omega:=2^V\times Q^X$, where $Q$ is a measurable space and $X$ is either $V$, $E$ or $V\cup E$. If $\mathbf{P}$ is $\Gamma$-invariant and insertion tolerant, then $\mathbf{P}$ has indistinguishable infinite clusters.
 \end{proposition}
 
 \noindent\textbf{Proof of \Cref{xorc}} Let $\Lambda=(V_{\Lambda},E_{\Lambda})$ be a subgraph of $\LL_2$ consisting of faces of $\LL_2$. Let $\Lambda_{*}=(V_{\Lambda_*},E_{\Lambda_*})$ be the dual graph of $\Lambda$, such that there is a vertex in $V_{\Lambda}$ corresponding to each triangle face in $\Lambda$, as well as the unbounded face; the edges in $E_{\Lambda}$ and $E_{\Lambda_*}$ are in 1-1 correspondence by duality.
 
 Consider an XOR Ising model on $\Lambda$ with respect to two i.i.d. Ising models $\sigma_3$, $\sigma_4$ with free boundary conditions and coupling constants $J\geq 0$ satisfying the assumption of \Cref{coii}. The partition function of the XOR Ising model can be computed by
 \begin{eqnarray*}
 Z_{\Lambda,f}=\sum_{\sigma_3,\sigma_4\in \{\pm 1\}^{V_{\Lambda}}}\prod_{(u,v)\in E_{\Lambda}}e^{J(\sigma_{3,u}\sigma_{3,v}+\sigma_{4,u}\sigma_{4,v})}.
 \end{eqnarray*}
 
 Following the same computations as in \cite{bd14}, we obtain
 \begin{eqnarray}
 Z_{\Lambda,f}= C_1\sum_{P_*\in \mathcal{P}_{*},P\in \mathcal{P},P\cap P_*=\emptyset}\left(\frac{2 e^{-2J}}{1+e^{-4J}}\right)^{|P_*|}\left(\frac{1-e^{-4J}}{1+e^{-4J}}\right)^{|P|}.\label{zlf}
 \end{eqnarray}
 where $\mathcal{P}_*$ (resp.\ $\mathcal{P}$) consists of all the contour configurations on $E_{\Lambda_*}$ (resp.\ $E_{\Lambda}$) such that each vertex of $V_{\Lambda_*}$  (resp.\ $V_{\Lambda}$) has an even number of incident present edges, and $C_1=2^{|V_{\Lambda}|-|E_{\Lambda}|+2}(e^{2j}-e^{-2J})^{|E_{\Lambda}|}$ is a constant. 
 
 When $J,K$ satisfies (\ref{jkr}), we have
 \begin{eqnarray*}
 \frac{2e^{-2J}}{1+e^{-4J}}&=&\frac{1-e^{-4K}}{1+e^{-4K}};\\
 \frac{2e^{-2K}}{1+e^{-4K}}&=&\frac{1-e^{-4J}}{1+e^{-4J}}.
 \end{eqnarray*}
Thus the partition function $Z_{\Lambda,f}$, up to a multiplicative constant, is the same as the partition function of the XOR Ising model on $\Lambda_*$ with coupling constant $K$. 

Recall that there is exactly one vertex $v_{\infty}\in V_{\Lambda_{*}}$ corresponding to the unbounded face in $\Lambda$. The XOR Ising model $\sigma_{XOR}=\sigma_1\sigma_2$ on $\Lambda_*$, corresponds to an XOR Ising model on $\Lambda_{*}\setminus \{v_{\infty}\}$ (which is a subgraph of $\LL_1$) with the boundary condition that all the boundary vertices have the same state in $\sigma_1$ and all the boundary vertices have the same state in $\sigma_2$. Hence the boundary condition must be a mixture of $++$, $+-$, $-+$ and $--$. However, each one of the 4 possible boundary conditions gives the same distribution of contours in the XOR Ising model. From the expression (\ref{zlf}), we can see that there is a natural probability measure on the set of contours $\Phi=\{(P,P_*):P\in\mathcal{P},P_*\in \mathcal{P}_*,P\cap P_*=\emptyset\}$, such that the probability of each pair of contours $(P,P_*)\in \Phi$ is proportional to $\left(\frac{2 e^{-2J}}{1+e^{-4J}}\right)^{|P_*|}\left(\frac{1-e^{-4J}}{1+e^{-4J}}\right)^{|P|}$, and the marginal distribution on $\mathcal{P}$ is the distribution of contours for the XOR Ising model on $\LL_2$ with coupling constant $J$ and free boundary conditions, while the marginal distribution on $\mathcal{P}_*$ is the distribution of contours for the XOR Ising model on $\LL_1$ with coupling constant $K$ and $++$ boundary conditions. 

We let $\Lambda$ and $\Lambda_*\setminus \{v_{\infty}\}$ increase and approximate the graph $\LL_2$ and $\LL_1$, respectively. If with a positive $\mu_{++}$ probability, there exists exactly one infinite contour $C$ in $\LL_2$, then $\mu_f$-a.s. there exists an infinite cluster in $\LL_2$ containing $C$, since contours in $\LL_1$ and $\LL_2$ are disjoint. Consider the XOR Ising spin configuration as a site percolation on $\LL_2$, with scenery given by contour configurations of $\LL_2$ within the ``$+$'' clusters of $\LL_2$. In the notation of \Cref{idd}, $Q=\{0,1\}$, and $X=E(\LL_2)$. An edge in $E(\LL_2)$ is present (has state ``1'') if and only if it is in a ``$+$''-cluster of the XOR Ising configuration on $\LL_2$ and present in the contour configuration of $\LL_2$. This way we obtain an automorphism-invariant and insertion-tolerant percolation with scenery. Let $\mathcal{A}\subset 2^V(\LL_2)\times 2^V(\LL_2)\times 2^E(\LL_2)$ be the triple $(C,\omega, q)$ such that 
\begin{itemize}
\item $\omega$ is an XOR Ising spin configuration on $\LL_2$; and
\item $C$ is an infinite ``$+$''-cluster; and
\item $q$ is the $\LL_2$-contour configuration within ``$+$''-clusters of $\omega$; and
\item $C$ contains an infinite contour in $q$.
\end{itemize}
We can see that $\mathcal{A}$ is invariant under diagonal actions of automorphisms. 
By \Cref{coii}, $\mu_f$-a.s. there exists infinitely many infinite ``$+$''-clusters in $\LL_2$. By \Cref{idc}, either all the infinite clusters are in $\mathcal{A}$, or no infinite clusters are in $\mathcal{A}$. Similar arguments applies for ``$-$''-clusters of $\LL_2$. Hence almost surely the number of infinite contours in $\LL_2$ is 0 or $\infty$. Since the distribution of infinite contours in $\LL_2$ is exactly that of contours for the XOR Ising model on $\LL_1$ with coupling constant $K$ and $++$ (or $+-$, $-+$, $--$) boundary condition, the theorem follows.
 $\hfill\Box$

\section{Proof of \Cref{chi,lth}}\label{p412}

In this section, we prove \Cref{chi,lth}.

Consider the XOR Ising model with spins located on vertices of the hexagonal lattice $\HH$ with coupling constants $J_a$, $J_b$, $J_c$ on horizontal, NW/SE, NE/SW edges, see \Cref{htst} for an embedding of $\HH$ into the plane such that all the edges are either horizontal, NW/SE or NE/SW.

\begin{figure}
\centering
\subfloat[$\HH$ and $\TT$]{\includegraphics[width=.45\textwidth]{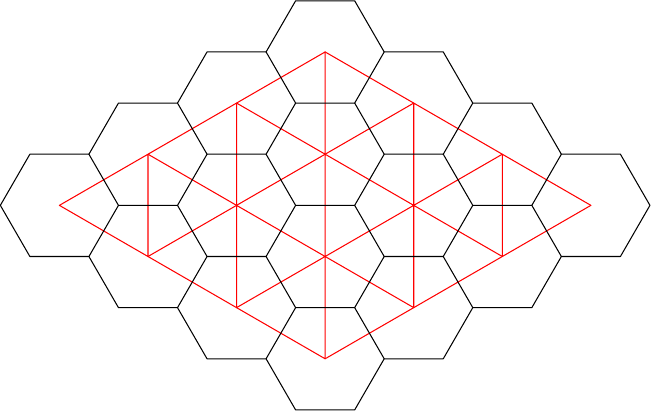}}\qquad
\subfloat[$\TT$ and $\HH'$]{\includegraphics[width = .45\textwidth]{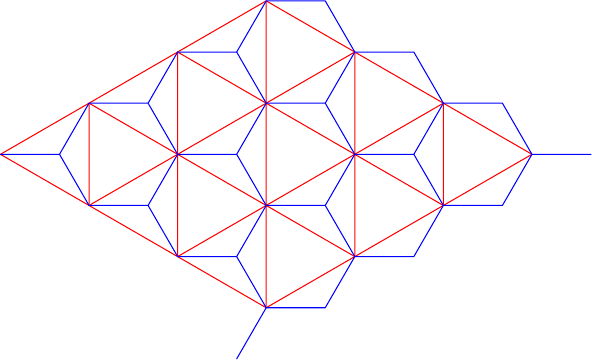}}
\caption{Hexagonal lattice $\HH$ (represented by black lines), dual triangular lattice $\TT$ (represented by red lines) and hexagonal lattice $\HH'$ obtained from the star-triangle transformation (represented by blue lines).}
\label{htst}
\end{figure}

A [4,6,12] lattice is a graph that can be embedded into the Euclidean plane $\RR^2$ such that each vertex is incident to 3 faces with degrees 4, 6, and 12, respectively. See \Cref{4612}.  We shall explain the relation between perfect matchings on the [4,6,12] lattice and constrained percolation configurations in the [3,4,6,4] lattice as discussed in \Cref{xorh}.

\begin{figure}
\centering
\includegraphics[width=.5\textwidth]{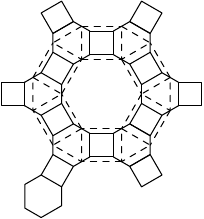}
\caption{[3,4,6,4] lattice (represented by dashed lines) and the [4,6,12] lattice $\mathbb{A}$ (represented by black lines).}
\label{4612}
\end{figure}

A \textbf{perfect matching}, or a \textbf{dimer configuration} on a [4,6,12] lattice is a subset of edges such that each vertex of the [4,6,12] lattice is incident to exactly one edge in the subset.  

A \df{Type-I} edge of a [4,6,12] lattice is an edge of a square face. Any other edge of the [4,6,12] lattice is a \df{Type-II} edge. We say two Type-II edges $e_1,e_2$ are \df{adjacent} if there exists a Type-I edge $e_3$, such that both $e_1$ and $e_2$ share a vertex with $e_3$ in the [4,6,12] lattice. A subset of Type-II edges is \df{connected} if for any two edges $e$ and $f$ in the subset, there exist a sequence of Type-II edges $e_0 (=e),\ e_1,\ \ldots,\ e_n(=f)$ in the subset, such that $e_i$ and $e_{i-1}$ are adjacent, for $1\leq i\leq n$. A \df{Type-II cluster} is a maximal connected set of present Type-II edges in  a perfect matching.

The [3,4,6,4] lattice is constructed as follows. A vertex of the [3,4,6,4] lattice is placed at the midpoint of each Type-II edge of the [4,6,12] lattice. Two vertices of the [3,4,6,4] lattices are joined by an edge if and only if they are midpoints of two adjacent Type-II edges. See \Cref{4612}. The restriction of any dimer configuration on the [4,6,12] lattice to Type-II edges naturally correspond to a constrained percolation configuration on the [3,4,6,4] lattice in $\Omega$. A Type-II edge is present in a dimer configuration if and only if its midpoint has state ``1'' in the corresponding constrained percolation configuration. It is straightforward to check that this way we obtain a bijection between restrictions to Type-II edges of dimer configurations on the [4,6,12] lattice and constrained percolation configurations on the [3,4,6,4] lattice in $\Omega$. Recall that constrained percolation configurations on the [3,4,6,4] lattice induces contour configurations on the hexagonal lattice $\HH$ and the triangular lattice $\TT$ by a 2-to-1 mapping $\phi:\Omega\rightarrow\Phi$. See \Cref{3464contour}.

From the connection of the [4,6,12] lattice and the [3,4,6,4] lattice, as well as the connection of the [3,4,6,4] lattice with the hexagonal lattice $\HH$ and the dual triangular lattice $\TT$ as described in \Cref{xorh}, we can see that each square face of the $[4,6,12]$ lattice is crossed by a unique edge of $\HH$ and a unique edge of $\TT$. Each vertex of $\HH$ is located at the center of a hexagon face of the [4,6,12] lattice, and each vertex of $\TT$ is located in the center of a degree-12 face of the [4,6,12] lattice.

Note that the $\HH$ is a bipartite graph; i.e. all the vertices can be colored black and white such that the vertices of the same color are not adjacent. Let $\Gamma$ be the translation group of the hexagonal lattice generated by translations along two different directions, such that the set of black vertices and the set of white vertices form two distinct orbits under the action of $\Gamma$. Note that $\Gamma$ acts transitively on the dual triangular lattice $\TT$.

In order to define a probability measure for perfect matchings on the [4,6,12] lattice, we introduce edge weights. We assign weight 1 to each Type-II edge, and weight $w_e$ to the Type-I edge $e$. Assume that the edge weights of the [4,6,12] lattice satisfy the following conditions.
 \begin{enumerate}[label=({B}{\arabic*})]
\item The edge weights are $\Gamma$-invariant.
\item If $e_1$, $e_2$ are two parallel Type-I edges around the same square face,  then $w_{e_1}=w_{e_2}$.
\item If $e_1$, $e_2$ are two perpendicular Type-I edges around the same square face, then $w_{e_1}^2+w_{e_2}^2=1$.
\end{enumerate}

The reason we assume (B1) is to define a $\Gamma$-translation-invariant probability measure. The reason we assume (B2) and (B3) is to define a probability measure for dimer configurations of the [4,6,12] lattice, which, under the connection described above to constrained percolation configurations in $\Omega$, will induce a probability measure on $\Omega$ satisfying the symmetry assumption (A3).

Under (B1)--(B3), the edge weights are described by three independent parameters. We may sometimes assume that the parameters satisfy the identity below, which reduces the number of independent parameters to two.
\begin{enumerate}[label=({B}{\arabic*})]
\setcounter{enumi}{3}
\item Let 
\begin{eqnarray}
h(x,y,z)=x+y+z+xy+xz+yz-xyz-1.\label{hxyz}
\end{eqnarray}
For each edge $e$ of the hexagonal lattice $\HH$, let $e_1$ (resp.\ $e_2$) be a Type-I edge of the [4,6,12] lattice parallel (resp.\ perpendicular) to $e$. Let
\begin{eqnarray*}
t_e=\frac{1-w_{e_1}}{w_{e_2}},
\end{eqnarray*}
where $w_{e_1}$ (resp.\ $w_{e_2}$) is the edge weight of $e_1$ (resp.\ $e_2$) for dimer configurations on the [4,6,12] lattice. Under the assumption (B1), $t_e$ is uniquely defined independent of the $e_1$, $e_2$ chosen - as long as $e_1$ is parallel to $e$ and $e_2$ is perpendicular to $e$.  Let $e_a, e_b, e_c$ be three edges of $\HH$ with distinct orientations in the embedding of $\HH$ into $\RR^2$. Then $h(t_{e_a},t_{e_b},t_{e_c})=0$.
\end{enumerate}

In \cite{KOS06}, the authors define a probability measure for any bi-periodic, bipartite, 2-dimensional lattice. Specializing to our case, let $\mu_{n,D}$ be the probability measure of dimer configurations on a toroidal $n\times n$ [4,6,12] lattice $\mathbb{A}_n$ (see \cite{KOS06} for details). Let $\mathcal{M}_n$ be the set of all perfect matchings on $\mathbb{A}_n$, and let $M\in \mathcal{M}_n$ be dimer configuration, then
\begin{eqnarray}
\mu_{n,D}(M)=\frac{\prod_{e\in M}w_{e}}{\sum_{M\in \mathcal{M}_n}\prod_{e\in M}w_e},\label{mnd}
\end{eqnarray}
where $w_e$ is the weight of the edge $e$. As $n\rightarrow\infty$, $\mu_{n,D}$ converges weakly to a translation-invariant measure $\mu_{D}$ (see \cite{KOS06}).

\begin{theorem}\label{m31}For the dimer model on the [4,6,12] lattice.
\begin{enumerate}
\item If the edge weights satisfy (B1)-(B4), $\mu_D$ almost surely there are neither infinite Type-II clusters nor infinite contours.
\item If the edge weights satisfy (B1)-(B3), $\mu_D$ almost surely there exists at most one infinite contour.
\end{enumerate}
\end{theorem}

\begin{proof}Let $\widetilde{\mu}_D$ be the marginal distribution of $\mu_D$ restricted on Type-II edges. Recall that the restriction of dimer configurations to Type-II edges on the [4,6,12] lattice, are in 1-1 correspondence with constrained percolation configurations on the $[3,4,6,4]$ lattice in $\Omega$, as described before. See also \Cref{4612}. 

Also recall that $\widetilde{\mu}_D$ is the weak limit of measures on larger and larger tori; since the edge weights are translation-invariant, the measures on tori are translation-invariant. Hence $\widetilde{\mu}_D$ satisfies (A1) if edge weights satisfies (B1).

The measure $\widetilde{\mu}_D$ is both translation-invariant and mixing (see \cite{KOS06}), hence $\widetilde{\mu}_D$ is totally ergodic and satisfies (A2).

If the edge weights satisfy (B2) and (B3), the measures on tori are symmetric under $\theta$. Hence $\widetilde{\mu}_D$ satisfies (A3).

By the results in \cite{bd14}, the marginal distribution of contours in $\HH$ (resp.\ $\TT$) under $\widetilde{\mu}_D$ is the same as the distribution of contours of an XOR Ising model $\sigma_{XOR,\TT}$ (resp. $\sigma_{XOR,\HH}$) with spins located on vertices of $\TT$ (resp.\ $\HH$), if the dimer edge weights and Ising coupling constants satisfy the following conditions:
\begin{itemize}
\item each Type-II edge has weight 1;
\item each Type-I edge parallel to an edge of $e$ with coupling constants $J_e$ has weight $w_e$ such that $w_e=\frac{1-e^{-4J_e}}{1+e^{-4J_e}}$;
\item each Type-I edge perpendicular to an edge $e$ with coupling constants $J_e$ has weight $w_e$ such that $w_e=\frac{2e^{-2J_e}}{1+e^{-4J_e}}$.
\end{itemize}

Moreover, when the edge weights satisfy (B1)-(B3), $\sigma_{XOR,\TT}$ and $\sigma_{XOR,\HH}$ are dual to each other, i.e.\ the coupling constants $J_{\tau}$ and $K_{\tau}$ on a pair of dual edges $e\in \HH,e^*\in \TT$ satisfy (\ref{dua}). The finite energy assumptions (A4) and (A5) follows from the finite energy of $\sigma_{XOR,\HH}$ and $\sigma_{XOR,\TT}$.

Since $\sigma_{XOR,\HH}$ and $\sigma_{XOR,\TT}$ are dual to each other, one of the following cases might occur
\begin{enumerate}
\item $\sigma_{XOR,\HH}$ is in the low-temperature state, and $\sigma_{XOR,\TT}$ is in the high-temperature state;
\item $\sigma_{XOR,\HH}$ is in the high-temperature state, and $\sigma_{XOR,\TT}$ is in the low-temperature state;
\item both $\sigma_{XOR,\HH}$ and $\sigma_{XOR,\TT}$ are in the critical state.
\end{enumerate}
See \Cref{xori} for definitions of the low-temperature state, high-temperature state and critical state for XOR Ising models.

In Case III, both (A6) and (A7) are satisfied because of the ergodicity of measures for the critical XOR Ising model on $\HH$ and $\TT$; see \Cref{ec}. Moreover, Case III occurs if and only if the edge weights satisfy (B4). Hence when the edge weights satisfy (B1)-(B4), $\widetilde{\mu}_D$ satisfy (A1)-(A7). \Cref{m31} I follows from \Cref{m21}II(c).

Note that the measures for the high-temperature XOR Ising models on $\HH$ and $\TT$ are also ergodic; see \Cref{he}. Therefore, in each case of I, II and III, at least one of (A6) and (A7) is satisfied. Then \Cref{m31} II follows from \Cref{m21} II(a)(b) and \Cref{l81,l82}.

\end{proof}

In order to prove the \Cref{chi,lth}, we prove the following lemmas.

\begin{lemma}\label{ec}The measure for the critical XOR Ising model on $\HH$ (resp.\ $\TT$), obtained as the weak limit of measures on tori, is ergodic. 
\end{lemma}
\begin{proof}Let $\rho=\frac{\sigma+1}{2}$, where $\sigma:V_{\HH}\rightarrow\{\pm 1\}$ is the spin configuration for an Ising model on $\HH$.
Following the same arguments as in the proof of Lemma 10.2 in \cite{HL16}, it suffices to show that for the critical Ising model on $\HH$ (resp.\ $\TT$), we have
\begin{eqnarray*}
\lim_{|u-v|\rightarrow\infty} |\langle\rho(u)\rho(v) \rangle-\langle\rho(u)\rangle\langle\rho(v)\rangle|=0,
\end{eqnarray*}
which is equivalent to show that 
\begin{eqnarray}
\lim_{|u-v|\rightarrow\infty}\langle \sigma(u)\sigma(v) \rangle=0.\label{vs}
\end{eqnarray}
Note that $\langle\sigma(u)\sigma(v)\rangle$ is an even spin correlation function (i.e.\ the expectation of the product of spins on an even number of vertices), and hence for all the infinite-volume, translation-invariant Gibbs measures of the Ising model on $\HH$ (resp.\ $\TT$) corresponding to the given coupling constant, $\langle\sigma_u\sigma_v\rangle$ has a unique value; see \cite{Leb77}. 

Consider the FK random cluster representation of the Ising model with $q=2$, the two-point spin correlation $\langle\sigma(u)\sigma(v)\rangle$ is exactly the connectivity probability of $u$ and $v$ in the random cluster model, up to a multiplicative constant; see Chapter 1.4 of \cite{GrGrc}. Therefore in order to show (\ref{vs}), it suffices to show that the connectivity probabilities of two vertices in the corresponding random cluster model, as the distances of the two vertices go to infinity, converge to zero. 

 By Theorem 4 of \cite{BR07}, we infer that the connectivity probabilities of two vertices in the $q=2$ random cluster model corresponding to the critical Ising model on the triangular lattices with coupling constants $K_a,K_b,K_c$ satisfying $g(K_a,K_b,K_c)=0$ converge to zero as the distances of the two vertices go to infinity. 

Note that the hexagonal lattice is a bipartite graph, i.e., all the vertices can be colored black and white such that vertices of the same color can never be adjacent.  
 Recall that the star-triangle transformation is a replacement of each black vertex of $\HH$, as well as its incident edges, into a triangle. The resulting graph is a triangular lattice $\TT'$; see the right graph of \Cref{htst}. The parameters of the random cluster model on $\HH$ and the random cluster model $\TT'$ satisfy certain identities, such that the probabilities of connections of any two adjacent vertices in $\TT'$ (which are also vertices in $\HH$) internal to each triangle face in $\TT'$ which has a black vertex of $\HH$ in the center are the same for the random cluster model on $\HH$ and the random cluster model on $\TT'$; see Page 160-161 of \cite{GrGrc}.
 
 Using a star-triangle transformation (see the right graph of \Cref{htst}), and (6.69) of \cite{GrGrc}, we deduce that the connectivity probabilities of two vertices in the $q=2$ random cluster model on the hexagonal lattice corresponding to the critical Ising model on $\HH$ with coupling constants $J_a,J_b,J_c$ satisfying $f(J_a,J_b,J_c)=0$ converge to zero as the distances of the two vertices go to infinity. Note that the weak limit of of measures with free boundary conditions is known to exist and translation-invariant, see Theorem (4.19) of \cite{GrGrc}. By the uniqueness of $\langle \sigma(u)\sigma(v) \rangle$ under all the translation-invariant measures, we obtain that (\ref{vs}) holds under the measure obtained as the weak limit of measures with periodic boundary conditions. 
\end{proof}

\begin{lemma}\label{he}The measure for the high-temperature XOR Ising model on $\HH$ (resp.\ $\TT$), obtained as the weak limit of measures on tori, is ergodic. 
\end{lemma}
\begin{proof}The identity (\ref{vs}) holds under the measure for the high-temperature Ising model on $\HH$ (resp.\ $\TT$); see \cite{ZL12,DC13}.
\end{proof}

\noindent\textbf{Proof of \Cref{chi}.} We first show that in the critical XOR Ising model on $\HH$ or $\TT$, almost surely there are no infinite contours. It is proved in \cite{bd14} that the contours of XOR Ising model with spins located on $\HH$ (resp.\ $\TT$) have the same distribution as contours in $\TT$ (resp.\ $\HH$) for the Type-II clusters of dimer configurations on the $[4,6,12]$ lattice, if the coupling constants of the XOR Ising model and the edge weights of the $[4,6,12]$ lattice satisfy certain conditions. It is not hard to check that when the coupling constants of the XOR Ising model on $\TT$ (resp.\ $\HH$) are critical, then the edge weights of the corresponding dimer model on the [4,6,12] lattice satisfy (B1)-(B4). By \Cref{m31} I, almost surely there are no infinite contours in the critical XOR Ising model on $\TT$ or $\HH$.

Now we prove that almost surely there are no infinite clusters for the critical XOR Ising model on $\HH$ or $\TT$. We write down the proof for the critical XOR Ising model on $\HH$ here, the case for the XOR Ising model on $\TT$ can be proved in a similar way. 

Let $\sA$ be the event that there exists an infinite cluster for the XOR Ising model on $\HH$. Assume that $\mu(\sA)>0$; we will obtain a contradiction. By translation-invariance of $\sA$ and \Cref{ec}, if $\mu(\sA)>0$ then $\mu(\sA)=1$. Let $\sA_1$ (resp. $\sA_2$) be the event that there exists an infinite ``$+$''-cluster (resp.\ ``$-$''-cluster)
for the XOR Ising model on $\HH$, then
\begin{eqnarray}
\mu(\sA_1\cup\sA_2)=1.\label{u12}
\end{eqnarray}
By symmetry $\mu(\sA_1)=\mu(\sA_2)$. By translation-invariance of $\sA_1$, $\sA_2$ and \Cref{ec}, either $\mu(\sA_1)=\mu(\sA_2)=1$, or $\mu(\sA_1)=\mu(\sA_2)=0$. By (\ref{u12}), we have $\mu(\sA_1)=\mu(\sA_2)=1$, hence $\mu(\sA_1\cap\sA_2)=1$, i.e. $\mu$-a.s.\ there exist both an infinite ``$+$''-cluster and an infinite ``$-$''-cluster in the critical XOR Ising configuration on $\HH$.

Let $\phi_{\TT}$ be the contour configuration associated to the critical XOR Ising configuration on $\HH$. Let $\omega\in \phi^{-1}(\phi_{\TT})$ be a constrained percolation configuration on the $[3,4,6,4]$ lattice whose contour configuration is $\phi_{\TT}$. It is not hard to check that in $\omega$ there exist both an infinite 1-cluster and an infinite 0-cluster if in the original XOR Ising model on $\HH$, there exist both an infinite ``$+$''-cluster and an infinite ``$-$''-cluster. By \Cref{io}, $\mu$-a.s.\ there exists an infinite contour in $\phi_{\TT}$. The contradiction implies that $\mu$-a.s. there are no infinite clusters in the critical XOR Ising model on $\HH$.
$\hfill\Box$

\bigskip
\noindent\textbf{Proof of \Cref{lth}.} By the correspondence between contours in an XOR Ising model with spins located on vertices of $\HH$ (resp.\ $\TT$) and contours on $\TT$ (resp.\ $\HH$) for the Type-II clusters of dimer configurations on the $[4,6,12]$ lattice, as proved in \cite{bd14}, as well as correspondence between Type-II clusters of dimer configurations on the [4,6,12] lattice and clusters of constrained configurations on the $[3,4,6,4]$ lattice and \Cref{m21}, it suffices to show that the probability measure for the low-temperature XOR Ising model on $\HH$ (resp.\ $\TT$) satisfies (A1)-(A5) and (A7) (resp.\ (A1)-(A6)), for $\LL_1=\HH$.

 It is straightforward to verify (A1)-(A5). The assumption (A6) (resp.\ (A7)) follows from \Cref{he}.
$\hfill\Box$

\appendix
\section{Deterministic Results about Contours and Clusters}\label{ctcl}

In this section, we prove deterministic results concerning contours and clusters for the constrained percolation model on the $[m,4,n,4]$ lattice in preparation to prove the main theorems.

Let $P$ be the underlying plane into which the $[m,4,n,4]$ lattice is embedded. Recall that when $\frac{1}{m}+\frac{1}{n}=\frac{1}{2}$, $P$ is the Euclidean plane $\RR^2$ and the graph $G$ is amenable. When $\frac{1}{m}+\frac{1}{n}<\frac{1}{2}$, $P$ is the hyperbolic plane $\HH^2$ and the graph $G$ is non-amenable.

We consider an embedding of the $[m,4,n,4]$ lattice into $P$ in such a way that each edge has length 1.
Let $\phi\in\Phi$ be a contour configuration, and let $C$ be a contour in $\phi$. To each component of $P\setminus \phi$, we associate an \df{interface}, which is a closed set consisting of all the points in the component whose distance to $C$ is $\frac{1}{4}$. Here by distance, we mean either Euclidean distance or hyperbolic distance depending on whether $P$ is $\RR^2$ or $\HH^2$. Obviously each interface is either a self-avoiding cycle or a doubly infinite self-avoiding walk. See \Cref{3464interface} for an example of interfaces on the $[3,4,6,4]$ lattice.

\begin{figure}
\includegraphics[width=.5\textwidth]{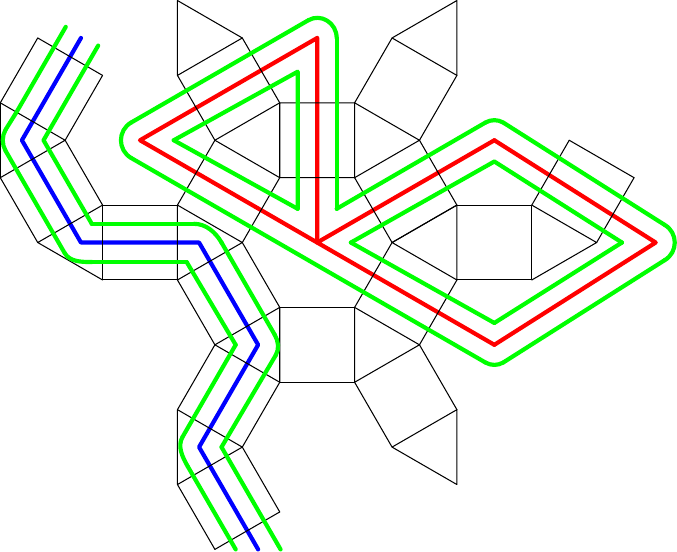}
\caption{Contour configuration and interfaces. Blue lines represent contours of $\HH$. Red lines represent contours of $\TT$. Green lines represent interfaces.}
\label{3464interface}
\end{figure}

Note that when $\frac{1}{m}+\frac{1}{n}<\frac{1}{2}$ and $\min\{m,n\}\geq 3$, both $\LL_1$ and $\LL_2$ are vertex-transitive, non-amenable, planar graphs with one end. 

\begin{lemma}\label{rl}Whenever we have an interface $I$, let $F_{I}$ be the set consisting of all the vertices of $G$ whose (Euclidean or hyperbolic) distance to the interface is $\frac{1}{4}$. Then all the vertices of $F_{I}$ are in the same cluster. If $I$ is a doubly-infinite self-avoiding path, then $F_{I}$ is part of an infinite cluster.
 \end{lemma}
 \begin{proof}Recall that the interface $I$ is either a self-avoiding cycle or a doubly-infinite self-avoiding walk. Give $I$ a fixed direction. Moving along $I$ following the fixed direction, let $\{S_j\}_{j\in J}$ be the set of faces crossed by $I$ in order, where $J\subseteq \ZZ$ is a set of integers, such that for any $j_1<j_2$, $j_1,j_2\in J$, $I$ crosses $S_{j_1}$ first, and then crosses $S_{j_2}$. Note that it is possible to have $S_{j_1}=S_{j_2}$ for $j_1\neq j_2$.
 
For any two vertices $u,v\in F_{I}$, we can find a sequence of indices $j_1<j_2<\ldots<j_k\in J$, and a sequence of vertices of $G$,
\begin{eqnarray}
u=v_{j_1,1},v_{j_1,2}(=v_{j_2,1}),v_{j_2,2}(=v_{j_3,1})=\ldots=v_{j_k,2}=v,\label{vss}
\end{eqnarray}
such that for any $1\leq i\leq k$, $v_{j_i,1}$ and $v_{j_i,2}$ are two vertices (which may not be distinct) in $F_I\cap \partial S_{j_i}$, and there exists a path $\ell_{v_{j_i,1}v_{j_i,2}}$ connecting $v_{j_i,1}$ and $v_{j_i,2}\subset \partial S_{j_i}$, and $\ell_{v_{j_i,1}v_{j_i,2}}\cap I=\emptyset$.  Note that any two consecutive vertices in (\ref{vss}) are in the same cluster, therefore $u$ and $v$ are in the same cluster.

 If $I$ is a doubly-infinite self-avoiding path, then $I$ crosses infinitely many faces.  Each face crossed by $I$ has at least one boundary vertex in $F_I$. Each vertex of $F_I$ is a boundary vertex of at most 4 faces. Therefore $|F_I|=\infty$. Since all the vertices in $F_I$ are in the same cluster, $F_I$ is part of an infinite cluster.
 \end{proof}
 
 In the following lemma, contours may be primal or dual as usual.

\begin{lemma}\label{cc}If there exist at least two infinite contours, then there exists an infinite 0-cluster or an infinite 1-cluster. Moreover, if $C_1$ and $C_2$ are two infinite contours, then there exists an infinite cluster incident to $C_1$.
\end{lemma}
\begin{proof}If there exist at least two infinite contours, then we can find two distinct infinite contours $C_1$ and $C_2$, two points $x\in C_1$ and $y\in C_2$ and a self-avoiding path $p_{xy}$, consisting of edges of $G$ and two half-edges, one starting at $x$ and the other ending at $y$, and connecting $x$ and $y$, such that $p_{xy}$ does not intersect any infinite contours except at $x$ and at $y$. Indeed, we may take any path intersecting two distinct contours, and then take a minimal subpath with this property.

Let $v\in V$ be the first vertex of $G$ along $p_{xy}$ starting from $x$. Let $u$ be the point along the line segment $[v,x]$ lying on an interface of $C_1$. Let $\ell_u$ be the interface of $C_1$ containing $u$. Then $\ell_u$ is either a doubly-infinite self-avoiding path or a self-avoiding cycle.

We consider these two cases separately. Firstly, if $\ell_u$ is a doubly-infinite self-avoiding path, then we claim that $v$ is in an infinite (0 or 1-)cluster of the constrained site configuration on $G$. Indeed, this follows from \Cref{rl}.

Secondly, if $\ell_u$ is a self-avoiding cycle, then $P\setminus \ell_u$ has two components, $Q_v$ and $Q_v'$, where $Q_v$ is the component including $v$. Since $\ell_u$ is a cycle, exactly one of $Q_v$ and $Q_v'$ is bounded, the other is unbounded. Since $C_1\subseteq Q_v'$, and $C_1$ is an infinite contour, we deduce that $Q_v'$ is unbounded, and $Q_v$ is bounded. Since $y\notin \ell_u$, either $y\in Q_v$, or $y\in Q_v'$. If $y\in Q_v'$, then any path, consisting of edges of $G$ and one half-edge incident to $y$, connecting $v$ and $y$ must cross $C_1$. In particular, $p_{xy}$ crosses $C_1$ not only at $x$, but also at some point other than $x$. This contradicts the definition of $p_{xy}$. Hence $y\in Q_v$. Since $C_1\cap C_2=\emptyset$, this implies $C_2\subseteq Q_v$; because if $C_2\cap Q_v'\neq \emptyset$, then $C_2\cap C_1\neq \emptyset$. But $C_2\subseteq Q_v$ is impossible since $C_2$ is infinite and $Q_v$ is bounded. Hence this second case is impossible.

Therefore we conclude that if there exist at least two infinite contours, then there exists an infinite (0 or 1)-cluster.
\end{proof}

\begin{lemma}\label{io} Let $x\in V$ be in the infinite 0-cluster, let $y\in V$ be in the infinite 1-cluster, and  let $\ell_{xy}$ be a path, consisting of edges of $G$ and connecting $x$ and $y$. Then $\ell_{xy}$ has an odd number of crossings with infinite contours in total.

In particular, if there exist both an infinite 0-cluster and an infinite 1-cluster in a constrained percolation configuration $\omega\in\Omega$, then there exists an infinite contour in $\phi(\omega)\in\Phi$.
\end{lemma}

\begin{proof}Same as Lemma 2.8 of \cite{HL16}.
\end{proof}

\begin{lemma}\label{icic}Let $C_{\infty}$ be an infinite contour. Then each infinite component of $G\setminus C_{\infty}$ contains an infinite cluster that is incident to $C_{\infty}$.
\end{lemma}
\begin{proof} The lemma can be proved using similar technique as in Lemma 2.7 of \cite{HL16}.
\end{proof}

\begin{lemma}\label{icazo}Let $\omega\in \Omega$. Assume that there is exactly one infinite 0-cluster and exactly one infinite 1-cluster in $\omega$. Assume that there exist a vertex $x$ in the infinite 0-cluster,  a vertex $y$ in the infinite 1-cluster, and a path $\ell_{xy}$, consisting of edges of $G$ and joining $x$ and $y$, such that $\ell_{xy}$ crosses exactly one infinite contour, $C_{\infty}$. Then $C_{\infty}$ is incident to both the infinite 0-cluster and the infinite 1-cluster.

\end{lemma}

\begin{proof}By \Cref{icic}, there is an infinite cluster in each infinite component of $G\setminus C_{\infty}$. Since there are exactly two infinite clusters, $G\setminus C_{\infty}$ has at most 2 infinite components. Since each infinite cluster lies in some infinite component of $G\setminus C_{\infty}$, the number of infinite components of $G\setminus C_{\infty}$ is at least one.

If $G\setminus C_{\infty}$ has exactly two infinite components, then we can construct two infinite connected set of vertices in the two infinite components of $G\setminus C_{\infty}$, as a consequence of \Cref{icic}, denoted by $W_1$ and $W_2$, such that $C_{\infty}$ is incident to both $W_1$ and $W_2$. Moreover, $W_1$ and $W_2$ are exactly part of the infinite 0-cluster and part of the infinite 1-cluster. Therefore $C_{\infty}$ is incident to both the infinite 0-cluster and the infinite 1-cluster.

If $G\setminus C_{\infty}$ has exactly one infinite component, denoted by $R$, then both the infinite 0-cluster and the infinite 1-cluster lie in $R$, and in particular $x,y\in R$. We can find a path $\ell'_{xy}$, connecting $x$ and $y$, using edges of $G$, such that $\ell'_{xy}$ does not cross $C_{\infty}$ at all. We can change path from $\ell'_{xy}$ to $\ell_{xy}$ by choosing finitely many faces $S_1$,$S_2$,\ldots, $S_k$ of $G$; along the boundary of each face, make every present edge in the path absent and every absent edge in the path present; and we perform this procedure for $S_1,\ldots,S_k$ one by one.  Such a the path modification procedure does not change the parity of the number of crossings of the path with $C_{\infty}$. Hence we infer that $\ell_{xy}$ intersects $C_{\infty}$ an even number of times. But this is a contradiction to Lemma \ref{io} which says $\ell_{xy}$ crosses $C_{\infty}$ an odd number of times, since we assume that $\ell_{xy}$ crosses exactly one infinite contour $C_{\infty}$.
\end{proof}

\begin{lemma}\label{isc}Assume that $\xi$ is an infinite cluster, and $C$ is an infinite contour. Assume that $x$ is a vertex of $G$ in  $\xi$, and let $y\in C$ be the midpoint of an edge of $G$. Assume that there exists a path $p_{xy}$ connecting $x$ and $y$, consisting of edges of $G$ and a half-edge incident to $y$, such that $p_{xy}$ crosses no infinite contours except at $y$.  Let $z$ be the first vertex of $G$ along $p_{xy}$ starting from $y$. Then $z\in\xi$.
\end{lemma}

\begin{proof}The proof is an adaptation of the proof of Lemma 2.9 in \cite{HL16}. Lemma 2.9 in \cite{HL16} applies when the graph $G$ is a square grid embedded in the Euclidean plane $\RR^2$. This lemma applies when the graph $G$ is a general $[m,4,n,4]$ lattice embedded in either the Euclidean plane or the hyperbolic plane.

Since $p_{xy}$ crosses no infinite contours except at $y$, let $C_{1},\ldots,C_{m}$ be all the finite contours crossing $p_{xy}$.  We claim that $P\setminus \cup_{i=1}^{m}C_{i}$ has a unique unbounded component, which contains both $x$ and $y$. Indeed,  since $x\in \xi$ and $y\in C$; neither the infinite cluster $\xi$ nor the infinite contour $C$ can lie in a bounded component of $P\setminus \cup_{i=1}^{m}C_{i}$.

Let $I$ be the intersection of the union of the interfaces of $C_1,\ldots, C_m$ with the unique unbounded component of $P\setminus \cup_{i=1}^{m}C_{i}$. Since each $C_{i}$, $1\leq i\leq m$, is a finite contour, each interface of $C_{i}$ is finite. In particular, $I$ consists of finitely many disjoint self-avoiding cycles, denoted by $D_1,\ldots,D_t$. For $1\leq i\leq t$, $P\setminus D_i$ has exactly one unbounded component, and one bounded component.  Moreover, for $i\neq j$, $D_i$ and $D_j$ come from interfaces of distinct contours.

Let $B_i$ be the bounded component of $P\setminus D_i$.  We claim that each $B_i$ is simply-connected, and $B_i\cap B_j=\emptyset$, for $i\neq j$. Indeed, $B_i$ is simply connected, since the boundary of $B_i$, $D_i$ is a self-avoiding cycle, whose embedding on the plane is a simple closed curve, for $1\leq i\leq t$. Let $1\leq i<j\leq t$. Since $D_i$ and $D_j$ are disjoint, either $B_i\cap B_j=\emptyset$, or one of $B_i$ and $B_j$ is a proper subset of the other. Without loss of generality, assume $B_i$ is a proper subset of $B_j$. Then $D_i$ is a proper subset of $B_j$. Hence $D_i$ is in a bounded component of $P\setminus\cup_{k=1}^{m}C_k$, which contradicts the definition of $D_i$.

Let $R_i$ be the set of faces $F$ of $G$, for which $B_i\cap F\neq \emptyset$. Let $\widetilde{B}_i=\cup_{F\in R_i}F$. Note that for $1\leq i\leq t$, each $\widetilde{B}_i$ is a simply-connected, closed set. Let $B_i'$ be the interior of $\widetilde{B}_i$. Then each $B_i'$ is a simply-connected, open set; moreover, $B_i'\cap B_j'=\emptyset$, if $i\neq j$. This follows from the fact that for $i\neq j$, $D_i$ and $D_j$ come from interfaces of distinct contours, and the fact that $B_i\cap B_j=\emptyset$, for $i\neq j$.

Let $B'=\cup_{i=1}^{t}B'_i$. Then $B'$ is open, and $x,y,z\in P\setminus B'$, although $x$ and $z$ may be on the boundary of $B'$.

There is a path $p_{xy}'\subseteq [p_{xy}\cap (P\setminus B')]\cup \partial B'$, connecting $x$ and $y$, where $\partial B'$ is the boundary of $B'$. More precisely, $p_{xy}$ is divided by $\partial B'$ into  segments; on each  segment of $p_{xy}$ in $P\setminus B'$, $p'_{xy}$ follows the path of $p_{xy}$; for each segment of $p_{xy}$ in $B'$, $p'_{xy}$ follows the boundary of $B'$ to connect the two endpoints of the segment. This is possible since $B'$ consists of bounded, disjoint, simply-connected, open sets $B_i'$, for $1\leq i\leq t$, and both $x$ and $v$ are in the complement of $B'$ in $P$.

All the vertices along $p'_{xy}$ are in the same cluster. In particular, this implies that $x$ and $z$ are in the same infinite cluster $\xi$.
\end{proof}

\begin{lemma}\label{ct}If there exist exactly two infinite contours, then there exists an infinite cluster incident to both infinite contours. 
\end{lemma}
\begin{proof}Let $C_1$, $C_2$ be the two infinite contours. Since there exist only two infinite contours, we can find two points $x\in C_1$, $y\in C_2$, and a self-avoiding path $p_{xy}$, consisting of edges of $G$ and two half-edges, one starting at $x$ and the other ending at $y$, and connecting $x$ and $y$, such that $p_{xy}$ does not intersect infinite contours except at $x$ and at $y$.

Let $v\in V$ be the first vertex of $G$ along $p_{xy}$ starting from $x$. By the proof of \Cref{cc}, $v$ is in an infinite cluster $\xi$ incident to $C_1$. Let $z$ be the first vertex of $G$ along $p_{xy}$. By \Cref{isc},  $z\in \xi$, and therefore $\xi$ is an infinite cluster incident to both $C_1$ and $C_2$.
\end{proof}

\begin{lemma}\label{cac}Let $\omega\in \Omega$. If there is exactly one infinite 0-cluster and exactly one infinite 1-cluster in $\omega$, then there exists an infinite contour that is incident to both the infinite 0-cluster and the infinite 1-cluster in $\omega$.
\end{lemma}

\begin{proof} Let $x$ be a vertex in the infinite 0-cluster, and let $y$ be a vertex in the infinite 1-cluster. Let $\ell_{xy}$ be a path joining $x$ and $y$ and consisting of edges of $G$.

By \Cref{io}, $\ell_{xy}$ must cross infinite contours an odd number of times. By \Cref{icazo}, if $\ell_{xy}$ crosses exactly one infinite contour, $C_{\infty}$, then $C_{\infty}$ is incident to both the infinite 0-cluster and the infinite 1-cluster, and so the lemma is proved in this case.

Suppose that there exist more than one infinite contour crossing $\ell_{xy}$. Let $C_1$ and $C_2$ be two distinct infinite contours crossing $\ell_{xy}$. 

Let $u\in C_1\cap \ell_{xy}$ and $v\in C_2\cap \ell_{xy}$ (Here we interpret the contours and the paths as their embeddings to $P$, so that $u,v$ are points in $P$), such that the portion of $\ell_{xy}$ between $u$ and $v$, $p_{uv}$, does not cross any infinite contours except at $u$ and at $v$.  As in the proof of Lemma \ref{cc}, let $u_1$ be the first vertex of $G$ along $p_{uv}$, starting from $u$; and let $v_1$ be the first vertex of $G$ along $p_{uv}$ starting from $v$. Let $u_2$ (resp.\ $v_2$) be the point along the line segment $[u,u_1]$ (resp.\ $[v,v_1]$) lying on an interface. Following the procedure in the proof of Lemma \ref{cc}, we can find an infinite cluster $\xi_1$, such that $u_1\in \xi_1$.  The following cases might happen:
\begin{enumerate}[label=\Roman*]
\item $x\notin \xi_1$ and $y\notin \xi_1$;
\item $x\notin \xi_1$ and $y\in \xi_1$;
\item $x\in \xi_1$ and $y\notin \xi_1$;
\item $x\in \xi_1$ and $y\in\xi_1$.
 \end{enumerate}
 First of all, Case IV is impossible because we assume $x$ and $y$ are in two distinct infinite clusters. Secondly, if Case I is true, then there exist at least three infinite clusters, which is a contradiction to our assumption.
 
 Case II and Case III can be handled using similar arguments, and we write down the proof of Case II here.
 
 If Case II is true, first note that $y\in\xi_1$ implies that $C_1$ is incident to the infinite 1-cluster.  Let $z$ be the first point in $C_1\cap \ell_{xy}$ (again interpret edges as line segments), when traveling along $\ell_{xy}$ starting from $x$. Let $p_{xz}$ be the portion of $\ell_{xy}$ between $x$ and $z$. 

Next, we will prove the following claim by induction on the number of complete edges of $G$ along $p_{xz}$ (in contrast to the half edge along $p_{xz}$ with an endpoint $z$).
\begin{claim}\label{cl}
 Under Case II, there is an infinite contour  
 incident to both the infinite 0-cluster and the infinite 1-cluster.
\end{claim}
Assume that the number of complete edges of $G$ along $p_{xz}$ is $n$, where $n=0,1,2,\ldots$.

First of all, consider the case when $n=0$. This implies that $C_1$ is incident to the infinite 0-cluster at $x$. Recall that $C_1$ is also incident to the infinite 1-cluster at $y$, and so \Cref{cl} is proved.

We make the following induction hypothesis:
\begin{itemize}
\item Claim \ref{cl} holds for $n\leq k$, where $k\geq 0$.
\end{itemize}

Now we consider the case when $n=k+1$.
The \df{interior points} of $p_{xz}$ are all points along $p_{xz}$ except $x$ and $z$. We consider two cases:
\begin{enumerate}[label=(\alph*)]
\item at interior points, $p_{xz}$ crosses only finite contours but not infinite contours;
\item at interior points, $p_{xz}$ crosses infinite contours.
\end{enumerate}

 We claim that if Case (a) occurs, then $C_1$ is incident to both the infinite 0-cluster and the infinite 1-cluster. It suffices to show that $C_1$ is incident to the infinite 0-cluster.

Let $z_1$ be the first vertex in $V$ along $p_{xz}$ starting from $z$. According to \Cref{isc}, both  $x$ and $z_1$ are in the infinite 0-cluster.  We infer that $C_1$ is incident to the infinite 0-cluster at $x$, if $p_{xz}$ intersects only finite contours at interior points. 

Now we consider Case (b). Let $C_3$ be an infinite contour crossing $p_{xz}$ at interior points. Obviously, $C_3$ and $C_1$ are distinct, because $C_1$ crosses $p_{xz}$ only at $z$. Let $w$ be the last point in $C_3\cap p_{xz}$, when traveling along $p_{xz}$, starting from $x$, and let $p_{wz}$ be the portion of $p_{xz}$ between $w$ and $z$. Assume $p_{wz}$ does not cross infinite contours at interior points.

Let $w_1$ be the first vertex of $G$ along $p_{wz}$, starting from $w$, and let $w_2$ be the midpoint of $w$ and $w_1$. According to the proof of Lemma \ref{cc},  we can find an infinite cluster $\xi_3$ including $w_1$. The following cases might happen:
\begin{enumerate}[label=\roman*]
\item $x\notin \xi_3$, and $y\notin\xi_3$;
\item $x\in \xi_3$, and $y\notin\xi_3$;
\item $x\notin \xi_3$, and $y\in\xi_3$;
\item $x\in \xi_3$, and $y\in\xi_3$.
\end{enumerate} 

First of all, Case iv is impossible because we assume $x$ and $y$ are in two distinct infinite clusters. Secondly, if Case i is true, then there exist at least three infinite clusters, which is a contradiction to the assumption that there exists exactly one infinite 0-cluster and one infinite 1-cluster. 

If Case ii is true, then $C_3$ is incident to the infinite 0-cluster including $x$. Since $w_1\in \xi_3$, and $p_{wz}$ does not cross infinite contours except at $w$ and $z$, by \Cref{isc}, we infer that $z\in \xi_3$, and $\xi_3$ is exactly the infinite 0-cluster including $x$. We conclude that $C_1$ is incident to the infinite 0-cluster including $x$ as well, and Claim \ref{cl} is proved.

If Case iii is true, then $C_3$ is incident to the infinite 1-cluster including $y$. Let $t$ be the first vertex in $p_{xz}\cap C_3$, when traveling from $p_{xz}$, starting at $x$, and let $p_{xt}$ be the portion of $p_{xz}$ between $x$ and $t$. We explore the path $p_{xt}$ as we have done for $p_{xz}$. Since the length of $p_{xz}$ is finite, and the number of full edges of $G$ along $p_{xt}$ is less than that of $p_{xz}$ by at least 1, we apply the induction hypothesis with $C_1$ replaced by $C_3$, $C_2$ replaced by $C_1$, $\xi_1$ replaced by $\xi_3$, $p_{xz}$ replaced by $p_{xt}$, and we conclude that there exists an infinite contour adjacent to both the infinite 0-cluster and infinite 1-cluster.
\end{proof}

\begin{lemma}\label{2c2c}Let $C_1$ and $C_2$ be two infinite contours, and let $\xi_0$ and $\xi_1$ be two infinite clusters. The following two cases cannot occur simultaneously.
\begin{itemize}
\item $\xi_0$ is incident to both $C_1$ and $C_2$;
\item $\xi_1$ is incident to both $C_1$ and $C_2$.
\end{itemize}
\end{lemma}
\begin{proof}Same arguments as Lemma 6.3 of \cite{HL16}.
\end{proof}

\begin{lemma}\label{l147}Let $\LL_2$ be the regular tiling of the hyperbolic plane with triangles, such that each vertex has degree $n\geq 7$. Let $\omega\in\{0,1\}^{V(\LL_2)}$ be a site percolation configuration on $\LL_2$.  Assume that there exists an infinite 1-cluster $\eta\subseteq\omega$ with infinitely many ends. Then there exist at least two infinite 0-clusters in $\omega$.
\end{lemma}

\begin{proof}Since $\eta\subseteq \omega$ has infinitely many ends, there exists a finite box $B$ of $\LL_1$, such that $\eta\setminus B$ has at least two distinct infinite components. Let $X$, $Y$ be two distinct infinite components of $\eta\setminus B$. Define the boundary of $X$ (resp.\ Y), $\partial X$ (resp.\ $\partial Y$) to be the set of all edges in $\LL_1$, such that each dual edge has exactly one endpoint in $X$ (resp.\ Y), and one endpoint in $\LL_2\setminus B$ and not in $X$ (resp.\ not in $Y$). Then $\partial X$ and $\partial Y$ are part of contours - each vertex of $\LL_1$ in $\partial X$ and $\partial Y$ has degree 1 or 2, whose degree-1 vertices are along $\partial B$ (here $\partial B$ consists of all the edges of $\LL_1$ on the boundary of the finite box $B$).

Each component of $\partial X$ or $\partial Y$ (here we assume that points on $\partial B$ are not included in $\partial X$ or $\partial Y$) must be one of the following three cases:
\begin{enumerate}
\item a finite component; or
\item a doubly infinite self-avoiding path which does not intersect $\partial B$;
\item a singly-infinite self-avoiding path starting from a vertex along $\partial B$.
\end{enumerate}

Let $B_1\supset B$ be a box of $\LL_2$ containing $B$. Then the embedding of $\partial B_1$ into $\HH^2$ is a simple closed curve consisting of edges in $\LL_2$. Since $X$ and $Y$ are two infinite components of $\eta\setminus B$, we deduce that $\partial B_1\cap X\neq \emptyset$ and $\partial B_1\cap Y\neq \emptyset$. Since $\partial B_1$ is a closed curve, there exist $x_1,x_2\in \partial B_1\cap X$ and $y_1,y_2\in \partial B_1\cap Y$, such that there are segments $p_{x_1y_1}\subset \partial B_1$ joining $x_1$ and $y_1$, and $p_{x_2y_2}\subset \partial B_1$ joining $x_2$ and $y_2$ such that $p_{x_1y_1}$ and $p_{x_2y_2}$ does not intersect each other except possibly at $x_1,x_2,y_1,y_2$; $p_{x_1y_1}$ does not intersect $X\cup Y$ except at $x_1,y_1$; and that  $p_{x_2y_2}$ does not intersect $X\cup Y$ except at $x_2,y_2$.

Then we claim that both $p_{x_1y_1}$ and $p_{x_2y_2}$ cross a component of $\partial X$ of Type III, and a component of $\partial Y$ of Type III. To see why that is true, assume that $p_{x_1y_1}$ crosses only components of $\partial X$ of Type I or II, then we can find a path $q_{x_1y_1}$ consisting of edges in $\LL_2\setminus B$ and joining $x_1$ and $y_1$ such that $q_{x_1y_1}$ does not cross $\partial X$ at all. Then $Y$ and $X$ are the same component of $\eta\setminus B$. The contradiction implies that $p_{x_1y_1}$ must cross a component of $\partial X$ of Type III. Similarly, $p_{x_1y_1}$ must cross a component of $\partial Y$ of Type III;  $p_{x_2y_2}$ must cross a component of $\partial X$ of Type III, and a component of $\partial Y$ of Type III.

Let $\ell_1$ (resp.\ $\ell_2$) be a component of $\partial X$ of Type III crossed by $p_{x_1y_1}$ (resp.\ $p_{x_2y_2}$). Let $V_1$ (resp.\ $V_2$) consist of all the vertices of $\LL_2$ on a triangle face crossed by $\ell_1$ (resp.\ $\ell_2$) but not in $X$. Then $V_1$ is part of an infinite 0-cluster $\xi$, and $V_2$ is part of an infinite 0-cluster $\zeta$; and moreover, $\xi$ and $\zeta$ are distinct since both $p_{x_1y_1}$ and $p_{x_2y_2}$ also cross components of $\partial Y$ of Type III. This completes the proof.
 \end{proof}
 
 \begin{lemma}\label{01c}Let $G=(V,E)$ be a square tiling of the hyperbolic plane satisfying I and II. Let $\omega\in \Omega$ be a constrained percolation configuration on $G$. If there exists a contour in the corresponding contour configuration of $\omega$, then there exist at least one infinite 1-cluster and at least one infinite 0-cluster in $\omega$.
\end{lemma}

\begin{proof}Let $C$ be a contour in the corresponding contour configurations. By \Cref{cit}, $C$ is an infinite tree in which each vertex is incident to 2 or 4 edges. Since $C$ has no cycles, the complement $\HH^2\setminus C$ of $C$ in the hyperbolic plane $\HH^2$ has no bounded components.

We claim that each unbounded component of $\HH^2\setminus C$ contains at least one infinite cluster. Let $\Lambda$ be an unbounded component of $\HH^2\setminus C$. Let $V_{\Lambda,C}\subset V$ consist of all the vertices in $\Lambda$ that are also in a face of $G$ intersecting $C$. Then all the vertices in $V_{\Lambda,C}$ are in the same cluster of $\omega$ and $|V_{\Lambda,C}|=\infty$. Hence there exists an infinite cluster in $\omega$ containing every vertex in $V_{\Lambda,C}$.

Let $e\in C$ be an edge crossing a pair of opposite edges of a black face $b$ of $G$. Let $v_1,v_2,v_3,v_4$ be the 4 vertices of $b$. Assume that $v_1$ and $v_2$ are on one side of $e$ while $v_3$ and $v_4$ are on the other side of $e$. By the arguments above, $v_1$ and $v_2$ are in an infinite cluster $\xi_1$ of $\omega$; similarly, $v_3$ and $v_4$ are in an infinite cluster $\xi_2$ of $\omega$. Moreover, $e\in C$ implies that $v_1$ and $v_3$ have different state; and therefore exactly one of $\xi_1$ and $\xi_2$ is an infinite 0-cluster, and the other is an infinite 1-cluster.
\end{proof}

\begin{lemma}\label{ifc}Let $G=(V,E)$ be a square tiling of the hyperbolic plane satisfying I and II. Let $\omega\in \Omega$ be a constrained percolation configuration on $G$, and let $\phi$ be the corresponding contour configuration of $\omega$. Then each component of $\HH^2\setminus \phi$ contains an infinite cluster in $\omega$.
\end{lemma}

\begin{proof}By \Cref{cit}, $\phi$ is the disjoint union of infinite trees, in which each vertex has degree 2 or 4. Since $\phi$ contains no cycles, each component of $\HH^2\setminus \phi$ is unbounded.

Let $\Lambda$ be an unbounded component of $\HH^2\setminus \phi$. Let $V_{\Lambda,\phi}\subset V$ consist of all the vertices in $\Lambda$ that are also in a face of $G$ intersecting $\phi$. Then all the vertices in $V_{\Lambda,\phi}$ are in the same cluster of $\omega$ and $|V_{\Lambda,\phi}|=\infty$. Hence there exists an infinite cluster in $\omega$ containing every vertex in $V_{\Lambda,\phi}$. Since every vertex in $V_{\Lambda,\phi}$ is in $\Lambda$, and any cluster intersecting $\Lambda$ is completely in $\Lambda$, we conclude that $\Lambda$ contains an infinite cluster of $\omega$.
\end{proof}

\bigskip
\noindent\textbf{Acknowledgements.} ZL thanks Alexander Holroyd for stimulating discussions in the preparation of the paper, and Geoffrey Grimmett and Russ Lyons for comments. ZL is grateful for anonymous reviewers' careful reading of the paper and valuable suggestions to improve the readability.  ZL's research is supported by Simons Foundation grant 351813 and National Science Foundation grant 1608896.

\bibliography{cph}
\bibliographystyle{amsplain}

\end{document}